\documentclass[a4paper,oneside,11pt]{article}
\usepackage{a4wide, dsfont}
\usepackage{graphicx}
\usepackage{array}
\usepackage{booktabs}
\usepackage{tikz}
\usepackage{bbm}   

\usepackage[utf8]{inputenc}
\usepackage{amsmath,amssymb,amsthm}
\usepackage{url}
\usepackage[colorlinks, linkcolor = black, citecolor = black, filecolor = black, urlcolor = blue]{hyperref} 

\usepackage{ifthen}
\usepackage{natbib}
\usepackage{color}
\usepackage{epstopdf}
\usepackage{caption}
\usepackage{subcaption}
\usepackage{arydshln}
\usepackage{graphicx}
\usepackage{longtable}
\usepackage{marvosym}
\usepackage{cases}
\usepackage{lscape}

\newtheorem{satz}{Satz}

\theoremstyle{definition}

\newtheorem{example}{Example}
\newtheorem{remark}[example]{Remark}
\newtheorem{assumption}{Assumption}

\theoremstyle{plain}
\newtheorem{theorem}[satz]{Theorem}

\newtheorem{lemma}{Lemma}

\newcommand*{\defeq}{\mathrel{\vcenter{\baselineskip0.5ex \lineskiplimit0pt
			\hbox{\scriptsize.}\hbox{\scriptsize.}}}%
	=}

\newcommand{\f}[1]{{\mathbf{#1}}}

\newcommand{\subg}{{\operatorname{subG}}}
\newcommand{\sube}{{\operatorname{subE}}}

\newcommand{\Oop}[1]{{\operatorname{\mathcal{O}}\left(#1\right)}}

\newcommand{\one}{\mathbbm{1}}

\newcommand{\normz}[1]{{\left\lVert#1\right\rVert_2}}
\newcommand{\normi}[1]{{\left\lVert#1\right\rVert_\infty}}
\newcommand{\norme}[1]{{\left\lVert#1\right\rVert_1}}
\newcommand{\normeq}[1]{{\left\lVert#1\right\rVert_1^2}}
\newcommand{\normzq}[1]{{\left\lVert#1\right\rVert_2^2}}

\newcommand{\skalar}[2]{{\left\langle#1,#2\right\rangle}}

\newcommand{\normiM}[1]{{\left\lVert#1\right\rVert_{\mathrm{M},\infty}}}
\newcommand{\normzM}[1]{{\left\lVert#1\right\rVert_{\mathrm{M},2}}}

\newcommand{\E}{{\mathbb{E}}}
\newcommand{\R}{{\mathbb{R}}}
\newcommand{\N}{{\mathbb{N}}}
\newcommand{\X}{{\mathbb{X}}}

\newcommand{\Prob}{{\mathbb{P}}}

\newcommand{\sign}{{\operatorname{sign}}}
\newcommand{\diag}{{\operatorname{diag}}}

\newcommand{\supp}{{\operatorname{supp}}}

\newcommand{\init}{{\mathrm{init}}}

\newcommand{\WLH}{{\mathrm{WLPH}}}
\newcommand{\PDW}{{\mathrm{PDW}}}
\newcommand{\ALH}{{\mathrm{ALPH}}}

\newcommand{\Hu}{{\mathrm{H}}}

\newcommand{\loss}{{\mathcal{L}}}

\newcommand{\betaS}{{\beta_{\alpha_n,\mathrm{supp}}^*}}

\newcommand{\Cpm}{m}

\newcommand{\Cm}{{C_{\epsilon,\mathrm{m}}}}

\newcommand{\CXl}{{c_{\f{X},\mathrm{l}}}}

\newcommand{\CXu}{{c_{\f{X},\mathrm{u}}}}

\newcommand{\CXsubs}{{c^2_{\f{X},\mathrm{sub}}}}

\newcommand{\CXsubt}{{c^3_{\f{X},\mathrm{sub}}}}

\newcommand{\CXsubf}{{c^4_{\f{X},\mathrm{sub}}}}

\newcommand{\CXsubu}{{c^u_{\f{X},\mathrm{sub}}}}

\newcommand{\Cbeta}{{C_{\beta}}}
\newcommand{\Cbetas}{{C_{\beta}^2}}

\newcommand{\CXsub}{{c_{\f{X},\mathrm{sub}}}}

\newcommand{\Cinit}{{C_{\mathrm{init}}}}

\newcommand{\Csminf}{{C_{\mathrm{S,\f{X}}}}}

\newcommand{\Csminfs}{{C_{\mathrm{S,\f{X}}}^2}}

\newcommand{\Capprox}{{C_{\mathrm{apx}}}}

\newcommand{\Crscf}{{c^{\mathrm{RSC}}_1}}

\newcommand{\Crscs}{{c^{\mathrm{RSC}}_2}}

\newcommand{\Crsct}{{c^{\mathrm{RSC}}_3}}

\newcommand{\Cgradf}{{c^{\mathrm{Grad}}_1}}

\newcommand{\Cgrads}{{c^{\mathrm{Grad}}_2}}

\newcommand{\Cgradt}{{c^{\mathrm{Grad}}_3}}

\newcommand{\Cgradv}{{c^{\mathrm{Grad}}_4}}

\newcommand{\cPo}{{c^{\mathrm{P}}_1}}

\newcommand{\cPt}{{c^{\mathrm{P}}_2}}

\newcommand{\calpha}{{c_\alpha}}

\newcommand{\CQS}{{C_{\mathrm{Q,S}}}}

\newcommand{\CQograd}{{C_{\mathrm{Q}, \loss}}}

\newcommand{\CQgrad}{{C_{\mathrm{Q}, \loss} \, \Cgrads}}

\newcommand{\Clambda}{{C_{\lambda}}}

\newcommand{\ClambdaL}{{C_{\lambda,\mathrm{L}}}}

\newcommand{\Clambdainit}{{C_{\lambda,\mathrm{init}}}}

\newcommand{\Choeld}{{q}}

\newcommand{\Calpha}{{C_{\alpha}}}

\begin{document}
	
	%
	%
	{\title{Support estimation and sign recovery in high-dimensional heteroscedastic mean regression}
		\author{Philipp Hermann and Hajo Holzmann\footnote{Corresponding author. Prof.~Dr.~Hajo Holzmann,  Fachbereich Mathematik und Informatik, Philipps-Universität Marburg, Hans-Meerweinstr.~6, 35043 Marburg, Germany} \\
			\small{Department of Mathematics and Computer Science}  \\
			\small{Philipps-Universität Marburg} \\
			\small{\{herm, holzmann\}@mathematik.uni-marburg.de}}
		\maketitle
	}
	\begin{abstract}
		A current strand of research in high-dimensional statistics deals with robustifying the methodology with respect to deviations from the pervasive light-tail assumptions. In this paper we consider a linear mean regression model with random design and potentially heteroscedastic, heavy-tailed errors, and investigate support estimation and sign recovery. We use a strictly convex, smooth variant of the Huber loss function with tuning parameter depending on the parameters of the problem, as well as the adaptive LASSO penalty for computational efficiency. For the resulting estimator we show sign-consistency and optimal rates of convergence in the $\ell_\infty$ norm as in the homoscedastic, light-tailed setting. In our simulations, we also connect to the recent literature on variable selection with the thresholded LASSO and false discovery rate control using knockoffs, and indicate the relevance of the Donoho - Tanner transition curve for variable selection.  The simulations illustrate the favorable numerical performance of the proposed methodology.

	\end{abstract}
	\noindent {\small {\itshape Keywords.}\quad convergence rates,  Huber loss function, knockoff filter, robust high-dimensional regression, sign recovery, support estimation, variable selection 
		
		\textit{Running title:} Support estimation in heteroscedastic mean regression
		
		%
		%
		%
		\section{Introduction}
		
		Data sets with a large number of features, often of the same order as or even of larger order than the number of observational repetitions, have become ever more common in applications such as microarray data analysis, functional magnetic resonance imaging or consumer data analysis. In consequence, much methodological research has been done in the area of high-dimensional statistics, and the field has developed rapidly. State of the art expositions are provided in \citet{Wainwright2019}, \citet{hastie2015statistical} and \citet{giraud2014introduction}.  
		
		A current strand of research deals with robustifying the available methodology with respect to deviations from light-tail assumptions in particular on the errors, and sometimes also on the predictors. 
		One common approach is to replace the squared loss function by some other, fixed, robust loss function such as the check function from quantile regression and in particular absolute deviation for the median \citep{belloni2011l1, Fan2014}. However, doing so generally changes the target parameter away from the mean, particularly in the random design regression models with potentially heteroscedastic, asymmetric errors that we shall focus on.  
		
		More specifically, consider the random design linear regression model
		\begin{align} 
			Y_j = \f{X_j^\top} \beta^* + \varepsilon_j, \qquad j=1, \ldots, n, \label{generalregressionmodel}
		\end{align}
		in which the real-valued responses $Y_j$ and the $p$-variate covariates $\f{X_j} \in \R^{p}$ are observed, and $\beta^* \in \R^p$ is the unknown parameter vector.  
		%
		We allow for a random design with heteroscedastic errors, and assume that the pairs $(\f{X_1},\varepsilon_1) ,\dotsc,(\f{X_n}, \varepsilon_n)$ are independent and identically distributed (however, $\f{X_j}$ and $\varepsilon_j$ are allowed to be dependent) with $\E\big[\varepsilon_j \mid \f{X_j}\big]=0$, so that $\E\big[Y_j \,| \, \f{X_j}\big]  = \f{X_j^\top} \beta^*$ is the identified conditional mean. 
		We focus on the high-dimensional case and consider heavy tailed, non sub-Gaussian errors $\varepsilon_j$ but restrict ourselves to light-tailed regressors $\f{X_j}$. 
		
		Results in \citet{lederer2020estimating} imply that for uniformly bounded covariates, if the errors have slightly more than a finite fourth moment the ordinary least squares LASSO estimator retains the rates of convergence known from the sub-Gaussian case. 
		For high-dimensional mean regression under still weaker assumptions, \citet{Fan2017} and \citet{Sun2019} considered LASSO estimates with the Huber loss function \citep{huber1964robust} with parameter $\alpha > 0 $ defined by
		\begin{equation}\label{eq:huberloss}
			\tilde l_\alpha(x) = 
			(2\alpha^{-1} |x| - \alpha^{-2})\, \one{\{|x|>\alpha^{-1}\}} + \,
			x^2 \, \one{\{|x|\leq\alpha^{-1}\}} \,.
		\end{equation}
		To deal with the resulting bias, they let the parameter $\alpha$ depend in a suitable way on sample size and dimension.  \citet{Sun2019} show that if the errors have a finite second moment and the covariates are sub-Gaussian, then for $\alpha \simeq (\log (p)/n)^{\frac{1}{2}}$ the  estimator has the same rates of convergence in $\ell_1$ and $\ell_2$ norms as in the light-tailed case. 
		
		We study support estimation, sign recovery and rates in $\ell_\infty$  norm in this framework. Previously, in robust high-dimensional regression \citet{Fan2014} addressed support estimation for quantile regression. \citet{Loh2017statistical} considered robust mean regression for homoscedastic models with independent covariates $\f{X_j}$ and errors $ \varepsilon_j$, which excludes the heteroscedastic setting that we focus on.   
		In a seminal paper, \citet{Wainwright2009} studied support estimation with the LASSO and introduced the primal-dual witness proof method, see also \citet{zhao2006model}.  
		A line of investigation to which we contribute here aims at providing results on support estimation and sign recovery for general design matrices, and in particular on getting rid of the irrepresentability condition required by the LASSO. To this end, a prominent approach is the use of the adaptive LASSO \citep{zou2006adaptive}. A high-dimensional analysis of the adaptive LASSO in homoscedastic regression models with light-tailed errors is provided in \citet{Zhou2009, van2011adaptive}.
		
		In this paper, in the random design, heteroscedastic regression model \eqref{generalregressionmodel} we show sign-consistency and optimal rates of convergence in the $\ell_\infty$ norm for a computationally feasible estimator with the adaptive LASSO penalty and the following variant of the Huber loss function, sometimes called pseudo Huber loss, 
		\begin{align} \label{lalpha}
			l_\alpha(x) = 2 \alpha^{-2} \Big( \sqrt{1 + \alpha^2 x^2} -1 \Big),
		\end{align}
		as proposed by \citet{charbonnier}. 
		Compared to the Huber loss, $l_\alpha$ is smooth and strictly convex, which facilitates our theoretical analysis. While not exactly equal to $x^2$ it approximates the squared loss sufficiently closely for $|x| \leq \alpha^{-1}$ so that mean estimation remains possible, while for large values of $|x|$ it keeps the robustness of the Huber loss. Our numerical experiments show little difference between the estimators based on the classical and the pseudo Huber loss functions.  For the errors we require only slightly more than second moments. The major novel issue in the proofs is that the supports of the target parameter in the linear mean regression model and its robustified version may differ even for small values of the tuning parameter of the Huber loss function.
		
		In our numerical experiments we connect to the recent literature on variable selection with the knockoff methodology and false discovery rate (FDR) control \citep{barber2015controlling, weinstein2020}. Robust versions of both the LASSO as well as the adaptive LASSO are investigated in combination with knockoffs. We consider simulation settings which are favorable as well as hard for sign recovery as theoretically determined by the Donoho-Tanner threshold \citep{donoho2009}. Overall, the robust version of the adaptive LASSO together with a knockoff threshold shows convincing numerical performance.

		The paper is organized as follows. In Section \ref{sec:estimator} we introduce the estimator and set up some notation. Section \ref{sec:mainresults} contains the main results of the paper. After reporting on numerical experiments in Section \ref{sec:sims}, we present a real-data illustration in Section \ref{sec:realdata}. Section \ref{sec:conclusions} concludes, while main steps of the proofs are presented in Section \ref{sec:proofmainsteps}. 
		Additional technical proofs as well as additional simulations are contained in the supplementary appendix.

		\section{The adaptive LASSO with pseudo Huber loss function}\label{sec:estimator}

		We consider an estimator based on minimizing the pseudo Huber loss function with a weighted LASSO penalty given by 
		
		\begin{align}
			\widehat{\beta}_n^{\,\WLH} \in \underset{\beta \in \R^p}{\arg\min} ~ \bigg( \loss_{n,\alpha_n}^{\,\Hu}\big(\beta\big) +  \lambda_n  \sum_{k=1}^p w_k\,|\beta_k|\bigg) \,  \label{defweightedlassohuber1}
		\end{align}

		with regularization parameter $\lambda_n$, robustification parameter $\alpha_n > 0$ and weights $w_k > 0$ for $k \in \{1,\dotsc,p\}$.
		In \eqref{defweightedlassohuber1}, the empirical loss function $\loss_{n,\alpha}^{\,\Hu}$ associated with the pseudo Huber loss is defined by
		\begin{align}\label{eq:emploss}
			\loss_{n,\alpha}^{\,\Hu}\big(\beta\big) \defeq \frac{1}{n} \sum_{i=1}^{n} l_{\alpha} \big( Y_i - \f{X_i^\top} \beta \big) \,,
		\end{align}
		and $l_\alpha$ is as in \eqref{lalpha}. 
		We shall call  $\widehat{\beta}_n^{\,\WLH}$ the weighted LASSO pseudo Huber estimator.
		It estimates the parameter 
		\begin{align}\label{eq:betaalpha}
			\beta_{\alpha_n}^* \defeq \underset{\beta \in \R^p}{\arg\min} ~ \E\Big[l_{\alpha_n} \big( Y_1 - \f{X_1^\top} \beta \big)\Big]\,.  
		\end{align}
		%
		%
		%
		%
		%
		We stress that the assumptions in \citet{Loh2017statistical}, in particular the assumed independence of $\f{X_1}$ and $\varepsilon_1$ together with a centering of $\f{X_1}$, directly implies that $\beta_{\alpha_n}^*$ is equal to the target parameter $\beta^*$. However, in our heteroscedastic setting in which $\f{X_1}$ and $\varepsilon_1$  are dependent, this is no longer the case, so that it will be necessary to let the regularization parameter $\alpha_n \downarrow 0$.

		
		Later on we assume $\E[\f{X_1} \f{X_1^\top}]$ to be positive definite, hence $\beta_{\alpha_n}^*$ is unique by the strict convexity of $l_{\alpha_n}$. For a suitable initial estimator $\widehat{\beta}_{n}^{\,\init} = (\widehat{\beta}_{n,1}^{\,\init}, \ldots, \widehat{\beta}_{n,p}^{\,\init})^\top$ of $\beta^*$ such as the LASSO Huber estimator from \citet{Fan2017}, choosing the (random) weights  
		\begin{align}\label{eq:weightsadaplassoinitest}
			w_k = \max \big\{ 1/\big| \widehat{\beta}_{n,k}^{\,\init} \big| , 1 \big\}, \qquad k=1, \ldots, p,
		\end{align}
		leads to the adaptive LASSO pseudo Huber estimator $\widehat{\beta}_n^{\,\ALH}$ which we shall focus on. 
		Here, if $\big|\widehat{\beta}_{n,k}^{\,\init}\big| = 0 $ so that formally $w_k = \infty$, we require that $\beta_k=0$.
		%
		%
		
		\smallskip
		
		We shall investigate the sign-consistency as well as the rate of convergence in the $\ell_\infty$ distance of $\widehat{\beta}_n^{\,\ALH}$.
		To this end, let us set up some notation used in the following. Denote the support of the coefficient vector $\beta^*$ and its regularized version $\beta_{\alpha_n}^*$ in \eqref{eq:betaalpha} by
		\begin{alignat*}{2}
			S &\defeq \supp\big(\beta^*\big) = \Big\{ k \in \{1,\dotsc,p\} \mid \beta_{k}^* \neq 0 \Big\} \,, \qquad \qquad &&\ \ \  s  \defeq | S |, 	\\
			S_{\alpha_n} &\defeq \supp\big(\beta_{\alpha_n}^*\big) = \Big\{ k \in \{1,\dotsc,p\} \mid \beta_{\alpha_n,k}^* \neq 0 \Big\} \,, && s_{\alpha_n}  \defeq \big| S_{\alpha_n} \big|, 
		\end{alignat*}
		where $| S |$ is the cardinality of $S$. A major additional issue in our investigation will be that the support $S$ of $\beta^*$, the object of interest, differs from the support $S_{\alpha}$ of $\beta^*_\alpha$, the parameter which is actually estimated. Indeed, even if $\beta^*$ is sparse in the sense that $S$ is of small cardinality, this need not be the case for $\beta^*_\alpha$. We illustrate this numerically at the beginning of the simulations in Section \ref{sec:sims}.
		However, our analysis will show that the adaptive LASSO penalty reliably sets the small superfluous entries of $\beta^*_\alpha$ to zero.  
		
		Various results on support recovery and sign consistency depend, in terms of so-called beta-min conditions, on the smallest absolute value of the entries of $\beta^*$ on its support $S$, which we denote by
		\begin{align}
			\beta_{\min}^* &\defeq \min_{k \in S} \big|\beta_{k}^*\big| \,.
		\end{align}
		%
		
		We shall further employ the following notations. For a subset $A \subseteq \{1, \ldots, p\}$, 
		$\beta_A$ denotes the vector $(\beta_A)_i = \beta_i\, \one \{i \in A\}$, $i \in \{1, \ldots, p\}$, and sometimes also the vector $(\beta_i)_{i \in A} \in \R^{|A|}$, where $|A|$ is the cardinality of $A$. $A^c = \{1, \ldots, p\} \setminus A$ is the complement of $A$. If $Q \in \R^{p \times p}$ is a $p \times p$ - matrix, and  $B \subseteq \{1, \ldots, p\}$ is a further subset, $Q_{AB} \in \R^{|A| \times |B|}$ has entries according to row indices in $A$ and column indices in $B$. The $\ell_1$, $\ell_2$ and $\ell_\infty$ norms of a vector $x$ are denoted by $\norme{x}$, $\normz{x}$, $\normi{x}$, and the corresponding matrix operator norms by $\normiM{M}$ for the $\ell_\infty$ norm, and $\normzM{M}$ for the $\ell_2$ norm, also called spectral norm. We have for $M \in \R^{p \times q}$ that
		\[ \normiM{M} = \max_{x \in \R^q \, , \|x\|_\infty \leq 1} \| M x \|_\infty = \max_{1 \leq i \leq p} \sum_{j=1}^q |M_{i,j}|.\]
		The symbol $a \lesssim b$, where $a$ and $b$ will depend on $n$, $p$ and $|S|$, means that $a$ is smaller than $b$ up to constants not depending on $n$, $p$ and $|S|$, and $a \simeq b$ means that $a$ and $b$ are of the same order, that is $a \lesssim b$ as well as $b \lesssim a$.
		Finally, $\sign(t)$ denotes the sign of a number $t$, that is, $\sign(t) = \one\{t>0\}- \one\{t < 0\}$, and $\sign$ is applied coordinate wise to a vector.   
		
		\section{Sign-consistency and rate of convergence in $\ell_\infty$ norm}\label{sec:mainresults}
		
		
		In this section we state our main results on sign-consistency and convergence rates in the $l_\infty$ norm of the adaptive LASSO pseudo Huber estimator in our setting with heteroscedastic, heavy-tailed and potentially asymmetric errors, focusing on the orders and discarding the constants. We give two versions of the result, the first assuming bounds on the $\ell_\infty$ - error of the initial estimator, the second using only bounds on the $\ell_1$ - and $\ell_2$ - norms.

		To derive our results we adopt the following assumptions from \citet{Fan2017}. 
		\begin{assumption} \label{assfan2017}\hfill
			\begin{itemize}
				\item[(i)] For $\Cpm  = 2$ or $\Cpm  = 3$ and $\Choeld >1$ we have that $
				\E \big[\E\big[|\varepsilon_1|^{\Cpm} \big| \f{X_1}\big]^\Choeld \big] \leq \Cm < \infty 
				$, where  $\Cm>0$ is a positive constant.
				\item[(ii)] For constants $ 0 < \CXl < \CXu$ we have that $0 < \CXl \leq \lambda_{\min} \big( \E\big[\f{X_1}\f{X_1^\top}\big] \big) \leq \lambda_{\max} \big( \E\big[\f{X_1}\f{X_1^\top}\big] \big) \leq \CXu < \infty
				$, where $\lambda_{\min}(A)$ and $\lambda_{\max}(A)$ denote the minimal and maximal eigenvalues of a symmetric matrix $A$.
				\item[(iii)] For any $v \in \R^p \setminus \{\f{0}_p\}$ the variable $v^\top \f{X_1}$ is sub-Gaussian with variance proxy at most $\CXsubs \normzq{v}$, $\CXsubs >0 $, that is $\Prob(|v^\top \f{X_1} | \geq t) \leq 2\, \exp\big(- t^2 /(2\, \CXsubs \normzq{v}) \big)$ for all $t \geq 0$. 
			\end{itemize}

		\end{assumption}
		\begin{remark}[Order of approximation and Assumption \ref{assfan2017}]
			The assumptions are essentially those from \citet{Fan2017}. In (i), we use a weaker moment assumption, but slightly more than the finite second moment as required in \citet{Sun2019}. Assumption \ref{assfan2017} leads to a bound on the approximation error of order
			$$	\normz{\beta_{\alpha_n}^* - \beta^*} \leq  \Capprox \, \alpha_n^{\Cpm-1} \,, $$
			for a constant $\Capprox>0$ specified in Lemma \ref{approximation:error:pseudo:huber} in Section \ref{sec:technicalpreps}. 
			As in \citet[Section 5]{Fan2017} for the loss function from \citet{catoni2012challenging}, in (i) we can only make use of moments up to order $3$ when bounding the approximation error. Assumption \ref{assfan2017}, (i) with the weaker $m=2$ leads to the order $\alpha$, which suffices so that the contribution from the approximation error is of the same order as the rate of estimation in supremum norm. Under the stronger assumption with $m=3$, the approximation order becomes $\alpha^2$ which makes this contribution negligible compared to the overall rate.  
			%
		\end{remark}

		\smallskip
		
		For the first result we assume that the initial estimator $\widehat{\beta}_n^{\,\init}$ in the adaptive LASSO pseudo Huber estimator achieves the following rate in the $\ell_\infty$ norm 
		\begin{align}\label{eq:conrateinitiallinf}
			\normi{\widehat{\beta}_n^{\,\init} - \beta^*} \leq \Cinit   \,\bigg(\frac{\log(p)}{n}\bigg)^\frac{1}{2}	,
		\end{align}
		where $\Cinit \geq 1$.
		The next theorem follows from Lemma \ref{signconsadaptivehuber} together with Lemmas \ref{lemmaassstrictdual(ii)} and \ref{rateofrnS} in Section \ref{sec:proofmainsteps}.
		
		\begin{theorem}[Sign-consistency and rate in the $\ell_\infty$ norm under initial $\ell_\infty$ - bound] \label{signconsadaptivehuberwithinitialellinf}
			In model \eqref{generalregressionmodel} under Assumption \ref{assfan2017}, consider the adaptive LASSO estimator $\widehat{\beta}_n^{\,\ALH}$ with initial estimator $\widehat{\beta}_n^{\,\init}$ assumed to satisfy \eqref{eq:conrateinitiallinf}.
			Further, suppose that 
			\begin{align}
				\normiM{\Big(\E\big[\f{X_1}\f{X_1^\top}\big]_{SS}\Big)^{-1}} \leq \Csminf \,, \label{lemmaassstrictdual(ii)(1)}
			\end{align}
			where $\Csminf>0$ is a positive constant, is also satisfied. Assume that the robustification parameter $\alpha_n$ is chosen of the order 
			\begin{align} \label{rangealphaalhwithinitial}
				\alpha_n \simeq \bigg( \frac{\log(p)}{n}\bigg)^{\frac{1}{2}} ,
			\end{align} 
			and that the regularization parameter $\lambda_n$ is chosen of order 
			$ \lambda_n \simeq \, (\log(p))/n$. 
			If 
			$n \gtrsim  s \log(p)$  
			and if $\beta^*$ satisfies a beta-min condition of order 
			$$\beta_{\min}^* \gtrsim \bigg( \frac{\log(p)}{n}\bigg)^{\frac{1}{2}},$$ 
			then with probability at least 
			\begin{align}\label{eq:theprob}
				1-c_1 \exp(-c_2 n) - \frac{c_3}{p^2}\,,
			\end{align}
			
			%
			%
			where $c_1,c_2,c_3>0$ are suitable constants, the adaptive LASSO pseudo Huber estimator $\widehat{\beta}_n^{\,\ALH}$ as a solution to \eqref{defweightedlassohuber1} with weights \eqref{eq:weightsadaplassoinitest} satisfies
			%
			\begin{align}
				\sign\big(\widehat{\beta}_{n}^{\,\ALH}\big) =  \sign\big(\beta^*\big) ~~~~~~\mathrm{and} ~~~~~~ \normi{\widehat{\beta}_{n}^{\,\ALH} - \beta^*} \lesssim \,\bigg(\frac{\log(p)}{n}\bigg)^\frac{1}{2}	. \label{alhwithinitialaussage} 
			\end{align}
			%
			
		\end{theorem}  
		
		While for the classical LASSO there are results under which the sup-norm rate of \eqref{eq:conrateinitiallinf} is guaranteed \citep{lounici2008, Wainwright2019}, these do not apply to the setting with heteroscedasticity and heavy-tailed errors that we consider here. 
		Therefore, we also present a result which makes use of rates in $\ell_2$ and $\ell_1$ norms for the initial estimator. 
		%
		\begin{align}\label{eq:conrateinitial}
			\normz{\widehat{\beta}_n^{\,\init} - \beta^*} \leq \Cinit \, \sqrt{s} \,\bigg(\frac{\log(p)}{n}\bigg)^\frac{1}{2}\, , \quad  \norme{\widehat{\beta}_n^{\,\init} - \beta^*} \leq \, \Cinit \,  s\, \bigg(\frac{\log(p)}{n}\bigg)^\frac{1}{2}. \end{align}
		Indeed, under Assumption \ref{assfan2017}
		the original LASSO Huber estimator given as a solution of \eqref{defweightedlassohuber1} with Huber loss $\tilde l_\alpha$, weights $w_k=1$, $k=1, \ldots, p$, 
		satisfies \eqref{eq:conrateinitial} for $n \gtrsim s \, \log(p)$ under the scaling $\alpha_n \simeq \big(\frac{\log(p)}{n}\big)^\frac{1}{2}$ of the robustification parameter and the choice of the regularization parameter as in \eqref{eq:conrateinitial}, with probability at least $1 - 3/p$, see \citet[Theorem 8]{Sun2019}. From our results in Section \ref{sec:technicalpreps} it follows that the same is true when using the pseudo Huber loss function $l_\alpha$ instead.

		The next theorem follows from Lemma \ref{signconsadaptivehuberwithinitial1} combined with Lemmas \ref{lemmaassstrictdual(ii)} and \ref{lemmaassstrictdual(i)} in Section \ref{sec:proofmainsteps}.

		\begin{theorem}[Sign-consistency and rate in the $\ell_\infty$ norm under inital $\ell_2$ - bound] \label{signconsadaptivehuberwithinitial}
			In model \eqref{generalregressionmodel} under Assumption \ref{assfan2017}, consider the adaptive LASSO estimator $\widehat{\beta}_n^{\,\ALH}$ with initial estimator $\widehat{\beta}_n^{\,\init}$ assumed to satisfy \eqref{eq:conrateinitial}.
			Further, suppose that \eqref{lemmaassstrictdual(ii)(1)} holds, that 
			the robustification parameter $\alpha_n$ is chosen of the order in 
			\eqref{rangealphaalhwithinitial}
			and that the regularization parameter $\lambda_n$ is chosen of order 
			\begin{align}
				\lambda_n \simeq \,\big|\overline{S}\big|^{1/2} \, \frac{ \log(p)}{ n}, \qquad \text{where} \quad \overline{S} = \Big\{ k \in \{1,\dotsc,p\} \,\Big| \,\big|\widehat{\beta}_{n,k}^{\,\init}\big| > \Cinit  \,\bigg(\frac{\log(p)}{n}\bigg)^\frac{1}{2} \Big\}. \label{lambdanalhwithinitial}
			\end{align}
			If 
			$n \gtrsim  s^2 \log(p)$  
			and if $\beta^*$ satisfies a beta-min condition of order 
			$$\beta_{\min}^* \gtrsim s\, \bigg(\frac{\log(p)}{n}\bigg)^\frac{1}{2},$$ 
			%
			%
			then with probability of at least  \eqref{eq:theprob} 
			the adaptive LASSO pseudo Huber estimator $\widehat{\beta}_n^{\,\ALH}$  again satisfies \eqref{alhwithinitialaussage}. 
			%
			
			If we drop assumption \eqref{lemmaassstrictdual(ii)(1)} but instead have $s \leq \log(p)$, 
			then we retain the sign-consistency in \eqref{alhwithinitialaussage} but only obtain a $\ell_\infty$-rate of order 
			\[ \normi{\widehat{\beta}_{n}^{\,\ALH} - \beta^*} \lesssim  \sqrt{s}\, \bigg(\frac{\log(p)}{n}\bigg)^\frac{1}{2}.\]
		\end{theorem}

		\begin{remark}
			Let us compare the results in Theorems \ref{signconsadaptivehuberwithinitialellinf} and \ref{signconsadaptivehuberwithinitial}.
			The assumption \eqref{lemmaassstrictdual(ii)(1)} in both theorems arises naturally in the primal-dual witness method, see e.g.~\citet{Loh2017}, assumption (21) to Corollary 1, and leads to the fast rate in \eqref{alhwithinitialaussage} under the supremum norm. While this rate and support recovery is achieved in both results,  Theorem  \ref{signconsadaptivehuberwithinitial} requires a stronger beta - min condition and a different, larger choice of the regularization parameter in the adaptive LASSO procedure in order to deal with the slower $\ell_2$  - rate of the initial estimator. Note that the bound \eqref{alhwithinitialaussage} together with the sign consistency implies that 
			\[ \normz{\widehat{\beta}_{n}^{\,\ALH} - \beta^*} \lesssim \sqrt{s}\, \bigg(\frac{\log(p)}{n}\bigg)^\frac{1}{2} \quad \text{and} \quad \norme{\widehat{\beta}_{n}^{\,\ALH} - \beta^*} \lesssim s\, \bigg(\frac{\log(p)}{n}\bigg)^\frac{1}{2},\]
			as for the ordinary LASSO Huber estimator.
			
			The final bound in Theorem \ref{signconsadaptivehuberwithinitial} without condition \eqref{lemmaassstrictdual(ii)(1)}  is as in \citet{Zhou2009}. Somewhat unfortunately, this result requires that $s \leq \log(p)$ and hence is only useful in high dimensions, however, at this stage we were not able to get rid of this restriction.
			
			%
			%
			
			%
		\end{remark}

		\section{Simulations}\label{sec:sims}

		In this section we numerically compare the performance of the classical LASSO (L) and adaptive LASSO (AL) with LASSO as first stage method, both with quadratic loss function, with methods based on the Huber loss $\tilde l_\alpha$ in \eqref{eq:huberloss} and the Pseudo Huber loss function $l_\alpha$ in \eqref{lalpha}. Here, we denote the LASSO Huber estimator by LH, the adaptive LASSO Huber estimator with LH as first stage estimator as ALH, both using the original Huber loss $\tilde l_\alpha$, and the  LASSO Pseudo Huber estimator by LPH as well as the adaptive LASSO Pseudo Huber estimator as ALPH. For this last estimator, we consider both LH and LPH as first stage estimators, resulting in the two step procedures ALPH(LH) and ALPH(LPH). 

		First, effects of light or heavy tails and of heteroscedasticity in a setting which is similar to that in \citet{Fan2017} and favorable for sign recovery are investigated in Section \ref{sec:simsasinfan}. 
		
		Second, both for the LASSO as well as also for the adaptive LASSO, and both for ordinary least squares and Huber losses, in Section \ref{sec:simsnew} we additionally employ the knockoff filter as originally introduced in \citet{barber2015controlling}. The adaptive LASSO seems not to have been previously used together with the knockoff filter. Our setting is similar to that in \citet{weinstein2020}. We use two constellations of $n$, $p$ and $s$, one slightly below, the other above the Donoho - Tanner transition curve, 
		resulting in simulation settings which are favorable respectively hard for sign recovery as theoretically determined by the Donoho-Tanner threshold \citep{donoho2009}. 
		See also  \citet{weinstein2020} for a discussion. 
		While using the knockoff filter results in much improved variable selection properties for the methods involving the ordinary LASSO, the best overall performance in our simulations is achieved by the adaptive LASSO with knockoff filter and Huber loss function.

		To compute the estimators in the simulation we use the functions of the packages \texttt{glmnet} (classic LASSO and adaptive LASSO) and \texttt{hqreg} (LASSO with Huber loss and adaptive LASSO with Huber loss). They have a factor of $1/2$ in the quadratic loss. Further, the definition of the Huber loss includes an additional scaling of $\alpha/2$ in the package \texttt{hqreg}, see \citet{Yi2017}. As a consequence, for the Huber loss the regularization parameter $\lambda$ of the (adaptive) LASSO includes this scaling factor of $\alpha$ as well, therefore we actually displayed $\lambda / \alpha$ for the Huber loss, which needs to be compared to $\lambda$ for the ordinary LASSO and the pseudo Huber loss. 
		To compute the estimator for the pseudo Huber loss, we modified the functions of the package \texttt{hqreg} which were provided on GitHub by \citet{Yi2017}. This package uses a semismooth Newton coordinate descent algorithm, in contrast to the classical coordinate descent algorithm  in \texttt{glmnet} or the iterative local adaptive majorize-minimization (I-LAMM) algorithm in \citet{fan2018}.

		The parameters $\alpha$ and $\lambda$ of the estimators are chosen such that the $\ell_2$ distance of the respective estimation error is minimal. For this purpose we use $100$ independent repetitions, where the errors have a specified distribution, and run through a one- or two-dimensional grid for the parameters in each set. 
		In the adaptive versions of the estimators the parameters of the initial estimators are fixed (and equal to the optimal choices for the LASSO), so that we do not require a four-dimensional grid search for the adaptive LASSO. The resulting choices of the robustification parameter $\alpha$ and the regularization parameter $\lambda$ are displayed in the subsequent tables. Somewhat surprisingly, the tuning parameter for the adaptive version of the estimators differs quite strongly between the LASSO and the estimators based on (pseudo) Huber loss, even for homoscedastic, normally distributed errors.

		
		Next we use these values of the parameters $\lambda$ and $\alpha$ in a Monte-Carlo-Simulation with $1000$ iterations. We compute the false positive rate FPR, the average percentage of false positives among true negatives
		as well as the false negative rate FNR, the proportion of false negatives among true positives 
		together with the $\ell_2$ and $\ell_\infty$ distance between estimates and parameters.
		
		\smallskip

		Let us conclude this introduction to the simulations with a numerical illustration of the effect of the regularization parameter $\alpha$ on the target parameter $\beta_\alpha^*$ in \eqref{eq:betaalpha}. Since an exact numerical evaluation is hard, we use estimates for $\beta^* = (3,3,0,0)$ with the small values $p=4$ and $s=2$ in a large sample of size $n=1000000$ without LASSO penalization. 
		Covariates are independent and standard normally distributed, and we use homoscedastic skewed t-distributed errors, $\widetilde{\varepsilon}_i = Q_i - \E[Q_i]$ with $Q_i \sim \text{St}(0,1,20,3)$. Note that standard errors are of order $0.001$. Mean regression recovers the target parameter $\beta^*$ as
		$$
		\widehat{\beta} = \big(3.005334, 3.002371, -0.000047, -0.000319 \big)^\top
		$$
		while for $\beta^*_\alpha$ with $\alpha = 10000$ we obtain the estimate
		$$
		\widehat{\beta}_\alpha = \big(3.000901, 2.999503, 0.002661, 0.002733 \big)^\top \ ,
		$$
		indicating that the last two entries of $\beta^*_\alpha$ are small but not identically zero. 
		Note that for large $\alpha$, $\widehat{\beta}_\alpha$ approximates the median regression estimate 
		$
		\widehat{\beta}_{\text{med}} = \big(3.002109, 3.000245, 0.003135, 0.003515 \big) \ .
		$
		
		\subsection{Effects of tails and heteroscedasticity}\label{sec:simsasinfan}
		
		We consider the high-dimensional linear regression model \eqref{generalregressionmodel} with  $p=400$ normally distributed covariates $\f{X_1},\dotsc,\f{X_n} \sim \mathcal{N}_{p} \big(\f{0}_{p}, \mathrm{I}_{p} \big)$ and $n=200$ observations, and a parameter vector given by
		\begin{align*}
			\beta^* = \big( 3,\dotsc,3,0,\dotsc,0 \big)^\top
		\end{align*}
		with $S = \supp \big( \beta^* \big) = \{1,\dotsc,20\}$ and $s=|S| = 20$. 
		This scenario is favorable for variable selection, since computing the parameters for the Donoho - Tanner transition curve we have that $\delta = n/p = 0.5$, $\rho_{DT}(0.5) \approx 0.35$, and $\rho = s/n =
		0.1 < 0.35$ is below this curve \citep[Figure 1]{donoho2009}. 
		In the following we discuss different types of errors (light/ heavy tails, symmetric/ asymmetric, homo-/ heteroscedastic). In the homoscedastic case we assume $\varepsilon_i = \widetilde{\varepsilon}_i$ with $\widetilde{\varepsilon}_1,\dotsc,\widetilde{\varepsilon}_n$ independent and identically distributed with $\E[\widetilde{\varepsilon}_1] = 0$ and independent of the covariates $\f{X_1},\dotsc,\f{X_n}$, while in the heteroscedastic case the errors are 
		\begin{align*}
			\varepsilon_i = \frac{1}{\sqrt{3}\,\normzq{\beta^*}} \, \big(\f{X_i}^\top \beta^* \big)^2 \, \widetilde{\varepsilon}_i \,.
		\end{align*}
		Evidently, $(\f{X_1}, \varepsilon_1), \ldots, (\f{X_n}, \varepsilon_n)$ are independent and identically distributed and $\E[\varepsilon_i \, |\, \f{X_i}  ]=0$. Furthermore, the factor $1/ \big(\sqrt{3}\,\normzq{\beta^*} \big)$ implies 
		\begin{align*}
			\E\big[\varepsilon_1^2\big] = \frac{1}{3\,\lVert\beta^*\rVert_2^4 } \, \E\Big[\big(\f{X_1}^\top \beta^* \big)^4\Big] \, \E\big[\widetilde{\varepsilon}_1^{\ 2}\big] = \frac{1}{3\,\lVert\beta^*\rVert_2^4 } \, 3\,\lVert\beta^*\rVert_2^4 ~ \E\big[\widetilde{\varepsilon}_1^{\ 2}\big] =  \E\big[\widetilde{\varepsilon}_1^{\ 2}\big] 
		\end{align*}
		since $\f{X_1}^\top \beta^* \sim \mathcal{N}\big(0,\normzq{\beta^*}\big)$. Hence the homo - and heteroscedastic errors are centered and have the same variance in our simulations. For the  $\widetilde{\varepsilon}_1,\dotsc,\widetilde{\varepsilon}_n$ we consider the following two scenarios. 
		
		\begin{itemize}
			\item[(a)] \textbf{Symmetric errors with light tails.}\\
			$\widetilde{\varepsilon}_i \sim \mathcal{N}(0,4)$ with variance equal to $4$.

			%
			%
			%
			%

			\item[(b)] \textbf{Asymmetric errors with heavy tails.}\\
			$\widetilde{\varepsilon}_i = Q_i -  \E[Q_i]$ with $Q_i \sim \text{St}(0,1,0.6,3)$ skew t-distributed with location parameter $0$, scale parameter $1$, skew parameter $0.6$ and $3$ degrees of freedom. An exact definition can be found in \citet{Azzalini2003} and it is $\E[Q_i] = \big(0.6/\sqrt{1.36}\big)\,\sqrt{3/\pi} \,/ \,\Gamma(3/2)$.
		\end{itemize} 
		We emphasize that the resulting marginal distributions of the overall errors $ \varepsilon_1$ in the  homo - and heteroscedastic frameworks considered above differ strongly: The  heteroscedastic errors have heavier tails but the bulk of the distribution is more concentrated near zero. See Figure \ref{fig:densitieserrors} for density plots. 
		
			\begin{figure}[!h]        
	\centering	
	\includegraphics[width=0.4\linewidth]{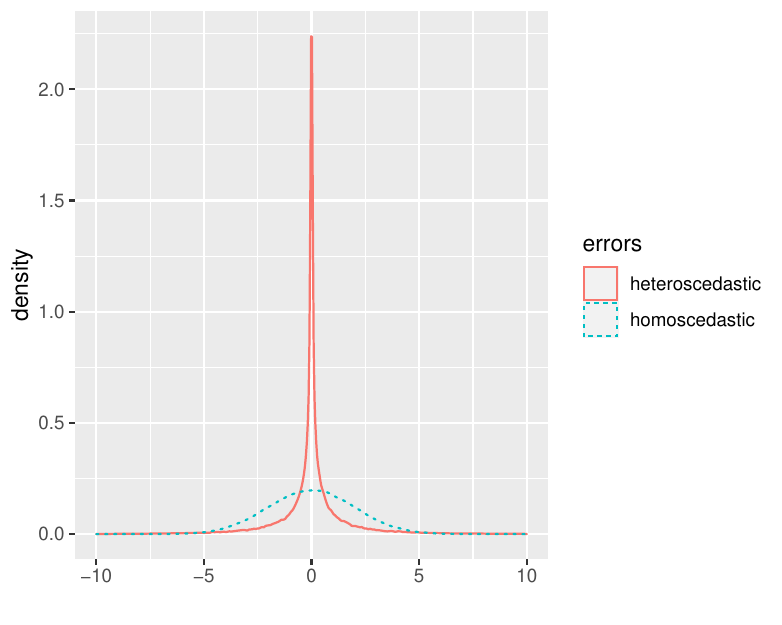}
	\includegraphics[width=0.4\linewidth]{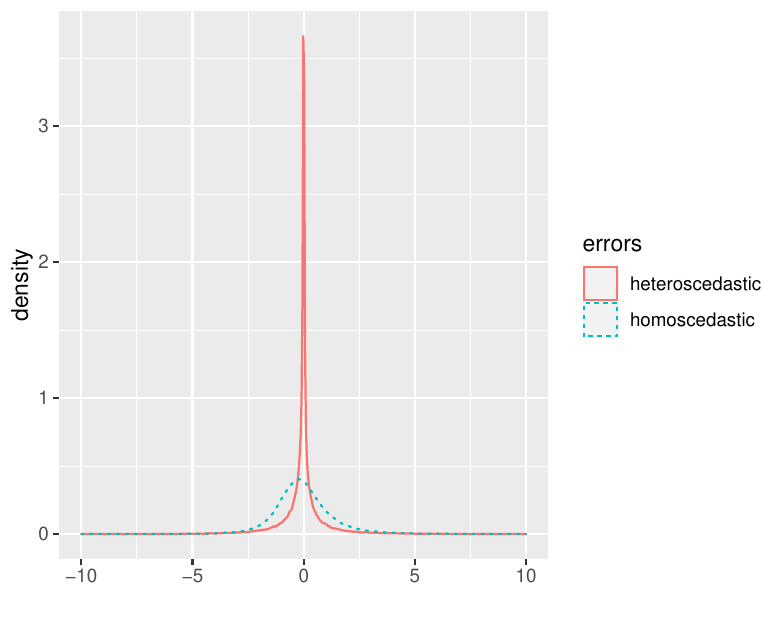}
	\caption{Left figure: Densities for normally distributed (blue) and conditionally normally distributed (red)  errors. Right figure: Densities for skew t-distributed (blue) and conditionally skew t-distributed (red) errors.}\label{fig:densitieserrors}
\end{figure}

		Thus, we expect that robust methods lead to better results for the heteroscedastic than for the homoscedastic setting, as turns out to be the case.


		The simulation results in the above scenarios are displayed in Tables \ref{table:1originalsim} - \ref{table:6originalsim}. Additional results for heavy - tailed but symmetric errors are contained in Section B 
		in the supplementary material.
		Overall we have the following main findings. 
		
		First, for all methods, the  version with adaptive weights is superior to that with ordinary weights for both $\ell_2$ and $\ell_\infty$ estimation error, as well as for the false positive rate (FPR). Second, estimators based on Huber and pseudo Huber loss function perform very similarly. Third, in particular for heteroscedastic errors these estimators have a much better performance than the ordinary LASSO, both in terms of estimation error as well as - in the adaptive versions - for their variable selection properties. Also, the standard errors for the errors in $\ell_2$ - and $\ell_\infty$ - norms displayed in Tables \ref{table:2originalsim} and \ref{table:6originalsim} are smaller.
		Of course, the price to pay is that the additional tuning parameter $\alpha$ has to be chosen. Still, the ordinary LASSO and adaptive LASSO estimator are at least somewhat robust to heteroscedastic errors in that the overall $\ell_2$- and $\ell_\infty$-errors do not increase, see Tables \ref{table:1originalsim} and \ref{table:2originalsim}, which is in line with the theoretical findings in \citet{lederer2020estimating}.


		\begin{table}[!h]
			\centering
			\caption{homoscedastic normally distributed errors. Simulations with $n=200$, $p=400$, $s=20$.}\label{table:1originalsim}
			\begin{tabular}{rrrrrrrr}
				\hline
				& L & AL & LH & LPH & ALH & ALPH (LH) & ALPH (LPH) \\ 
				\hline
				$\lambda$ & 0.154 & 0.695 & 0.157 & 0.150 & 0.066 & 0.067 & 0.069 \\ 
				$\alpha$  &  &  & 0.115 & 0.061 & 0.153  & 0.050 & 0.050 \\ 
				$\ell_2$ norm & 1.66 & 0.93 & 1.67 & 1.67 & 0.83  & 0.83 & 0.83 \\ 
				$\ell_\infty$ norm & 0.60 & 0.41 & 0.61 & 0.61 & 0.38  & 0.38 & 0.38 \\ 
				FPR in \% & 16.14 & 1.83 & 15.76 & 16.32 & 1.06  & 0.99 & 0.97 \\ 
				FNR in \% & 0.00 & 0.00 & 0.00 & 0.00 & 0.00 & 0.00 & 0.00\\ 
				\hline
			\end{tabular}
		\end{table}

		\begin{table}[!h]
			\centering 
			\caption{heteroscedastic, conditionally normally distributed errors. Simulations with $n=200$, $p=400$, $s=20$.}\label{table:2originalsim}
			\begin{tabular}{rrrrrrrr}
				\hline
				& L & AL & LH & ALH & ALPH (LH) \\ 
				\hline
				$\lambda$ & 0.150 & 0.715 & 0.018 & 0.0003 & 0.0003 \\ 
				$\alpha$ &  &  & 3.476 & 57.068 & 55.474 \\ 
				$\ell_2$ norm & 1.65 (0.39) & 0.98 (0.38) & 1.12 (0.24) & 0.23 (0.10) & 0.22 (0.11) \\ 
				$\ell_\infty$ norm & 0.59 (0.14) & 0.41 (0.12) & 0.37 (0.09) & 0.10 (0.04) & 0.09 (0.04) \\ 
				FPR in \% & 15.81 (3.46) & 1.91 (2.81) & 21.47 (1.17) & 0.96  (0.64) & 1.08 (0.74) \\ 
				FNR in \% & 0.00 (0.00) & 0.00 (0.00) & 0.00 (0.00) & 0.00 (0.00) & 0.00 (0.00) \\ 
				\hline
			\end{tabular}
			
		\end{table}
		
		\begin{table}[!h]
			\centering
			\caption{homoscedastic skew t-distributed errors. Simulations with $n=200$, $p=400$, $s=20$.}\label{tab:homoskewt}
			\begin{tabular}{rrrrrrrr}
				\hline
				& L & AL & LH & LPH & ALH & ALPH (LH) & ALPH (LPH) \\ 
				\hline
				$\lambda$ & 0.118 & 0.709 & 0.070 & 0.058 & 0.019 & 0.011 & 0.010\\ 
				$\alpha$ &  &  & 0.863 & 0.871 & 1.124 & 1.842 & 2.184 \\ 
				$\ell_2$ norm & 1.33 & 0.74 & 1.08 & 1.12 & 0.47 & 0.46 & 0.47 \\ 
				$\ell_\infty$ norm & 0.48 & 0.32 & 0.39 & 0.40 & 0.22 & 0.22 & 0.23\\ 
				FPR in \% & 16.37 & 1.56 & 16.48 & 16.49 & 0.52 & 0.63 & 0.53 \\ 
				FNR in \% & 0.01 & 0.01 & 0.00 & 0.00 & 0.02 & 0.00 & 0.01\\ 
				\hline
			\end{tabular}
			
		\end{table}

		\begin{table}[!h]
			\centering
			\caption{heteroscedastic, conditionally skew t-distributed errors. Simulations with $n=200$, $p=400$, $s=20$.}\label{table:6originalsim}
			\begin{tabular}{rrrrrrrr}
				\hline
				& L & AL & LH & ALH & ALPH (LH) \\ 
				\hline
				$\lambda$ & 0.110 & 0.649 & 0.009 & 0.0003 & 0.0002 \\ 
				$\alpha$ &  &  & 7.00 & 33.898 & 50.684 \\ 
				$\ell_2$ norm & 1.28 (0.68) & 0.77 (0.68) & 0.64 (0.15) & 0.11 (0.05) & 0.11 (0.05) \\ 
				$\ell_\infty$ norm & 0.45 (0.19) & 0.32 (0.19) & 0.22 (0.06) & 0.05 (0.02) & 0.05 (0.02) \\ 
				FPR in \% & 16.00 (5.34) & 1.80 (3.35) & 21.18 (1.76) & 0.43 (0.42) & 0.48 (0.45) \\ 
				FNR in \% & 0.02 (0.47) & 0.02 (0.63) & 0.00 (0.00) & 0.00 (0.00) & 0.00 (0.00) \\ 
				\hline
			\end{tabular}
			
		\end{table}

		We also considered a scenario with correlated regressors, which we simulated from a AR(1) model, a common model in econometrics. It gives a correlation between the entries $X_i$ and $X_j$ of the covariates of  $\text{Cor}(X_i, X_j)$ = $\rho^{|i-j|}$ for $i,j = 1,\dotsc,p$, where $\rho$ is the AR(1) - parameter. Thus, $\f{X_1},\dotsc,\f{X_n} \sim \mathcal{N}_{p} \big(\f{0}_{p}, \Sigma \big)$ with $\Sigma \in \R^{p \times p}$, $\Sigma_{ii}=1$ and $\Sigma_{ij}=\rho^{|i-j|}$. For the simulation we took $\rho = 0.9$ thus generating strongly correlated regressors. The results can be found in Tables \ref{tab:corrsimhomo} and \ref{tab:corrsimhetero}. While due to the strong correlation of the regressors overall estimation performance is less strong than in Tables \ref{tab:homoskewt} and \ref{table:6originalsim}, the results are still reasonable and the adaptive LASSO with Huber loss performs best.


		\begin{table}[!h]
			\centering
			\caption{homoscedastic skew t-distributed errors and correlated predictors. Simulations with $n=200$, $p=400$, $s=20$.}\label{tab:corrsimhomo}
			\begin{tabular}{rrrrrr}
				\hline
				& L & AL & LH & ALH \\ 
				\hline
				$\lambda$ & 0.2222 & 0.1542 & 0.0961 & 0.0270 \\ 
				$\alpha$ &  &  & 0.9035 & 0.9255  \\ 
				$\ell_2$ norm & 1.6422 & 1.5942 & 1.2659 & 1.1939 \\ 
				$\ell_\infty$ norm & 0.7818 & 0.7608 & 0.6048 & 0.5780 \\ 
				FPR in \% & 1.6232 & 0.1566 & 1.9413 & 0.0137 \\ 
				FNR in \% & 0.0100 & 0.0200 & 0.0000 & 0.0000 \\ 
				\hline
			\end{tabular}
			
		\end{table}

		\begin{table}[!h]
			\centering
			\caption{heteroscedastic, conditionally skew t-distributed errors and correlated predictors. Simulations with $n=200$, $p=400$, $s=20$.}\label{tab:corrsimhetero}
			\begin{tabular}{rrrrr}
				\hline
				& L & AL & LH & ALH \\ 
				\hline
				$\lambda$ & 4.7848 & 6.0761 & 0.1251 & 0.0380 \\ 
				$\alpha$ & & & 3.0769 & 4.7632 \\ 
				$\ell_2$ norm & 11.4760 & 11.8473 & 3.1094 & 2.5398 \\ 
				$\ell_\infty$ norm  & 5.2745 & 5.3410 & 1.4849 & 1.2584 \\ 
				FPR in \% & 0.1363 & 0.0529 & 2.0805 & 0.1405 \\ 
				FNR in \% & 27.8000 & 38.4650 & 0.0600 & 0.3500 \\ 
				\hline
			\end{tabular}
			
		\end{table}

		\subsection{Incorporating knockoffs}\label{sec:simsnew}

		Here we use a setting similar to \citet{weinstein2020}. Covariates are again independent and standard normally distributed, and heteroscedastic errors are generated as in Section \ref{sec:simsasinfan}.  
		
		Taking advantage of the independence of covariates, we use the methods to generate knockoff variables as described in \citet[Section 2.3 ]{weinstein2020}. For the `counting methods' we generate  a set of $r = 0.3 p$ knockoff variables, thus not pairing individual covariates and their knockoffs, while for the `augmented methods' we use $p$ knockoffs paired to the original covariables. In both cases, the knockoffs are generated as independent and standard normally distributed. As nominal FDR level for the knockoff methods we choose $q=0.05$. For the hyperparameters $\lambda$ and $\alpha$ we use the values as chosen for the estimators. In the counting method we choose the coefficients $\hat{\beta}_j>\tau$ in the tresholding, requiring a strict inequality, since in some simulations the threshold was $\tau = 0$, and using $\hat{\beta}_j \geq \tau$ all coefficients are selected which lead to an increased FPR. 
		
		For the vector of regression coefficients we choose
		\begin{align*}
			\beta^* = \big( 10,\dotsc,10,0,\dotsc,0 \big)^\top \, .
		\end{align*}
		
		For the first scenario we use the parameters $n = 100$, $p = 200$, $s=20$, so that $\delta = n/p = 0.5$, $\rho_{DT}(\delta) \approx 0.35$, and $\rho = s/n = 0.2 < 0.35$ is below the Donoho - Tanner  transition curve \citep[Figure 1]{donoho2009}, so that the setting is favorable for variable selection.

		In Tables \ref{tab1:newsimple}  - \ref{tab4:newsimple} we display the results for the ordinary LASSO, the LASSO with Huber loss, as well as the adaptive LASSO with pseudo Huber loss function. In addition to observations which are similar to those in Section \ref{sec:simsasinfan}, for all three methods, applying the knockoff filter improves the variable selection performance of the methods, substantially for the LASSO - methods, but also for the adaptive LASSO. This has to been seen together with the adaptive LASSO's superior estimation performance. 
		
		\bigskip
		
		Finally, we increase $s=40$ so that $\rho = s/n = 0.4 > 0.35$ is now above the Donoho - Tanner  transition curve, resulting in an unfavorable scenario for variable selection. As displayed in Table 
		\ref{tab6:newsimple}, in combination with the knockoff methodology the adaptive LASSO (with pseudo Huber loss) has a quite reasonable FPR control, which, however, results in more false negatives than for the LASSO methods. Surprisingly, the ordinary LASSO, for which we also used the package \texttt{hqreg} in this setting, has a more favorable performance then the LASSO with Huber loss function.

		\begin{table}[!h]
			\centering
			\caption{homoscedastic normal errors with $n=100$, $p=200$ and $s=20$.}\label{tab1:newsimple}
			\begin{tabular}{crrr}  \hline & L & LH & ALPH(LH) \\   \hline
				$\lambda$ & 0.1321 & 0.0420 & 0.1021 \\   
				$\alpha$ &  & 0.3242 & 0.1000 \\   
				$\ell_2$ norm & 2.7482 & 2.7898 & 1.5024 \\   
				$\ell_\infty$ norm & 1.0193 & 1.0448 & 0.6578 \\   
				FPR in \% & 26.4528 & 25.8134 & 3.6208 \\   
				FNR in \% & 0.0000 & 0.0000 & 0.0000 \\   
				FPR in \% (Knockoff Augmented) & 0.6166 & 0.5287 & 0.2732 \\   
				FNR in \% (Knockoff Augmented) & 0.0000 & 0.0000 & 0.0000 \\   
				FPR in \% (Knockoff Counting) & 1.7414 & 1.6843 & 0.8276 \\   
				FNR in \% (Knockoff Counting) & 0.0000 & 0.0000 & 0.3080 \\    
				\hline\end{tabular}
			
		\end{table}

		\begin{table}[!h]
			\centering
			\caption{heteroscedastic normal errors with $n=100$, $p=200$ and $s=20$.}
			\begin{tabular}{crrr}  \hline & L & LH & ALPH(LH) \\   \hline
				$\lambda$ & 0.1301 & 0.0702 & 0.0032 \\   
				$\alpha$ &  & 2.0652 & 8.2632 \\   
				$\ell_2$ norm & 2.6448 & 2.1385 & 0.6410 \\   
				$\ell_\infty$ norm & 0.9783 & 0.7822 & 0.2693 \\   
				FPR in \% & 25.5842 & 26.7555 & 3.8832 \\   
				FNR in \% & 0.0000 & 0.0000 & 0.0000 \\   
				FPR in \% (Knockoff Augmented) & 0.5894 & 0.5510 & 0.1680 \\   
				FNR in \% (Knockoff Augmented) & 0.0000 & 0.0052 & 0.0052 \\   
				FPR in \% (Knockoff Counting) & 1.7303 & 1.6512 & 0.5126 \\   
				FNR in \% (Knockoff Counting) & 0.0000 & 0.0000 & 0.0000 \\    
				\hline\end{tabular}
			
		\end{table}

		\begin{table}[!h]
			\centering
			\caption{homoscedastic skew t-distributed errors with $n=100$, $p=200$ and $s=20$.}
			\begin{tabular}{crrr}  \hline & L & LH & ALPH(LH) \\   \hline
				$\lambda$ & 0.1061 & 0.0591 & 0.0397 \\   
				$\alpha$ &  & 0.7410 & 0.6000 \\  
				$\ell_2$ norm & 2.1458 & 2.1141 & 0.8876 \\   
				$\ell_\infty$ norm & 0.7940 & 0.7933 & 0.4000 \\   
				FPR in \% & 25.7723 & 25.1944 & 2.5496 \\   
				FNR in \% & 0.0000 & 0.0000 & 0.0000 \\   
				FPR (in \% Knockoff Augmented) & 0.6072 & 0.5664 & 0.2133 \\   
				FNR in \% (Knockoff Augmented) & 0.0000 & 0.0000 & 0.0000 \\   
				FPR in \% (Knockoff Counting) & 1.6208 & 1.6397 & 0.1525 \\   
				FNR in \% (Knockoff Counting) & 0.0000 & 0.0000 & 20.2270 \\    
				\hline\end{tabular}
			
		\end{table}

		\begin{table}[!h]
			\centering
			\caption{heteroscedastic skew t-distributed errors with $n=100$, $p=200$ and $s=20$.}\label{tab4:newsimple}
			\begin{tabular}{crrr}  
				\hline & L & LH & ALPH(LH) \\   \hline
				$\lambda$ & 0.0961 & 0.0671 & 0.0019 \\   
				$\alpha$ &  & 2.3846 & 10.3158 \\   
				$\ell_2$ norm & 2.0060 & 1.3282 & 0.2908 \\   
				$\ell_\infty$ norm & 0.7368 & 0.4869 & 0.1264 \\   
				FPR in \% & 25.0355 & 26.2532 & 2.5040 \\   
				FNR in \% & 0.0000 & 0.0000 & 0.0000 \\   
				FPR in \% (Knockoff Augmented) & 0.6180 & 0.5699 & 0.1581 \\   
				FNR in \% (Knockoff Augmented) & 0.0000 & 0.0000 & 0.0000 \\   
				FPR in \% (Knockoff Counting) & 1.6615 & 1.6523 & 0.3442 \\   
				FNR in \% (Knockoff Counting) & 0.0000 & 0.0000 & 0.0000 \\    
				\hline\end{tabular}
			
		\end{table}

		\begin{table}[!h]
			\centering
			\caption{homoscedastic skew t-distributed errors with $n=100$, $p=200$ and $s=40$}\label{tab6:newsimple}
		\begin{tabular}{crrr}  \hline & L & LH & ALPH(LH) \\   \hline
			$\lambda$ & 0.0440 & 0.0010 & 0.0021 \\   
			$\alpha$ &  & 0.0100 & 7.2105 \\   
			$\ell_2$ norm & 19.5347 & 24.1196 & 19.5423 \\   
			$\ell_\infty$ norm & 5.9852 & 6.9416 & 5.5034 \\   
			FPR in \% & 47.1849 & 38.5621 & 28.3964 \\   
			FNR in \% & 0.3576 & 2.1893 & 2.6331 \\   
			FPR in \% (Knockoff Augmented) & 21.6288 & 16.0540 & 6.2848 \\   
			FNR in \% (Knockoff Augmented) & 8.1360 & 21.4882 & 36.7604 \\   
			FPR in \% (Knockoff Counting) & 14.3911 & 12.4852 & 4.5170 \\   
			FNR in \% (Knockoff Counting) & 2.0311 & 7.7337 & 16.4142 \\    
			\hline\end{tabular}
		
	\end{table}

	\newpage 
	
	\section{Real data example}\label{sec:realdata}

	We revisit the riboflavin data set from \citet{buhlmann2014}. The response variable is the log riboflavin production rate, and as in the mentioned paper we consider a sample of size $n=71$ with $p=4088$ covariates consisting of (log) gene expression levels. See also  \citet{Yi2017} and \citet{tardivel2018}.  
	
	We choose the tuning parameters $\lambda$ and $\alpha$ by 5-fold cross validation, which leads to models with sizes given in the following table. 
	\begin{table}[ht]
		\centering
		\begin{tabular}{rrrrr}  \hline 
			& Lasso & adapt.~Lasso & Lasso Huber & adap.~Lasso Huber \\   \hline 
			CVE (MSE) & 0.2419 & 0.0825 & 0.2215 & 0.3705 \\   
			Number of choosen covariates & 42 & 19 & 16 & 8 \\    \hline
		\end{tabular}
	\end{table}
	
	Note that \citet{buhlmann2014} used 10-fold cross validation, which results in a model with 30 covariates for the ordinary LASSO. 
	The gene \textit{YXLD\_at} which  is identified as significant in \citet{buhlmann2014} is contained in all models together with the gene \textit{ARGF\_at}, all other 40 covariates selected by the classical and hence all 17 selected by the  classical adaptive LASSO differ from the 14 selected by the LASSO with Huber loss. This can be interpreted as a further indication that the gene \textit{YXLD\_at} is significant, while few others are. While the adaptive LASSO with Huber loss functions selects the smallest model, estimation performance as assessed by cross validation is best for the ordinary adaptive LASSO with squared error loss.
	
	In Figures \ref{fig:residulas} and \ref{fig:residulas1} we plot the residuals obtained from the fits with the classical LASSO as well as the LASSO with Huber loss function, the residual plots for the adaptive versions were similar. The plots show departure from normality, and indicate heteroscedasticity as well.

			\begin{figure}[!h]       
	\centering	
	\includegraphics[width=0.9\linewidth]{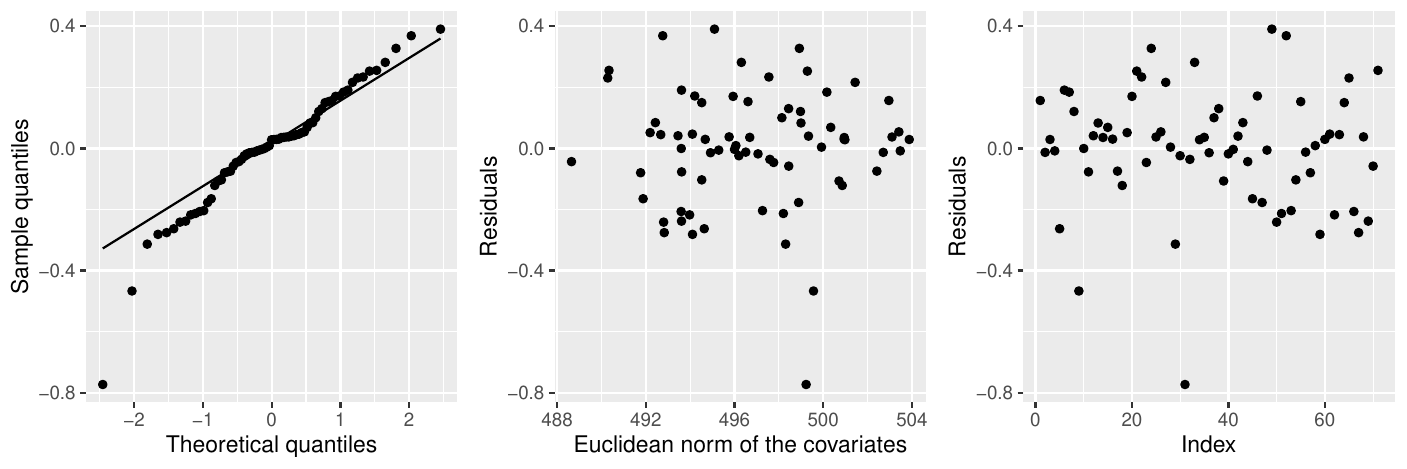}
	\caption{Residuals of LASSO}\label{fig:residulas}
\end{figure}

\begin{figure}[!h]       
	\centering	
	\includegraphics[width=0.9\linewidth]{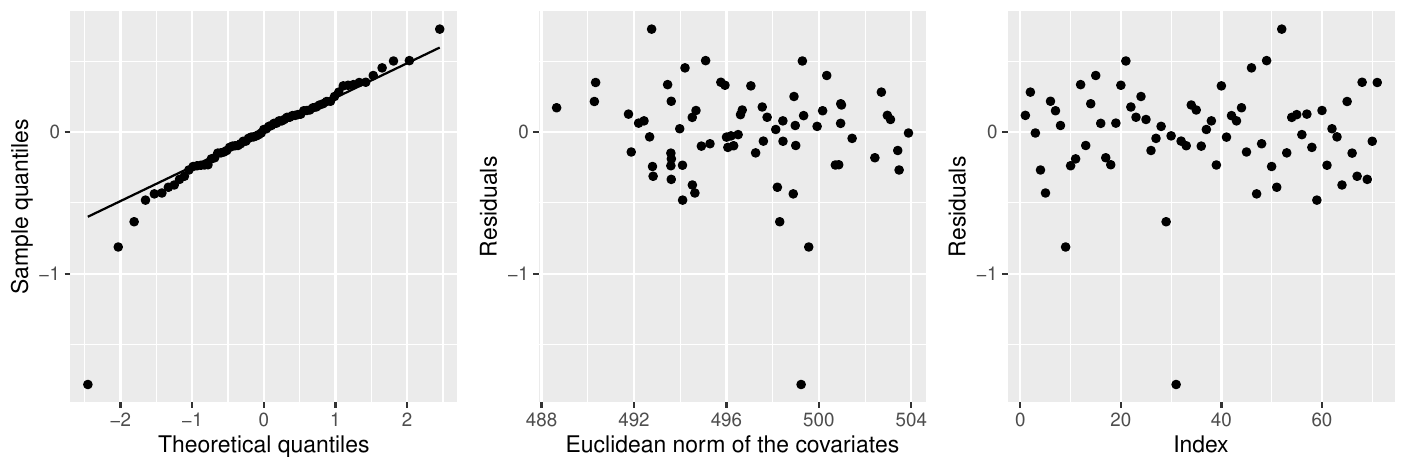}
	\caption{Residuals of LASSO Huber estimate in the real-data example}\label{fig:residulas1}
\end{figure}
	
	%
	%
	
	%
	\section{Conclusions}\label{sec:conclusions}

	In their recent paper, \citet{Sun2019} extended the analysis from \citet{Fan2017} to fixed designs, as well as to conditional moments of $\varepsilon_1$ of order strictly smaller than $2$, in which case they showed that the rates of convergence deteriorate. Results on support estimation, rates of convergence in the $\ell_\infty$ norm together with a data-driven choice of the robustification parameter would be of some interest in this setting as well. 
	Another possible extension or modification of our method would be the use of nonconvex penalty functions such as SCAD as in \citet{Loh2017}, with the methodological aim to avoid a two-stage procedure.
	
	
	The paper was partially motivated by the problem of selecting the random or correlated coefficients in a linear random coefficient regression model
	\begin{equation}\label{eq:linrancoef}
		Y_j = \f{ X_j^\top} \beta_j + \varepsilon_j\,, \qquad j=1, \ldots, n,
	\end{equation}
	where $\beta_j$ are also independent and identically distributed random vectors. 
	Versions of this model have been studied quite intensely - mainly in a nonparametric framework -  in the recent econometrics literature \citep{hoderlein2010analyzing, dunker2017tests}. 
	
	Writing $\f{ \bar X_j} = (1, \f{ X_j^\top})^\top$ and $\theta_j = (\varepsilon_j, \beta_j^\top)^\top $ we may consider the heteroscedastic regression model 
	\[ 	Y_j = \f{\bar X_j^\top} \E[\theta_j] + \f{\bar X_j^\top} \big(\theta_j - \E[\theta_j]\big) \]
	as model for the first moments of the $\beta_j$ and $\varepsilon_j$. A similar, but more complicated heteroscedastic mean regression model - involving products and squares of entries of $\f{\bar X_j}$ - can be designed for the entries of the covariance matrix of the $\theta_j$, where the response $\big(Y_j - \E[Y_j]\big)^2$ also involves the estimation error from the first stage mean regression. Support estimation then selects those coefficients with non-zero variances as well as the correlated pairs of coefficients. 

	Another extension of some interest would be to robustify asymmetric versions of least squares regression \citep{newey1987asymmetric, gu2016high}, that is, high-dimensional expectile regression.

	{\small
		
		\section{Proofs: Main steps}\label{sec:proofmainsteps}
		
		\subsection{Outline of the steps of the proof}
		
		Let us start with an outline of the main steps in the proofs of Theorems \ref{signconsadaptivehuberwithinitialellinf} and \ref{signconsadaptivehuberwithinitial}. Proofs of these main steps are then given in the subsequent sections, while various further technical details are deferred to Section A 
		in the supplementary appendix.
		%
		\begin{enumerate}
			\item \textit{(Reduction to compact parameter set)}. First, in Section \ref{sec:reduc} we reduce the optimization problem in the definition of the estimator $\widehat{\beta}_n^{\,\WLH}$ to a compact set, as is required in the subsequent analysis. 
			\item \textit{(Approximation error)}. Next, in Section \ref{sec:approx} we bound the approximation error $\|\beta^* - \beta^*_{\alpha_n}\|_2$.
			\item \textit{(Primal - dual witness approach under high-level assumptions for general weights)}. For the weighted LASSO estimator with pseudo-Huber loss function, \eqref{defweightedlassohuber1}, in Section \ref{sec:pdwsteps} we list the steps in the primal-dual witness proof method from \citet{wainwright2009information, Loh2017}. Lemma \ref{corsignconsweightedhuber} in Section \ref{sec:pdwsteps} then shows how to implement this approach for general weights under high-level assumptions, resulting in guarantees for support recovery and in $\ell_\infty$ - bounds.  This is the main novel result which shows how to deal with the additional terms arising from the fact that the support $S_{\alpha}$ of $\beta^*_\alpha$, may differ from $S$, the support of $\beta^*$ which is the object of interest.
			\item \textit{(PDW under low-level assumptions for general weights)} In Lemma \ref{signconsweightedhuber} in Section \ref{sec:resultsweightedlasso} we replace the high-level assumption on strict - dual feasibility from step 3.~with a low-level mutual incoherence condition involving the a-priori given weights. We obtain explicit bounds on the $\ell_\infty$ - error and guarantees for support recovery.  
			\item Next in  Lemma \ref{signconsadaptivehuber} we turn to the adaptive LASSO pseudo Huber estimator under bounds for the $\ell_\infty$ - risk of the initial estimator, and check that the weights satisfy the genric conditions from \ref{signconsweightedhuber} with high probability.  Together with Lemmas \ref{lemmaassstrictdual(ii)} and \ref{rateofrnS}, where the latter takes care of the mutual incoherence condition, this proves Theorem \ref{signconsadaptivehuberwithinitialellinf}. 
			\item  Finally in Lemma \ref{signconsadaptivehuberwithinitial1} we analyze the adaptive LASSO pseudo Huber estimator under bounds for the $\ell_2$ - risk of the initial estimator, which then leads to Theorem \ref{signconsadaptivehuberwithinitial}. Since this risk involves the parameter $s^{1/2}$ which then occurs in the choice of the regularization parameter $\lambda_n$ for the adaptive LASSO pseudo Huber estimator, we require an estimator of $s$ which is provided through the set $S$ in Lemma \ref{thresholdingprocedureinitial}. 
		\end{enumerate}

		We shall use the following additional notation.  
		%
		%
		$\X_n = \big(
		\f{X_1},\dotsc,\f{X_n}\big)^\top \in \R^{n \times p}$ is the design matrix, where $\f{X_i} \in \R^p$ is the covariate vector in model \eqref{generalregressionmodel}.
		$w = (w_1, \ldots, w_p)^\top$ denotes the vector of weights from \eqref{defweightedlassohuber1}, and we set $w_{\max}\big(S\big) = \max_{i \in S} w_i$ and $w_{\min}\big(S^c\big) = \min_{i \in S^c} w_i$.
		We denote by $\nabla$ the gradient of a smooth function, and by $\partial$ the subgradient  of a convex function. For vectors $x,y$ of same dimension we denote by $x \odot y$ their Hadamard product (that is coordinate wise product). Inequality signs such as $x < y$ are understood component wise. A diagonal matrix with real entries $d_1, \ldots, d_n$ is denoted by $\diag(d_1, \ldots, d_n)$. 
		\subsection{Preparations}\label{sec:technicalpreps}

		\subsubsection{Reduction of estimator to compact domain}\label{sec:reduc}
		
		Take a constant $\Cbeta \geq \max(1/8, 2\, \normz{\beta^*})$. Our analysis below will then be applied to the estimator restricted to  $\normz{\beta} \leq \Cbeta$, 
		\begin{align}
			\widehat{\beta}_n^{\,\WLH} \in \underset{\beta \in \R^p, \, \normz{\beta} \leq \Cbeta}{\arg\min} ~ \bigg( \loss_{n,\alpha_n}^{\,\Hu}\big(\beta\big) +  \lambda_n  \sum_{k=1}^p w_k\,|\beta_k|\bigg) \,  \label{defweightedlassohuber}
		\end{align}
		Set also
		\begin{align}\label{eq:betaalpha}
			\beta_{\alpha_n}^* \defeq \underset{\beta \in \R^p, \, \normz{\beta} \leq \Cbeta}{\arg\min} ~ \E\Big[l_{\alpha_n} \big( Y_1 - \f{X_1^\top} \beta \big)\Big]\,.  
		\end{align}
		From the error bounds in Lemma \ref{approximation:error:pseudo:huber} for $\beta_{\alpha_n}^*$ and in Theorem \ref{signconsadaptivehuberwithinitial}, first obtained for the definitions in \eqref{defweightedlassohuber} and \eqref{eq:betaalpha}, together with the uniqueness of the global minimum in \eqref{defweightedlassohuber1} \citep{Tibshirani2013}, it follows that the solutions of \eqref{defweightedlassohuber} and \eqref{defweightedlassohuber1} coincide with high probability and that the error bounds also apply to the global optimum. 
		
		However, the error bounds, for example of the approximation error in Lemma \ref{approximation:error:pseudo:huber} or the bound in the primal-dual witness construction in Lemma \ref{lemma:l2:norm:pdw} depend on $C_\beta$ and thus on the norm $\normz{\beta^*}$.

		\bigskip
		
		\subsubsection{Bounding the approximation error}\label{sec:approx}
		
		We start by bounding the approximation error $\|\beta^* - \beta^*_{\alpha_n}\|_2$. 
		
		The proofs of the lemmas in this and the next  subsection, which extend results from \citet{Fan2017} to the pseudo Huber loss function $\l_\alpha$ are provided in the supplement,  Section A.1. 
		See also \citet[Section 5]{Fan2017} for similar extensions to the Cantoni loss function \citep{catoni2012challenging}. 

		To start, straightforward differentiation gives 
		\begin{align}\label{eq:firstderivative}
			l_\alpha'(x) = \frac{2 x}{\sqrt{1+\alpha^2 x^2 }}\, \qquad \text{so that} \qquad \big|l_\alpha'(x)\big| \leq \frac{2 |x|}{\sqrt{\alpha^2 x^2 } } = 2 \alpha^{-1},
		\end{align}
		and 
		\begin{align} \label{eq:secondderivative}
			l_\alpha''(x)  = \frac{2 \alpha^{-3}}{(\alpha^{-2} + x^2)^{3/2}} \qquad \text{so that} \qquad 0 < l_\alpha''(x) \leq  \frac{2 \alpha^{-3}}{(\alpha^{-2})^{3/2}} = 2 \,.
		\end{align}
		In particular $l_\alpha$ is strictly convex. Also note that  $\lim_{\alpha \to 0} l_\alpha(x) = x^2$ for all $x \in \R$.  
		For the empirical loss function in \eqref{eq:emploss} this gives
		\begin{align}\label{eq:gradientgeneral}
			\nabla \loss_{n,\alpha}^{\,\Hu}(\beta) = - \frac{1}{n} \sum_{i=1}^{n} l_\alpha'\big(Y_i-\f{X_i^\top} \beta\big) \f{X_i}, \qquad 	\nabla^2 \loss_{n,\alpha}^{\,\Hu}(\beta) = \frac{1}{n} \sum_{i=1}^{n} l_\alpha''\big(Y_i-\f{X_i^\top} \beta\big) \f{X_i} \f{X_i^\top}\,.
		\end{align}

		The following result is similar to \citet[Theorem 1 and Theorem 6]{Fan2017}, however, we work with a weaker moment assumption.

		\begin{lemma} \label{approximation:error:pseudo:huber}
			Under Assumption \ref{assfan2017} we have for $\beta_{\alpha_n}^*$ in \eqref{eq:betaalpha} that 
			\begin{align}
				\normz{\beta_{\alpha_n}^* - \beta^*} \leq \Capprox  \, \alpha_n^{\Cpm-1} \,, \label{approx:error}
			\end{align} 
			where 
			\begin{align*}
				\Capprox = \frac{5\,2^{\Cpm } \, \CXsub}{\CXl} \, \Bigg[ \Bigg( \frac{q}{q-1} \, \Gamma\bigg(\frac{q}{2(q-1)} \bigg) \Bigg)^\frac{q-1}{q}  (\Cm)^\frac{1}{q}  + \Big(2 \, \big(2 \Cbetas \, \CXsubs\big)^\Cpm \, (2\Cpm)!\, \Gamma(\Cpm) \Big)^\frac{1}{2} \Bigg] \
			\end{align*}
			and $\Gamma(x) = \int_0^\infty t^{x-1} \, \exp(-t) \,dt$, $x>0$, is the gamma function.
		\end{lemma}
		\begin{remark}
			The above result leads to  $\big\|  \beta_{\alpha_n}^* - \beta^* \big\|_2 < \Cbeta/2$ for an (appropriate) choice of $\alpha_n$. Together with the assumption $\normz{ \beta^*} \leq \Cbeta/2$ this will imply that $\beta_{\alpha_n}^*$ is strictly feasible for \eqref{eq:betaalpha}, that is, 
			\begin{equation}\label{eq:strctfeasiblebetaalpha}
				\normz{\beta_{\alpha_n}^*} < \Cbeta,
			\end{equation}
			which we will assume from now on.	
		\end{remark}

		\subsubsection{Restricted strong convexity and properties of derivatives}\label{sec:restrstrognconv}

		Next we show along the lines of \citet[Lemmas 2 and 4]{Fan2017} that restricted strong convexity is satisfied by the pseudo  Huber loss function, which is a core ingredient for the primal-dual witness proof method for general regularized M - estimators \citep{Loh2017}. 

		\begin{lemma} \label{lemma:RSC:pseudo:huber}
			Under Assumption \ref{assfan2017} there exist $\calpha>0$ (depending on $\CXl$, $\CXu$, $\CXsub$ and $\Cbeta $) and $\cPo, \cPt>0$ (depending on $\CXl$ and $\CXsub$) such that for all $\normz{\beta} \leq 4 \Cbeta$, $\normz{\Delta } \leq 8 \Cbeta$	
			and $\alpha \leq \calpha$ with probability at least $1-\cPo \exp(-\cPt n)$ the empirical pseudo Huber loss function $\loss_{n,\alpha}^{\,\Hu}$ satisfies the restricted strong convexity condition
			
			\begin{equation}\label{eq:rsccond}
				\skalar{\nabla \loss_{n,\alpha}^{\,\Hu} (\beta+\Delta) - \nabla \loss_{n,\alpha}^{\,\Hu} (\beta)}{\Delta} \geq \Crscf
				\, \normzq{\Delta} - \Crscs
				\frac{\log(p)}{n} \normeq{\Delta}
			\end{equation}
			
			with
			\begin{align*}
				\Crscf = \frac{\CXl}{16}\,, \qquad \Crscs = \frac{1600\,\CXsubs \, \big(\max\big\{ 4 \CXsub \sqrt{\log(12 \CXsubs/\CXl)},1 \big\} \big)^4}{\CXl} \,.
			\end{align*}
		\end{lemma}
		
		Restricted strong convexity in particular implies ordinary strong convexity locally on the support $S$ of $\beta^*$. 
		
		\begin{lemma}\label{lem:derivativelem1}
			Under Assumption \ref{assfan2017}, if $\alpha \leq \calpha$ and $n \geq \Crsct s \log(p)$ with $\Crsct=2 \Crscs / \Crscf$ we have with probability at least $1-\cPo \exp(-\cPt n)$ for $\beta \in \R^p$ with $\lVert \beta \rVert_2 \leq 4 \Cbeta$ that
			\begin{align}
				\lambda_{\min}\Big(\big(\nabla^2 \loss_{n,\alpha}^{\,\Hu}(\beta) \big)_{S S} \Big) \geq \frac{\Crscf}{2} = \frac{\CXl}{32}. \label{positvedefinitenesshessian}
			\end{align}
		\end{lemma}

		The following result gives a bound on the gradient of the empirical loss function, and is analogous to \citet[Lemma 1]{Fan2017}.
		\begin{lemma}\label{ratelinftygradient} 
			Under Assumption \ref{assfan2017} there exist $\Cgradf, \Cgrads>0$ (depending on $\Choeld$, $\Cm$, $\CXsub$ and $\Cbeta$) such that for all $\alpha_n \geq \Cgradf \, \big(\frac{\log(p)}{n}\big)^{\frac{1}{2}}$ with probability at least $1-2/p^2$ the $\ell_\infty$ norm of the gradient of the empirical pseudo Huber loss function at $\beta_{\alpha_n}^*$ is bounded by 
			\begin{align}
				\normi{\nabla \loss_{n,\alpha_n}^{\,\Hu} \big(\beta_{\alpha_n}^*\big)} \leq \Cgrads \, \bigg( \frac{\log(p)}{n}\bigg)^\frac{1}{2} \label{defsetlinftygradient} \,.
			\end{align}
		\end{lemma}
		
		\subsection{Primal-dual witness approach}\label{sec:primaldualwit}

		\subsubsection{Steps in the primal-dual witness approach}\label{sec:pdwsteps}
		The proof of Theorem \ref{signconsadaptivehuberwithinitial} is based on the primal-dual witness (PDW) approach as originally introduced in \citet*{Wainwright2009}. 
		The main novel result in this section for implementing this approach in our setting is Lemma \ref{corsignconsweightedhuber}. It shows in particular  how to deal with the additional terms arising from the fact that the support $S_{\alpha}$ of $\beta^*_\alpha$, may differ from $S$, the support of $\beta^*$ which is the object of interest.
		
		\begin{itemize}
			\item[(i)] Optimize the restricted program 
			\begin{align}
				\widehat{\beta}_n^{\,\PDW} \in \underset{\beta \in \R^{p}, \supp(\beta) \subseteq S,  \normz{\beta} \leq \Cbeta}{\arg\min} ~ \bigg( \loss_{n,\alpha_n}^{\,\Hu} \big(\beta\big) +  \lambda_n  \sum_{k \in S} w_k\,|\beta_k|\bigg)\,, \label{beta:PDW}
			\end{align} 
			where we enforce the constraint that $\supp\big(\widehat{\beta}_n^{\,\PDW}\big) \subseteq S $, and show that all solutions have norm $<  \Cbeta$. 
			\item[(ii)] Choose $\widehat{\gamma} = \widehat{\gamma}_n \in \R^p$ such that (a.) $\widehat{\gamma}_S \in \partial \big\lVert \widehat{\beta}_{n,S}^{\,\PDW} \big\rVert_1$, (b.) it satisfies the zero-subgradient condition
			\begin{align}\label{eq:zerosubgradpdw}
				\nabla \loss_{n,\alpha_n}^{\,\Hu} \big(\widehat{\beta}_n^{\,\PDW}\big) + \lambda_n \big( w \odot \widehat{\gamma} \big) = \f{0}_p\,,
			\end{align}
			and (c.) such that $\widehat{\gamma}_{S^c}$ satisfies the strict dual feasibility condition $\normi{\widehat{\gamma}_{S^c}} < 1$.
			\item[(iii)] Show that $\widehat{\beta}_n^{\,\PDW}$ is also a minimum of the full program \eqref{defweightedlassohuber},
			\begin{align*}
				\underset{\beta \in \R^p, \,\normz{\beta} \leq \Cbeta}{\arg\min} ~ \Bigg( \loss_{n,\alpha_n}^{\,\Hu} \big(\beta\big) +  \lambda_n  \sum_{k =1}^p w_k\,|\beta_k|\Bigg) \,,
			\end{align*} 
			and moreover the uniqueness of the minimizer of this program. 
		\end{itemize}
		%
		
		We shall always assume that 
		\begin{align*}
			\alpha_n \leq \calpha
		\end{align*}
		holds, where $\calpha$ is given in Lemma \ref{lemma:RSC:pseudo:huber}. 
		
		\subsubsection{Solving the PDW construction}\label{sec:solvepdw}
		
		We introduce the notation 
		$$\widehat{Q} \defeq \int_{0}^{1} \nabla^2 \loss_{n,\alpha_n}^{\,\Hu} \Big(\beta_{\alpha_n}^* + t \, \big(\widehat{\beta}_n^{\,\PDW}-\beta_{\alpha_n}^*\big)\Big) dt \,.$$
		\begin{lemma}[Solving the PDW construction] \label{corsignconsweightedhuber}
			Suppose that Assumption \ref{assfan2017} holds and
			that $\beta^*$ satisfies the beta-min condition  
			%
			\begin{align}\label{betamin:fanreg}
				\beta_{\min}^*  > \Capprox  \, \alpha_n^{\Cpm-1}
			\end{align} 
			and that $n \geq \Crsct s \log(p)$.
			Let $\widehat{\beta}_n^{\,\PDW}$ be as in the PDW construction, and suppose that $\widehat{\gamma} \in \R^p$ satisfies $\widehat{\gamma}_S \in \partial \big\lVert \widehat{\beta}_{n,S}^{\,\PDW} \big\rVert_1$ and the zero-subgradient condition \eqref{eq:zerosubgradpdw}. 
			Then,  with probability at least $1-\cPo \exp(-\cPt n)$ the strict dual feasibility condition $\normi{\widehat{\gamma}_{S^c}} < 1$ is equivalent to the condition
			\begin{align}
				&\Bigg| \widehat{Q}_{S^c S} \big(\widehat{Q}_{S S}\big)^{-1} \bigg( \lambda_n \big( w_{S} \odot \widehat{\gamma}_{S} \big) + \Big( \nabla \loss_{n,\alpha_n}^{\,\Hu} \big(\beta_{\alpha_n}^*\big)\Big)_{S} \bigg) - \Big(\nabla \loss_{n,\alpha_n}^{\,\Hu} \big(\beta_{\alpha_n}^*\big)\Big)_{S^c}  \notag \\
				&~~~~~~~~~~~~~~~~~~~~~~~~~~~~+ \Big( \widehat{Q}_{S^c (S_{\alpha_n} \setminus S)} - \widehat{Q}_{S^c S} \big(\widehat{Q}_{S S}\big)^{-1} \widehat{Q}_{S (S_{\alpha_n} \setminus S)}  \Big) \,\beta_{\alpha_n,S_{\alpha_n} \setminus S}^* \Bigg| < w_{S^c}\, \lambda_n \, . \label{strictfeasibilitywh2}
			\end{align}
			Furthermore, if \eqref{strictfeasibilitywh2} is satisfied  we have   that the minimizer $\widehat{\beta}_n^{\,\WLH}$ in \eqref{defweightedlassohuber} is unique and given by $\widehat{\beta}_n^{\,\WLH} = \widehat{\beta}_n^{\,\PDW}$, so that in particular $\supp\big(\widehat{\beta}_n^{\,\WLH}\big) \subseteq \supp\big(\beta^*\big)$, and furthermore that
			\begin{equation}\label{eq:firstgenericnormbound}
				\normi{\widehat{\beta}_n^{\,\WLH}-\beta^*} \leq \phi_{n,\infty},
			\end{equation} 
			where 
			%
			\begin{align}
				\phi_{n,\infty} & = \normiM{\big(\widehat{Q}_{S S}\big)^{-1} }\, \normi{\Big( \nabla \loss_{n,\alpha_n}^{\,\Hu} \big(\beta_{\alpha_n}^*\big)\Big)_{S} } + w_{\max}\big(S\big) \, \lambda_n \,  \normiM{\big(\widehat{Q}_{S S}\big)^{-1} } \notag \\
				&~~~~~~~~~~ + \normi{\beta_{\alpha_n,S}^* - \beta_{S}^*} + \normiM{\big(\widehat{Q}_{S S}\big)^{-1} }\, \normi{\big( \widehat{Q}_{S (S_{\alpha_n} \setminus S)} \,\beta_{\alpha_n,S_{\alpha_n} \setminus S}^*}. \label{eq:boundesterrorwlgeneralpdw}
			\end{align}
			Furthermore, if we have in addition the beta-min condition of the same order  
			\begin{align}
				\beta_{\min}^* &> \phi_{n,\infty}, \label{betaminweightedhuber}
			\end{align}
			then we have the sign-recovery property  $\sign\big(\widehat{\beta}_n^{\,\WLH}\big) = \sign\big(\beta^*\big)$.
		\end{lemma}
		The beta-min condition \eqref{betamin:fanreg} is required so that the approximation error  $\normi{\beta_{\alpha_n}^* -  \beta^*}$ is small, which implies that the support $S_{\alpha_n}$ of $\beta_{\alpha_n}^*$ contains the support $S$ of $\beta^*$. 
		Later, we shall choose $\alpha_n$ and achieve a rate $\phi_{n,\infty}$, which in any case also includes an approximation term $\normi{\beta_{\alpha_n,S}^* - \beta_{S}^*}$, such that  \eqref{betaminweightedhuber} implies \eqref{betamin:fanreg}.

		Later, $\alpha_n$ is chosen of an order tending to zero, so that this is automatically satisfied. 
		
		\subsubsection{The PDW-approach: Technical details and proof of Lemma \ref{corsignconsweightedhuber}}
		
		The following lemma lists some technical properties of the derivatives of the empirical loss function.  
		\begin{lemma}\label{lem:derivativelem}
			We may write 
			\begin{equation}
				\widehat{Q} \defeq \int_{0}^{1} \nabla^2 \loss_{n,\alpha_n}^{\,\Hu} \Big(\beta_{\alpha_n}^* + t \, \big(\widehat{\beta}_n^{\,\PDW}-\beta_{\alpha_n}^*\big)\Big) dt \, = \frac{2}{n} \sum_{i=1}^{n} d_i \,\f{X_i}\, \f{X_i}^\top = \frac{2}{n} \, \X_{n}^\top\, D \, \X_{n} \, , \label{defofQhat}
			\end{equation} 
			where 	 $D=\diag\big(d_1,\dotsc,d_n\big)$ with 
			\begin{align*}
				d_i = \frac{1}{2}\, \int_{0}^{1} l_{\alpha_n}''\,\Big(Y_i-\f{X_i}^\top \big(\beta_{\alpha_n}^* + t \, \big(\widehat{\beta}_n^{\,\PDW}-\beta_{\alpha_n}^*\big)\big)\Big)  dt \quad \quad \in (0,1] \,. 
			\end{align*}
			
			Furthermore, under Assumption \ref{assfan2017}, if $n \geq \Crsct s \log(p)$ with probability at least $1-\cPo \exp(-\cPt n)$ the submatrix $\widehat{Q}_{S S}$ is invertible with minimal eigenvalue bounded below by $\CXl/32$ and we have the bound
			\begin{equation}\label{inftyoperatornormQss}
				\normiM{\big(\widehat{Q}_{S S}\big)^{-1}}  \leq \frac{32 \sqrt{s}}{\CXl}.
			\end{equation}
		\end{lemma}
		\begin{proof}[Proof of Lemma \ref{lem:derivativelem}]
			\eqref{defofQhat} follows from straightforward calculation, see \eqref{eq:gradientgeneral}. Moreover, every point $\beta \in \R^p$ between $\beta_{\alpha_n}^*$ and $\widehat{\beta}_n^{\,\PDW}$ has $\ell_2$ norm smaller than or equal to $\Cbeta$ because $\normz{ \beta_{\alpha_n}^* }, \big\lVert \widehat{\beta}_n^{\,\PDW} \big\rVert_2 \leq \Cbeta$. Hence \eqref{positvedefinitenesshessian} implies the invertibility of $\widehat{Q}_{S S}$ and 
			\eqref{inftyoperatornormQss} follows from  
			\[ \normiM{\big(\widehat{Q}_{S S}\big)^{-1}} \leq \sqrt{s}\, \normzM{\big(\widehat{Q}_{S S}\big)^{-1}} \leq \frac{32 \sqrt{s}}{\CXl}.\]
		\end{proof}

		In the following lemma we show that $\widehat{\beta}_n^{\,\PDW}$ is strictly feasible, meaning $\big\lVert \widehat{\beta}_n^{\,\PDW} \big\rVert_2 < \Cbeta$ holds, for an appropriate choice of $\lambda_n$ and $\alpha_n$.

		\begin{lemma} \label{lemma:l2:norm:pdw}
			Under Assumption \ref{assfan2017} we have for $\widehat{\beta}_n^{\,\PDW}$ in \eqref{beta:PDW} with $\alpha_n \geq \Cgradf \, \big(\frac{\log(p)}{n}\big)^{\frac{1}{2}}$ that 
			\begin{align}
				\normz{\widehat{\beta}_n^{\,\PDW} - \beta^*} \leq \Bigg( \Cgrads \, \bigg( \frac{\log(p)}{n}\bigg)^\frac{1}{2} + w_{\max}\big(S\big) \, \lambda_n + 2 \Cbeta \,\Crscs \frac{\sqrt{s} \log(p)}{n} \Bigg) \, \frac{\sqrt{s}}{\Crscf} + \Capprox  \, \alpha_n^{\Cpm-1}  \label{l2:norm:pdw}
			\end{align} 
			with probability at least $1-\cPo \exp(-\cPt n)-2/p^2$.
		\end{lemma}

		\begin{proof}[Proof of Lemma \ref{lemma:l2:norm:pdw}]
			Let 
			\begin{align}
				\betaS	= \underset{\beta \in \R^p, \supp(\beta) \subseteq S,  \normz{\beta} \leq \Cbeta}{\arg\min} ~ \E\Big[l_{\alpha_n} \big( Y_1 - \f{X_1^\top} \beta \big)\Big] \quad \text{and} \quad \Delta_n^{\PDW} = \widehat{\beta}_n^{\,\PDW} - \betaS\,, \label{def:beta:stern:S}
			\end{align}
			then $\widehat{\beta}_n^{\,\PDW}$ in \eqref{beta:PDW} is a M-estimator of $\betaS$. Following the proof of Lemma \ref{approximation:error:pseudo:huber} leads on the one hand to
			\begin{align*}
				\normz{\betaS - \beta^*} \leq\Capprox  \, \alpha_n^{\Cpm-1} \,.
			\end{align*}  
			In doing so note that 
			\begin{align*}
				\E\Big[ l_{\alpha_n}\big(Y_1 - \f{X_1^\top} \betaS \big) \Big] \leq  \E\Big[ l_{\alpha_n}\big(Y_1 - \f{X_1^\top} \beta^* \big) \Big] \quad \text{and} \quad \normz{\betaS} \leq \Cbeta
			\end{align*}
			because of \eqref{def:beta:stern:S}, $\supp\big(\beta^*\big) = S$ and $\normz{\beta^*} \leq \Cbeta$ by (iv) of Assumption \ref{assfan2017}. Further, $\widehat{\beta}_n^{\,\PDW}$ has to satisfy the first-order necessary condition of a convex constrained optimization problem over a convex set to be a minimum of \eqref{beta:PDW}, cf. \citet[Theorem 3.33]{Ruszczynski2006}, that is, there exists $\widehat{\gamma} \in \partial \big\lVert \widehat{\beta}_{n,S}^{\,\PDW} \big\rVert_1$ so that
			\begin{align*}
				\skalar{\nabla \loss_{n,\alpha_n}^{\,\Hu}\big(\widehat{\beta}_n^{\,\PDW}\big) + \lambda_n \big( w \odot \widehat{\gamma} \big) }{\beta - \widehat{\beta}_n^{\,\PDW}} \geq 0 \quad \text{for all feasible } \beta \in \R^p \,.
			\end{align*}
			Hence by the restricted strong convexity of the empirical pseudo Huber loss in Lemma \ref{lemma:RSC:pseudo:huber} and the first-order necessary condition it follows that
			\begin{align*}
				\Crscf\, \normzq{\Delta_n^{\PDW}} - \Crscs \frac{\log(p)}{n} \normeq{\Delta_n^{\PDW}} &\leq \skalar{\nabla \loss_{n,\alpha_n}^{\,\Hu} (\widehat{\beta}_n^{\,\PDW}) - \nabla \loss_{n,\alpha_n}^{\,\Hu} (\betaS)}{\Delta_n^{\PDW}} \\
				&\leq \skalar{ - \nabla \loss_{n,\alpha_n}^{\,\Hu} (\betaS) - \lambda_n \big( w \odot \widehat{\gamma} \big)}{\Delta_n^{\PDW}}\\
				&\leq \normi{ \nabla \loss_{n,\alpha_n}^{\,\Hu} (\betaS) }\, \norme{\Delta_n^{\PDW}} + w_{\max}\big(S\big) \, \lambda_n \, \norme{\Delta_n^{\PDW}} 
			\end{align*}
			with probability at least $1-\cPo \exp(-\cPt n)$. Here the last inequality follows since $\betaS$ and $\widehat{\beta}_n^{\,\PDW}$ both have support (contained in) $S$. Rearranging leads to
			\begin{align*}
				\Crscf\, \normzq{\Delta_n^{\PDW}} \leq \bigg( \normi{\nabla \loss_{n,\alpha_n}^{\,\Hu} (\betaS)} + w_{\max}\big(S\big) \, \lambda_n + \Crscs \frac{\log(p)}{n} \norme{\Delta_n^{\PDW}} \bigg)  \norme{\Delta_n^{\PDW}} \,.
			\end{align*} 
			We obtain $\norme{\Delta_n^{\PDW}} \leq \sqrt{s}\,\normz{{\Delta_n^{\PDW}}}$ and $\normz{{\Delta_n^{\PDW}}} \leq 2 \Cbeta$ because of \eqref{beta:PDW} and \eqref{def:beta:stern:S}. In addition by following the proof of Lemma \ref{ratelinftygradient} we get $\normi{\nabla \loss_{n,\alpha_n}^{\,\Hu} (\betaS)} \leq \Cgrads \, \big( \frac{\log(p)}{n}\big)^\frac{1}{2} $ with probability at least $1-p^2/2$. Hence it follows that
			\begin{align*}
				\normz{\widehat{\beta}_n^{\,\PDW} - \betaS} \leq \Bigg( \Cgrads \, \bigg( \frac{\log(p)}{n}\bigg)^\frac{1}{2} + w_{\max}\big(S\big) \, \lambda_n + 2 \Cbeta \,\Crscs \frac{\sqrt{s} \log(p)}{n} \Bigg) \, \frac{\sqrt{s}}{\Crscf}
			\end{align*}
			and in total the assertion of the lemma.
		\end{proof}

		\begin{remark}
			The results below will imply that with (appropriate) choices of $\lambda$ and $\alpha_n$, 
			\begin{align*}
				\normz{\widehat{\beta}_n^{\,\PDW} - \beta^*} = \Oop{\bigg(\frac{s \log(p)}{n}\bigg)^\frac{1}{2}}
			\end{align*}
			with high probability, so that in particular $\big\|  \widehat{\beta}_n^{\,\PDW} - \beta^* \big\|_2 < \Cbeta/2$. Together with the assumption $\normz{ \beta^*} \leq \Cbeta/2$ this will imply that $\widehat{\beta}_n^{\,\PDW}$ is strictly feasible for \eqref{beta:PDW}, that is, 
			\begin{equation}\label{eq:strctfeasible}
				\normz{\widehat{\beta}_n^{\,\PDW}} < \Cbeta,
			\end{equation}
			which we will assume from now on. 	
		\end{remark}

		\begin{proof}[Proof of Lemma \ref{corsignconsweightedhuber}]
			We start by showing that under assumption \eqref{betamin:fanreg} we have $S \subseteq S_{\alpha_n}$. To this end, we estimate
			\begin{align*}
				\big|\beta_{\alpha_n,k}^* \big| &= \big|\beta_{\alpha_n,k}^* -  \beta_{k}^* + \beta_{k}^*\big| \geq \big|\beta_{k}^*\big| - \big| \beta_{\alpha_n,k}^* -  \beta_{k}^* \big| \geq \beta_{\min}^* - \normi{\beta_{\alpha_n}^* -  \beta^*} \\
				& \geq \beta_{\min}^* - \normz{\beta_{\alpha_n}^* - \beta^*}\\
				&\geq \beta_{\min}^* - \Capprox  \, \alpha_n^{\Cpm-1} > 0 \,,
			\end{align*}
			for $k \in S$ where the first inequality in the last line follows from \eqref{approx:error} and the final inequality from \eqref{betamin:fanreg}.
			Now, using 
			\begin{align*}
				\widehat{Q} \, \big(\widehat{\beta}_n^{\,\PDW}-\beta_{\alpha_n}^* \big) = \nabla  \loss_{n,\alpha_n}^{\,\Hu} \big(\widehat{\beta}_n^{\,\PDW}\big) - \nabla  \loss_{n,\alpha_n}^{\,\Hu} \big(\beta_{\alpha_n}^*\big) \,,
			\end{align*}
			see \eqref{defofQhat} for the definition of $\widehat{Q}$, 	we may rewrite the subgradient condition \eqref{eq:zerosubgradpdw}, which holds since $\widehat{\beta}_n^{\,\PDW}$ is strictly feasible as in \eqref{eq:strctfeasible}, as 
			\begin{align*}
				\widehat{Q} \, \big(\widehat{\beta}_n^{\,\PDW}-\beta_{\alpha_n}^* \big) + \nabla  \loss_{n,\alpha_n}^{\,\Hu} \big(\beta_{\alpha_n}^*\big) + \lambda_n \big( w \odot \widehat{\gamma}\big)& = \f{0}_p
			\end{align*}
			or in block-form
			{\small 
				\begin{align*} 
					\begin{bmatrix}
						\widehat{Q}_{S S} & \widehat{Q}_{S (S_{\alpha_n} \setminus S)} & \widehat{Q}_{S S_{\alpha_n}^c} \\
						\widehat{Q}_{S^c S} & \widehat{Q}_{S^c (S_{\alpha_n} \setminus S)} & \widehat{Q}_{S^c S_{\alpha_n}^c}
					\end{bmatrix}
					\begin{pmatrix}
						\widehat{\beta}_{n,S}^{\,\PDW}-\beta_{\alpha_n,S}^*\\
						-\beta_{\alpha_n,S_{\alpha_n} \setminus S}^*\\
						\f{0}_{|S_{\alpha_n}^c|}
					\end{pmatrix} &+
					\begin{pmatrix}
						\Big(\nabla \loss_{n,\alpha_n}^{\,\Hu} \big(\beta_{\alpha_n}^*\big)\Big)_{S}\\
						\Big(\nabla \loss_{n,\alpha_n}^{\,\Hu} \big(\beta_{\alpha_n}^*\big)\Big)_{S^c}
					\end{pmatrix} + \lambda_n
					\begin{pmatrix}
						w_{S} \odot \widehat{\gamma}_{S}\\
						w_{S^c} \odot \widehat{\gamma}_{S^c}
					\end{pmatrix} = \f{0}_p, 
				\end{align*}
			}
			where we used that  $\widehat{\beta}_{n,S^c}^{\,\PDW} = \f0_{p-s} $ by the primal-dual witness construction. By invertibility of $\widehat{Q}_{S S}$, see Lemma \ref{lem:derivativelem}, this leads to
			\begin{align}\label{eq:equaesterrorhelp}
				\widehat{\beta}_{n,S}^{\,\PDW}-\beta_{\alpha_n,S}^* = \big(\widehat{Q}_{S S}\big)^{-1} \bigg( - \lambda_n \big( w_{S} \odot \widehat{\gamma}_{S} \big) -  \Big( \nabla \loss_{n,\alpha_n}^{\,\Hu} \big(\beta_{\alpha_n}^*\big)\Big)_{S} + \widehat{Q}_{S (S_{\alpha_n} \setminus S)} \, \beta_{\alpha_n,S_{\alpha_n} \setminus S}^*\bigg)
			\end{align}
			and
			{\small
				\begin{align*}
					\lambda_n \big( w_{S^c} \odot \widehat{\gamma}_{S^c}\big) &= \widehat{Q}_{S^c S} \big(\widehat{Q}_{S S}\big)^{-1} \bigg( \lambda_n \big( w_{S} \odot \widehat{\gamma}_{S} \big) + \Big( \nabla \loss_{n,\alpha_n}^{\,\Hu} \big(\beta_{\alpha_n}^*\big)\Big)_{S} \bigg) - \Big(\nabla \loss_{n,\alpha_n}^{\,\Hu} \big(\beta_{\alpha_n}^*\big)\Big)_{S^c}  \\
					&~~~~~~~~~~+ \Big( \widehat{Q}_{S^c (S_{\alpha_n} \setminus S)} - \widehat{Q}_{S^c S} \big(\widehat{Q}_{S S}\big)^{-1} \widehat{Q}_{S (S_{\alpha_n} \setminus S)}  \Big) \,\beta_{\alpha_n,S_{\alpha_n} \setminus S}^* \,.
			\end{align*}}
			\noindent
			The second equation shows the equivalence of the strict dual feasibility condition $\normi{\widehat{\gamma}_{S^c}} < 1$ and \eqref{strictfeasibilitywh2}. 
			Now, if this holds then we obtain that $\widehat{\gamma} \in \partial \big\lVert\widehat{\beta}_n^{\,\PDW}\big\rVert_1$, and since the loss function $\loss_{n,\alpha_n}^{\,\Hu}$ is convex (and obviously also the weighted $\ell_1$ norm) we obtain by \eqref{eq:zerosubgradpdw} that $\widehat{\beta}_n^{\,\PDW}$ is also a solution of \eqref{defweightedlassohuber}, cf.~\citet[Theorem 3.33]{Ruszczynski2006} and recall from \eqref{eq:strctfeasible} that $\widehat{\beta}_n^{\,\PDW}$ is (assumed to be) strictly feasible. 
			To conclude $\widehat{\beta}_n^{\,\WLH} = \widehat{\beta}_n^{\,\PDW}$ we need to show that this solution is unique. Then apparently $\supp\big(\widehat{\beta}_n^{\,\WLH}\big) \subseteq \supp\big(\beta^*\big)$ and \eqref{eq:boundesterrorwlgeneralpdw} follows from \eqref{eq:equaesterrorhelp}.
			If the beta-min condition \eqref{betaminweightedhuber} holds, then for $k \in S$
			\begin{align*}
				\Big|\widehat{\beta}_{n,k}^{\,\WLH} - \beta_{k}^*\Big| &\leq \normi{\widehat{\beta}_{n}^{\,\WLH}-\beta^*} < \beta_{\min}^* \leq \big| \beta_{k}^* \big|\,,
			\end{align*}
			which implies $\sign\big( \widehat{\beta}_{n,k}^{\,\WLH} \big) = \sign\big( \beta_{k}^* \big)$ and hence the sign-consistency of $\widehat{\beta}_{n}^{\,\WLH}$.

			It remains to show uniqueness of the solution of the program \eqref{defweightedlassohuber}. 	
			To this end we show that 
			all stationary points $\widetilde{\beta}$, that is points satisfying $\nabla \loss_{n,\alpha_n}^{\,\Hu} \big(\widetilde{\beta} \big) = - \lambda_n \big( w \odot \widetilde{\gamma} \big)$ with $\widetilde{\gamma} \in \partial \lVert \widetilde{\beta} \rVert_1$
			have support $S$ (cf. \citet[Lemma 3]{Loh2017} or \citet[Section 2.3]{Tibshirani2013}). 
			Then strict convexity of the loss function restricted to vectors with support $S$, as implied by \eqref{positvedefinitenesshessian}, concludes the proof. 
			
			From the form \eqref{eq:gradientgeneral} of the gradient of the loss function we see that uniqueness of the fitted values $\X_n \, \widetilde{\beta}$ for all stationary points implies uniqueness of the subgradient $\widetilde{\gamma}$, that is $\widetilde{\gamma} = \widehat{\gamma}$. The strict dual feasibility condition $\normi{\widehat{\gamma}_{S^c}} < 1$ for $\widehat{\gamma}$ (cf. \citet[Section 2.3]{Tibshirani2013} or \citet[Lemma 1 (b)]{Wainwright2009}) then implies that $\widetilde{\beta}$ must also have support in $S$.
			Now, uniqueness of the fitted values follows from the strict convexity of the pseudo Huber loss by using Lemma 1 (ii) in \citet{Tibshirani2013}. This concludes the proof of the lemma. 
		\end{proof}

		\subsection{Support recovery and $\ell_\infty$ - bounds for the weighted LASSO pseudo Huber estimator}\label{sec:resultsweightedlasso}

		In the main result in this section, Lemma \ref{signconsweightedhuber},  similarly to \citet[Lemma 8.2]{Zhou2009} we consider support recovery and bounds for a generic form of the weighted LASSO pseudo Huber estimator. To this end we require technical conditions stated in Lemma \ref{lemmaassstrictdual(ii)}. These show how to take care of the terms involving the gradient of the loss in the strict dual feasibility assumption \eqref{strictfeasibilitywh2}.

		\begin{lemma}[Strict dual feasibility and norm bound I] \label{lemmaassstrictdual(ii)} 
			Suppose that Assumption \ref{assfan2017} and \eqref{lemmaassstrictdual(ii)(1)} are satisfied and assume that the robustification parameter $\alpha_n$ is chosen as in \eqref{rangealphaalhwithinitial}. If  $n \geq C_3\, s^2 \log(p)$ for a sufficiently large positive constant $C_3 >0$ then there exist constants $C_1,C_2,\CQS>0 $ and $\CQograd \geq 1$ such that
			\begin{align}\label{eq:boundinvmat}
				\normiM{\big(\widehat{Q}_{S S}\big)^{-1}} \leq \CQS
			\end{align}
			is satisfied with probability at least $1-C_1/ p^{2} - 6/p^{5s}$, and
			\begin{align}\label{eq:esstrictdualfeascondsecond}
				\normi{\widehat{Q}_{S^c S} \, \big(\widehat{Q}_{S S}\big)^{-1} \Big(\nabla \loss_{n,\alpha_n}^{\,\Hu} \big(\beta_{\alpha_n}^*\big)\Big)_{S} - \Big(\nabla \loss_{n,\alpha_n}^{\,\Hu} \big(\beta_{\alpha_n}^*\big)\Big)_{S^c} } \leq \CQgrad \,\bigg( \frac{\log(p)}{n}\bigg)^\frac{1}{2}
			\end{align}
			with probability at least $1 - (4+C_1+C_2)/ p^{2} - 6/p^{5s}$, where $\Cgrads$ is as in Lemma \ref{ratelinftygradient}.
		\end{lemma}

		The proof is deferred to the supplement, Section A.2.

		For clarity of formulation in the next main result we shall now simply impose \eqref{eq:esstrictdualfeascondsecond}, and  \eqref{eq:boundinvmat} in the second part, as high-level conditions. 
		
		\begin{lemma} \label{signconsweightedhuber} 
			
			Consider model \eqref{generalregressionmodel} under Assumption \ref{assfan2017}. Suppose that \eqref{eq:esstrictdualfeascondsecond}	holds true, and that 
			the weights satisfy the mutual incoherence condition, that is for some $ \eta \in (0,1) $ we have that
			\begin{align}\label{Qmutualincoherence}
				\bigg| \widehat{Q}_{S^c S} \big( \widehat{Q}_{S S} \big)^{-1} \big( w_{S} \odot \widehat{\gamma}_{S} \big) \bigg| \leq w_{S^c} \big(1 - \eta\big) \,.
			\end{align}
			For the regularization parameter $\lambda_n$, assume that 
			\begin{align}
				w_{\min}\big(S^c\big)\,\lambda_n > \frac{4 \, \CQgrad}{\eta} \bigg(\frac{\log(p)}{n}\bigg)^\frac{1}{2}. \label{lambdanwminconwlhcor}
			\end{align}
			Furthermore, suppose that the robustification parameter $\alpha_n$ is chosen in the range
			\begin{align}
				\Cgradf \, \bigg(\frac{\log(p)}{n}\bigg)^{\frac{1}{2}} \leq \alpha_n \leq \Bigg(\frac{  \Cgrads  }{80\,\Capprox\,\CXsubs}\bigg(\frac{\log(p)}{n}\bigg)^\frac{1}{2} \Bigg)^\frac{1}{\Cpm-1}  \label{rangealphawlhcor}
			\end{align}
			
			and
			
			\begin{align}
				\beta_{\min}^* >  \phi_{n,\infty,s}, \quad \text{where} \quad \phi_{n,\infty,s} = \frac{128}{\CXl}  \, \max \bigg\{ \Cgrads\, \bigg(\frac{s \log(p)}{n} \bigg)^{\frac{1}{2}} \,, \, w_{\max}\big(S\big)\,\lambda_n\,\sqrt{s} \bigg\} \,. \label{betaminconditionwlhcor}
			\end{align}
			
			Then for $n \geq \max\big\{\Crsct s \log(p), 6 \log(p) \big\}$ with probability at least 
			\begin{align}\label{eq:boundprob}
				1-\cPo \exp(-\cPt n) - 2 \exp(- 2n) - \frac{4}{p^2}, 
			\end{align}
			the weighted LASSO pseudo Huber estimator, as a solution to the program \eqref{defweightedlassohuber}, is unique and given by $\widehat{\beta}_n^{\,\WLH} = \widehat{\beta}_n^{\,\PDW}$ and satisfies 
			\begin{align}
				\sign\big(\widehat{\beta}_{n}^{\,\WLH}\big) =  \sign\big(\beta^*\big) ~~~~~~\mathrm{and} ~~~~~~ \normi{\widehat{\beta}_{n}^{\,\WLH} - \beta^*} \leq \phi_{n,\infty,s} . \label{wlhaussage} 
			\end{align}
			with $\phi_{n,\infty,s}$ in \eqref{betaminconditionwlhcor}. 
			
			If in addition \eqref{eq:boundinvmat} is satisfied, we may replace $\phi_{n,\infty,s}$ in the beta-min condition \eqref{betaminconditionwlhcor} and in the $\ell_\infty$ bound in \eqref{wlhaussage} by 
			\begin{align}
				\phi_{n,\infty,f} =  4 \, \CQS \,\max \bigg\{ \Cgrads \, \bigg(\frac{\log(p)}{n} \bigg)^{\frac{1}{2}} \, , \, w_{\max}\big(S\big)\,\lambda_n \bigg\} \,.	\label{betaminconditionwlh2cor}
			\end{align}
			
		\end{lemma}

		If we drop the requirement \eqref{lemmaassstrictdual(ii)(1)} we still obtain a bound of the form 
		\eqref{eq:esstrictdualfeascondsecond} under the somewhat restrictive scaling $s \leq \log(p)$. The bound  \eqref{eq:boundinvmat} is no longer valid and needs to be replaced by \eqref{inftyoperatornormQss}.
		\begin{lemma}[Strict dual feasibility and norm bound II] \label{lemmaassstrictdual(i)} 
			Suppose that Assumption \ref{assfan2017} holds and assume that the robustification parameter satisfies $\alpha_n \geq \sqrt{4/3} \,\Cgradf \, \big(\frac{\log(p)}{n}\big)^{\frac{1}{2}}$, where $\Cgradf$ is as in Lemma \ref{ratelinftygradient}.
			%
			%
			Then for $s \leq \log(p)$ and $n \geq \max\big\{ \Crsct s \log(p), 6 \log(p) \big\}$ we still have \eqref{eq:esstrictdualfeascondsecond}
			
			with probability at least $1-\cPo \exp(-\cPt n) - 6/p^2$.

		\end{lemma}

		The proof is provided in the supplement, Section A.3.

		\begin{proof}[Proof of Lemma \ref{signconsweightedhuber}]
			We shall apply Lemma \ref{corsignconsweightedhuber}. Using  the mutual incoherence condition \eqref{Qmutualincoherence}, in order to show strict dual feasibility as in \eqref{strictfeasibilitywh2} it suffices to prove that 
			\begin{align}
				&\normi{ \widehat{Q}_{S^c S} \big(\widehat{Q}_{S S}\big)^{-1} \Big( \nabla \loss_{n,\alpha_n}^{\,\Hu} \big(\beta_{\alpha_n}^*\big)\Big)_{S}  - \Big(\nabla \loss_{n,\alpha_n}^{\,\Hu} \big(\beta_{\alpha_n}^*\big)\Big)_{S^c} } \notag \\
				&\quad \quad \quad \quad \quad \quad + \normi{ \Big( \widehat{Q}_{S^c (S_{\alpha_n} \setminus S)} - \widehat{Q}_{S^c S} \big(\widehat{Q}_{S S}\big)^{-1} \widehat{Q}_{S (S_{\alpha_n} \setminus S)}  \Big) \,\beta_{\alpha_n,S_{\alpha_n} \setminus S}^* }  < \frac{w_{\min}\big(S^c\big)\,\eta}{2}\, \lambda_n \label{proofsignconsweightedhuber1}\,.
			\end{align}
			The first term is bounded by \eqref{eq:esstrictdualfeascondsecond} (which is satisfied by assumption). We prove in the supplement, Section A.4 that
			\begin{align}
				\normi{ \Big( \widehat{Q}_{S^c (S_{\alpha_n} \setminus S)} - \widehat{Q}_{S^c S} \big(\widehat{Q}_{S S}\big)^{-1} \widehat{Q}_{S (S_{\alpha_n} \setminus S)}  \Big) \,\beta_{\alpha_n,S_{\alpha_n} \setminus S}^* }  \leq 80 \, \Capprox\,\CXsubs\,\alpha_n^{\Cpm-1} \label{proofsignconsweightedhuber2}
			\end{align}
			with probability at least $1-2\exp(-2n)-2/p^2$. Then the choices of $\lambda_n$ and $\alpha_n$  in \eqref{lambdanwminconwlhcor} and in \eqref{rangealphawlhcor} imply \eqref{proofsignconsweightedhuber1}. 
			Since the first beta-min condition in Lemma \ref{corsignconsweightedhuber} is also satisfied in both cases by choice of $\alpha_n$ in \eqref{rangealphawlhcor}, the first part of that lemma up to \eqref{eq:boundesterrorwlgeneralpdw} applies. 
			Here we assumed that $\sqrt{s}\geq \CXl/(2560 \, \CXsubs)$ for \eqref{betaminconditionwlhcor} and $320 \, \CXsubs \, \CQS \geq 1$ for \eqref{betaminconditionwlh2cor}, which can be arranged by choosing the constants appropriately.

			Now we show that $\phi_{n, \infty}$ in \eqref{eq:boundesterrorwlgeneralpdw} is bounded by $\phi_{n, \infty,s}$ and, under the additional condition \eqref{eq:boundinvmat} is even bounded by $\phi_{n, \infty,f}$. Then \eqref{betaminconditionwlhcor} (or the analogous condition with $\phi_{n, \infty,f}$) 
			implies the beta-min condition \eqref{betaminweightedhuber} in Lemma \ref{corsignconsweightedhuber}, which concludes the proof. 
			
			To this end, note that $\phi_{n, \infty}$ is bounded by four times the maximum of the summands in \eqref{eq:boundesterrorwlgeneralpdw}. In addition \eqref{inftyoperatornormQss} leads to
			\begin{align*}
				4 \,w_{\max}\big(S\big) \,\lambda_n \, \normiM{\big(\widehat{Q}_{S S}\big)^{-1} } \leq \frac{128 \,w_{\max}\big(S\big) \, \lambda_n \,\sqrt{s}}{\CXl}
			\end{align*}
			and together with Lemma \ref{ratelinftygradient} and the lower bound of $\alpha_n$ in \eqref{rangealphawlhcor} leads to
			\begin{align*}
				4 \, \normiM{\big(\widehat{Q}_{S S}\big)^{-1}} \, \normi{ \Big( \nabla \loss_{n,\alpha_n}^{\,\Hu} \big(\beta_{\alpha_n}^*\big)\Big)_{S} } \leq \frac{128 \, \Cgrads}{\CXl} \bigg( \frac{s \log(p)}{n}\bigg)^\frac{1}{2}
			\end{align*}
			with probability at least $1 - \cPo \exp(-\cPt n) - 2 / p^{2} $ . Further, Lemma \ref{approximation:error:pseudo:huber} implies 
			\begin{align*}
				4 \, \normi{\beta_{\alpha_n,S}^* - \beta_{S}^*} \leq 4 \, \normz{\beta_{\alpha_n}^* - \beta^*} & \leq 4 \, \Capprox  \, \alpha_n^{\Cpm-1} \leq \frac{128 \,\Cgrads}{\CXl} \Bigg(\frac{s \log(p)}{n} \Bigg)^{\frac{1}{2}}
			\end{align*}
			with the choice of $\alpha_n$ in \eqref{rangealphawlhcor}. Finally, in the supplement, Section \ref{sec:proofsignconsweightedhuber2}, we also show that 
			\begin{equation}\label{secondsupplequweigthuber}
				\normi{ \widehat{Q}_{S (S_\alpha \setminus S)} \,\beta_{\alpha,S_\alpha \setminus S}^*} \leq 80 \, \Capprox\,\CXsubs\,\alpha_n^{\Cpm-1}
			\end{equation}
			with high probability. Together with \eqref{inftyoperatornormQss} these imply 
			\begin{align*}
				4\,\normiM{\big(\widehat{Q}_{S S}\big)^{-1}} \, \normi{ \widehat{Q}_{S (S_\alpha \setminus S)} \,\beta_{\alpha,S_\alpha \setminus S}^*} 
				\leq \frac{10240\, \Capprox\,\CXsubs}{\CXl} \,\sqrt{s}\,\alpha_n^{\Cpm-1} \leq \frac{128 \,\Cgrads}{\CXl} \Bigg(\frac{s \log(p)}{n} \Bigg)^{\frac{1}{2}}
			\end{align*}
			by the choice of $\alpha_n$, which concludes the proof of $\phi_{n,\infty} \leq \phi_{n,\infty,s}$.

			To show $\phi_{n,\infty} \leq \phi_{n,\infty,f}$ under the assumption \eqref{eq:boundinvmat}, after arranging $80 \, \CXsubs \, \CQS \geq 1$ we proceed analogously (and use the estimate \eqref{eq:boundinvmat} instead of \eqref{inftyoperatornormQss} in the previous inequalities). This concludes the proof of the lemma. 

		\end{proof}

		\subsection{The adaptive LASSO pseudo Huber estimator with $\ell_\infty$ - bounds for the first-stage estimator }\label{sec:resultsadaplassogeneral}

		Now in Lemma \ref{signconsadaptivehuber} we provide a result on the estimation error of the adaptive LASSO pseudo Huber estimator $\widehat{\beta}_n^{\,\ALH}$ in \eqref{defweightedlassohuber} with weights in \eqref{eq:weightsadaplassoinitest} for an initial estimator which is assumed to satisfy bounds in $\ell_\infty$, similarly to \citet[Theorem 4.3]{Zhou2009}. Together with Lemma \ref{lemmaassstrictdual(ii)} and Lemma \ref{rateofrnS} to deal with the mutual incoherence condition \eqref{rnconditionalh} this implies Theorem \ref{signconsadaptivehuberwithinitialellinf}. 

		\begin{lemma}[Adaptive LASSO pseudo Huber estimator with $\ell_\infty$ - bounds for the first-stage estimator] \label{signconsadaptivehuber}
			
			Consider model \eqref{generalregressionmodel} under Assumption \ref{assfan2017}. Suppose that \eqref{eq:esstrictdualfeascondsecond}	holds 
			and that $\alpha_n$ is chosen according to \eqref{rangealphawlhcor}. 
			For the estimation error 
			\begin{align}
				\Delta_n \defeq \widehat{\beta}_n^{\,\init} - \beta^*\, \label{defdeltaninitial}
			\end{align}
			of the initial estimator $\widehat{\beta}_n^{\,\init}$, assume that we have upper bounds of the form
			\begin{equation}\label{eq:reateinitialest}
				\normi{\Delta_{n,S}} \leq a_n < 1, \qquad \normi{\Delta_{n,S^c}} \leq b_n < 1,
			\end{equation}
			with sequences $(a_n)$ and $(b_n)$ tending to zero. 
			Furthermore, assume that
			for some $\eta \in (0,1)$ and $\Clambda > 4/\eta$ the 
			regularization parameter is chosen from the range
			\begin{align}
				\frac{4 \, \CQgrad \, b_n}{\eta} \bigg(\frac{\log(p)}{n}\bigg)^\frac{1}{2} < \lambda_{n} \leq \Clambda\,\CQgrad\, b_n \bigg(\frac{\log(p)}{n}\bigg)^\frac{1}{2} \, , \label{lambdanconalh}
			\end{align}
			and in addition suppose that there is a sequence $q_n \leq (1-\eta)/b_n$, which may grow if $b_n \downarrow 0$, such that 
			\begin{equation}
				\normiM{\widehat{Q}_{S^c S} \big( \widehat{Q}_{S S} \big)^{-1}} \leq q_n \label{rnconditionalh}\,.
			\end{equation}
			Finally, setting
			\[ \phi_{n, \infty, s,1} = \frac{128}{\CXl}  \, \max \bigg\{ \Cgrads\, \bigg(\frac{s \log(p)}{n} \bigg)^{\frac{1}{2}} \,, \,\lambda_n\,\sqrt{s} \bigg\},\]
			suppose that the beta-min assumption
			\begin{equation}\label{betaminconditionalh}
				\beta_{\min}^* >  \max \bigg\{ 2 \, a_n \, , \, \phi_{n, \infty, s,1} \, , \, 2 \max\Big\{ \frac{ q_n}{1-\eta} \, ,  \Clambda\,\CQograd \Big\} \, b_n  \bigg\} \,
			\end{equation}
			is satisfied. 
			Then for $n \geq \max\big\{\Crsct s \log(p), 6 \log(p) \big\}$ with probability at least equal to \eqref{eq:boundprob}, 
			%
			the adaptive LASSO pseudo Huber estimator, given as a solution to the program \eqref{defweightedlassohuber} with weights in \eqref{eq:weightsadaplassoinitest}, is unique and given by $\widehat{\beta}_n^{\,\ALH} = \widehat{\beta}_n^{\,\PDW}$ and satisfies 
			\begin{align}
				\sign\big(\widehat{\beta}_{n}^{\,\ALH}\big) =  \sign\big(\beta^*\big) ~~~~~~\mathrm{and} ~~~~~~ \normi{\widehat{\beta}_{n}^{\,\ALH} - \beta^*} \leq \phi_{n, \infty, s,1} \, . \label{alhaussage} 
			\end{align}
			If in addition \eqref{eq:boundinvmat} is also assumed, we can replace $\phi_{n, \infty, s,1}$ in the beta-min condition \eqref{betaminconditionalh} and in the upper bound of the $\ell_\infty$ distance of the estimation error by
			\begin{align}
				\phi_{n, \infty, f,1} = %
				4 \, \CQS \,\max \bigg\{ \Cgrads \, \bigg(\frac{\log(p)}{n} \bigg)^{\frac{1}{2}} \, , \, \lambda_n \bigg\} \,.\label{betaminconditionalh2}
			\end{align}
		\end{lemma}
		
		\begin{proof}[Proof of Lemma  \ref{signconsadaptivehuber}]
			We shall apply Lemma \ref{signconsweightedhuber}. To this end, we start by checking the mutual incoherence condition \eqref{Qmutualincoherence} and the condition \eqref{lambdanwminconwlhcor} on the regularization parameter $\lambda_n$.
			For \eqref{lambdanwminconwlhcor}, since 
				$\big|\widehat{\beta}_{n,k}^{\,\init} \big| = \big|\beta_{k}^* + \Delta_{n,k}\big| = \big|\Delta_{n,k}\big| \leq \big\lVert\Delta_{n,S^c}\big\rVert_\infty$ 
				for $k \in S^c$, from \eqref{eq:reateinitialest} we obtain that
				\begin{align}
					w_{\min}\big(S^c\big) = \min_{k \in S^c} \Big\{ \max \big\{ \big(\big| \widehat{\beta}_{n,k}^{\,\init} \big|\big)^{-1} , 1 \big\} \Big\} \geq \big\lVert\Delta_{n,S^c}\big\rVert_\infty^{-1} \geq \frac{1}{b_n} \,  \label{wminalh}	\end{align}
				and hence 	$w_{\min} \big(S^c\big) \,\lambda_n  \geq \lambda_{n}/b_n $, which together with the assumption \eqref{lambdanconalh} on $\lambda_n$ gives \eqref{lambdanwminconwlhcor}.	
				Next, we turn to the mutual incoherence condition \eqref{Qmutualincoherence}, for which it suffices to prove
				\begin{align}\label{eq:mutualincohtwo}
					\normiM{\widehat{Q}_{S^c S} \big( \widehat{Q}_{S S} \big)^{-1}} \leq \frac{w_{\min}\big(S^c\big)}{w_{\max}\big(S\big)} \big(1 - \eta\big)\,.
				\end{align}

				From the beta-min condition \eqref{betaminconditionalh} and the bounds in \eqref{eq:reateinitialest} we have in particular that 
					$\beta_{\min}^*/2 > a_n \geq \normi{\Delta_{n,S}} \geq \big|\Delta_{n,k}\big|$
				and hence that
				\begin{align*}
					\big|\widehat{\beta}_{n,k}^{\,\init} \big| = \big|\beta_{k}^* + \Delta_{n,k} \big| \geq \big|\beta_{k}^*\big| - \big| \Delta_{n,k}\big| > \beta_{\min}^* - \frac{\beta_{\min}^*}{2} = \frac{\beta_{\min}^*}{2} 
				\end{align*}
				for $k \in S$. This together with the definition of the weights implies
				\begin{align}
					w_{\max}\big(S\big)  = \max_{k \in S} \Big\{ \max \big\{ \big(\big| \widehat{\beta}_{n,k}^{\,\init} \big|\big)^{-1} , 1 \big\} \Big\} \leq \max \{ 2/\beta_{\min}^* , 1 \} \, . \label{wmaxalh}
				\end{align}
				In order to conclude \eqref{eq:mutualincohtwo} we now consider two cases. 
				If $\beta_{\min}^*\leq 2$, then we have $w_{\max}\big(S\big) \leq 2 / \beta_{\min}^*$ because of \eqref{wmaxalh} and hence with \eqref{wminalh} and the last term in the beta-min condition \eqref{betaminconditionalh} we obtain
				\begin{align*}
					\frac{w_{\min}\big(S^c\big)}{w_{\max}\big(S\big)} \big(1 - \eta\big) \geq \frac{\beta_{\min}^* \, \big(1 - \eta\big)}{2\,b_n} > q_n \geq  \normiM{\widehat{Q}_{S^c S} \big( \widehat{Q}_{S S} \big)^{-1}} \,
				\end{align*}
				by \eqref{rnconditionalh}. 
				If $\beta_{\min}^*> 2$, then $w_{\max}\big(S\big) \leq 1$ and by \eqref{rnconditionalh}, \eqref{wminalh} and the choice of $q_n$ it follow that
				\begin{align*}
					\frac{w_{\min}\big(S^c\big)}{w_{\max}\big(S\big)} \big(1 - \eta\big) \geq \frac{1 - \eta}{b_n} \geq q_n = \normiM{\widehat{Q}_{S^c S} \big( \widehat{Q}_{S S} \big)^{-1}} \,, 
				\end{align*}
				so that \eqref{eq:mutualincohtwo} is satisfied in both cases. 
				
				Next we show that $\phi_{n,\infty, s} \leq \phi_{n,\infty, s, 1}$, then the beta-min condition \eqref{betaminconditionalh} directly implies \eqref{betaminconditionwlhcor}.
				Comparing $\phi_{n,\infty, s, 1}$ and $\phi_{n,\infty, s}$ it remains to show that  
				\begin{align}
					\frac{128 \, w_{\max}\big(S\big)\,\lambda_n\,\sqrt{s} }{\CXl} \leq \frac{128}{\CXl}  \, \max \bigg\{ \Cgrads\, \bigg(\frac{s \log(p)}{n} \bigg)^{\frac{1}{2}} \,, \,\lambda_n\,\sqrt{s} \bigg\}. \label{eq:theboundbetaminagain}
				\end{align} 
				To this end, note that 
				%
				the last lower bound in the inequality \eqref{betaminconditionalh} implies
				\begin{align*}
					\frac{128 \, \Cgrads}{\CXl} \bigg(\frac{s \log(p)}{n} \bigg)^{\frac{1}{2}}  > \frac{128 \, \Cgrads}{\CXl} \bigg(\frac{s \log(p)}{n} \bigg)^{\frac{1}{2}} \,\frac{ 2 \,\Clambda\,\CQograd\, b_n}{ \beta_{\min}^*} \geq \frac{256\,\lambda_n \sqrt{s}}{\CXl\,\beta_{\min}^*} 
				\end{align*}
				
				by the choice of the regularization parameter $\lambda_n$ in \eqref{lambdanconalh}. This together with \eqref{wmaxalh} implies \eqref{eq:theboundbetaminagain}. 	
				So Lemma \ref{signconsweightedhuber} applies and we conclude that the $\ell_\infty$ bound in \eqref{wlhaussage}  can be reduced to \eqref{alhaussage}.
				
				For the sharper bound $ \phi_{n,\infty, f} \leq \phi_{n,\infty, f, 1}$, 
				under assumption \eqref{eq:boundinvmat} one argues similarly. This concludes the proof. 
			\end{proof}
			%
			%
			
			\subsection{The adaptive LASSO pseudo Huber estimator with $\ell_2$ - bounds for the first-stage estimator}\label{sec:resultsadaplassospecial}

			Next, we discuss the performance of the LASSO pseudo Huber estimator in case that the initial estimator satisfies bounds on the $\ell_2$ - and $\ell_1$ - errors. In this case, the regularization parameter $\lambda_n$ for the adaptive LASSO pseudo Huber estimator requires the parameter $s$. Therefore we first present Lemma \ref{thresholdingprocedureinitial}, which is analogous to \citet[Lemma 4.2]{Zhou2009} and which gives a superset $\overline{S}$ of the support $S$, the cardinality of which is of the same order $s$. 
			We set
			$$\lambda_n^{\,\init} = \Big(\frac{\log p}{n}\Big)^{1/2},$$
			the notation pointing to the amount of regularization in the initial estimator. 
			\begin{lemma}[Thresholding procedure] \label{thresholdingprocedureinitial}
				If the initial estimator $\widehat{\beta}_n^{\,\init}$ satisfies \eqref{eq:conrateinitial}, and if the following beta-min condition 
				\begin{align}
					\beta_{\min}^* > 2 \, \Cinit \, \lambda_n^{\,\init} \sqrt{s} \label{thresholdingprocedureinitialbetamin}
				\end{align}
				holds, then the set
					$	\overline{S} = \big\{ k \in \{1,\dotsc,p\} \, \big|\, \big|\widehat{\beta}_{n,k}^{\,\init}\big| > \lambda_n^{\,\init} \big\} \,$
				satisfies
				\begin{align}
					S \subseteq \overline{S} ~~~~~\mathrm{and}~~~~~ s \leq \big|\overline{S}\big| \leq  2 \, \Cinit \, s \,.
				\end{align}
			\end{lemma}  
			
			\begin{proof}[Proof of Lemma \ref{thresholdingprocedureinitial}]\hfill\\
				Let $\Delta_n = \widehat{\beta}_n^{\,\init} - \beta^*$. Then from \eqref{eq:conrateinitial} it follows that 
				\begin{align*}
					\normi{\Delta_{n,S}} \leq \normi{\Delta_{n}} \leq \normz{\Delta_{n}} \leq \Cinit \, \lambda_n^{\,\init} \sqrt{s}
				\end{align*}
				and hence for all $k \in S$ that
				\begin{align*}
					\big| \widehat{\beta}_{n,k}^{\,\init} \big| = \big|\beta_{k}^* + \Delta_{n,k} \big| \geq \big|\beta_{k}^*\big| - \big| \Delta_{n,k} \big| \geq \beta_{\min}^* - \normi{\Delta_{n,S}} > \Cinit \, \lambda_n^{\,\init} \sqrt{s} \,
				\end{align*}
				because of inequality \eqref{thresholdingprocedureinitialbetamin}. In consequence, the definition of the set $\overline{S}$ implies the membership $S \subseteq \overline{S}$. Furthermore, for $k \in S^c$ (since $\beta_{k}^* = 0$) it is
				\begin{align*}
					\big|\widehat{\beta}_{n,k}^{\,\init} \big| = \big|\beta_{k}^* + \Delta_{n,k}\big| = \big|\Delta_{n,k}\big| 
				\end{align*}
				and the upper bound of the $\ell_1$ norm of the estimation error in \eqref{eq:conrateinitial} leads to
				\begin{align*}
					\norme{\widehat{\beta}_{n,S^c}^{\,\init}} = \norme{\Delta_{n,S^c}} \leq \norme{\Delta_n} \leq \Cinit \, \lambda_n^{\,\init} \, s \,.
				\end{align*}
				Hence we include at most $ \Cinit \, s $ more entries from $S^c$ in $\overline{S}$, thus
				\begin{align*}
					s \leq \big|\overline{S}\big| \leq s + \Cinit \, s \leq 2 \,\Cinit \, s \,,
				\end{align*}
				which completes the proof.
			\end{proof}
			\begin{lemma} \label{rateofrnS}
				Suppose Assumption \ref{assfan2017} and $n \geq \max\big\{ \Crsct s \log(p), 6 \log(p) \big\}$ hold. Then 
				\begin{align}\label{eq:rateconvQQ}
					\max_{k \in \{1\dotsc,p-s\}} \normz{ \Big(e_k^\top \widehat{Q}_{S^c S} \big( \widehat{Q}_{S S} \big)^{-1} \Big)^\top} \leq \frac{33\,\CXsub}{\sqrt{\CXl}} \quad \text{and} \quad \normiM{\widehat{Q}_{S^c S} \big( \widehat{Q}_{S S} \big)^{-1}} \leq \frac{33\,\CXsub\,\sqrt{s}}{\sqrt{\CXl}}
				\end{align}
				with probability at least $1 - \cPo \exp(-\cPt n) - 2/p^2$.
			\end{lemma}
			The technical proof of this lemma is deferred to the supplement, Section A.3.

			\smallskip
			
			For clarity of formulation in the following result we shall again impose \eqref{eq:esstrictdualfeascondsecond}, and  \eqref{eq:boundinvmat} in the second part, as high-level conditions. Theorem \ref{signconsadaptivehuberwithinitial} then follows from the following Lemma \ref{signconsadaptivehuberwithinitial1} together with Lemmas \ref{lemmaassstrictdual(ii)} and \ref{lemmaassstrictdual(i)}.

			\begin{lemma}[Adaptive LASSO pseudo Huber under $\ell_2$ - bound for first stage estimator] \label{signconsadaptivehuberwithinitial1}
				Consider model \eqref{generalregressionmodel} under Assumption \ref{assfan2017}. Suppose that \eqref{eq:esstrictdualfeascondsecond}	holds true, and that the robustification parameter $\alpha_n$ is chosen according to \eqref{rangealphawlhcor}.
				Suppose that the initial estimator $\widehat{\beta}_{n}^{\,\init}$ satisfies \eqref{eq:conrateinitial} with $\lambda_n^{\,\init} = \Clambdainit \,\big(\frac{\log(p)}{n}\big)^\frac{1}{2}$ for some constant $\Clambdainit \geq 16\, \Cgrads/\CXl$, and that for suitable 
				$\eta \in (0,1)$ and $\ClambdaL > 4\,(2 \, \Cinit)^\frac{1}{2}/\eta $
				the regularization parameter is chosen from the range 
				\begin{align}
					\frac{4 \, \CQgrad \, \Cinit\,\lambda_{n}^{\,\init}}{\eta} \bigg(\frac{\big|\overline{S}\big| \log(p)}{n}\bigg)^\frac{1}{2}  < \lambda_n \leq \ClambdaL\, \CQgrad \,\lambda_{n}^{\,\init} \,  \bigg(\frac{\Cinit\,\big|\overline{S}\big| \log(p)}{2 \, n}\bigg)^\frac{1}{2}  \label{lambdanalhwithinitial1}
				\end{align}
				with $\overline{S} = \big\{ k \in \{1,\dotsc,p\} \, \big| \, \big|\widehat{\beta}_{n,k}^{\,\init}\big| > \lambda_n^{\,\init} \big\}$ as above. In addition suppose that the sample size satisfies
				\begin{gather}
					n \geq \max \Bigg\{ \bigg(\frac{33\,\CXsub \, \Cinit\,\Clambdainit}{(1-\eta) \sqrt{\CXl} }\bigg)^2 s^2 \log(p) \,,\, \max \bigg\{ \Crsct \, , \, \bigg(\frac{64\, \Cgrads }{\CXl}\bigg)^2 \bigg\} s \log(p) \, , \, 6 \log(p) \Bigg\} \,, \label{ncondalhwithinitial}
				\end{gather}

				and that we have the beta-min condition 
				\begin{gather}
					\beta_{\min}^* > 2 \max\bigg\{ \frac{33\,\CXsub\,\sqrt{s}}{\sqrt{\CXl}\, (1-\eta)}, \ClambdaL \, \CQograd\bigg\} \Cinit\,\lambda_n^{\,\init} \sqrt{s}\,.  \label{betaminconditionalhwithinitial}
				\end{gather}
					Then with probability at least 
					\begin{align*}
						1-\cPo \exp(-\cPt n) - 2 \exp(- 2n) - \frac{4}{p^2} 
					\end{align*}
					the adaptive LASSO pseudo Huber estimator, given as a solution to the program \eqref{defweightedlassohuber} with weights in \eqref{eq:weightsadaplassoinitest}, is unique and given by $\widehat{\beta}_n^{\,\ALH} = \widehat{\beta}_n^{\,\PDW}$ and satisfies 
					\begin{align}
						\sign\big(\widehat{\beta}_{n}^{\,\ALH}\big) =  \sign\big(\beta^*\big) ~~~~~~\mathrm{and} ~~~~~~ \normi{\widehat{\beta}_{n}^{\,\ALH} - \beta^*} \leq 2 \, \ClambdaL \, \CQograd\, \Cinit \,\lambda_n^{\,\init} \sqrt{s} \, . \label{alhwithinitialaussage(1)} 
					\end{align} 
					If in addition \eqref{eq:boundinvmat} is also assumed, the upper bound of the $\ell_\infty$ distance of the estimation error reduces to 
					\begin{align}
						\normi{\widehat{\beta}_{n}^{\,\ALH} - \beta^*} \leq \max\bigg\{  \frac{4\,\CQS\,\Cgrads}{\Clambdainit} \, , \, \frac{\ClambdaL\, \CQS\,\CQograd \, \Cinit \,\CXl }{16 } \bigg\} \, \lambda_n^{\,\init} \,. \label{betaminconditionalhwithinitial2}
					\end{align}
				\end{lemma}  

				\begin{proof}[Proof of Lemma \ref{signconsadaptivehuberwithinitial1}] 
					We shall apply Lemma \ref{signconsadaptivehuber}. To check the assumptions, for \eqref{eq:reateinitialest} using \eqref{eq:conrateinitial} we get 
					$ \normi{\Delta_{n}} \leq  \normz{\Delta_{n}} \leq \Cinit\,\lambda_n^{\,\init} \sqrt{s} = a_n = b_n$.
					For the lower bound in \eqref{lambdanconalh}, using \eqref{lambdanalhwithinitial1}, Lemma \ref{thresholdingprocedureinitial} and the choice of $b_n$ we estimate
					\begin{align*}
						\lambda_{n} & > 
						\frac{4 \, \CQgrad \, \Cinit\,\lambda_{n}^{\,\init}}{\eta} \bigg(\frac{\big|\overline{S}\big| \log(p)}{n}\bigg)^\frac{1}{2} \geq \frac{4 \, \CQgrad \, \Cinit\,\lambda_{n}^{\,\init}}{\eta} \bigg(\frac{s \log(p)}{n}\bigg)^\frac{1}{2} \\
						& = \frac{4 \, \CQgrad \, b_n}{\eta} \bigg(\frac{\log(p)}{n}\bigg)^\frac{1}{2},
					\end{align*}
					and similarly for the upper bound
					\begin{align*}
						\lambda_{n} 
						\leq \ClambdaL\, \CQgrad \, \Cinit \,\lambda_{n}^{\,\init} \,  \bigg(\frac{ s \log(p)}{ n}\bigg)^\frac{1}{2} =  \ClambdaL\, \CQgrad \,b_n \,  \bigg(\frac{\log(p)}{ n}\bigg)^\frac{1}{2} \,.
					\end{align*}
					with $\ClambdaL > 4/\eta$.
					%
					Next, \eqref{rnconditionalh} follows from Lemma \ref{rateofrnS} with $q_n = 33\,\CXsub\,\sqrt{s}/\sqrt{\CXl}$ with high probability. In addition the choice of $b_n$ and the lower bound \eqref{ncondalhwithinitial} of the sample size implies
					\begin{align*}
						q_n \leq \frac{33\,\CXsub}{\sqrt{\CXl}} \, \frac{(1-\eta)\,\sqrt{\CXl}}{33\,\CXsub\,\Cinit\,\Clambdainit\,(s\log(p)/n)^\frac{1}{2}} = \frac{1-\eta}{b_n} \,.
					\end{align*}
					So finally we have to check the beta-min condition in \eqref{betaminconditionalh}, which concludes the proof of the lemma in this setting. The last term in the maximum is given by \eqref{betaminconditionalhwithinitial} and the choice of $b_n$ and $q_n$, and $\beta_{\min}^* \geq 2 \, a_n$ is clear because of the choice of $a_n$ and \eqref{betaminconditionalhwithinitial}. 
					Hence for applying Lemma \ref{signconsadaptivehuber} it remains to show that
					\begin{align*}
						\phi_{n, \infty, s,1} \leq 2 \, \ClambdaL \, \CQograd\, \Cinit \,\lambda_n^{\,\init} \sqrt{s} \,.
					\end{align*}
					This bound implies then also \eqref{alhwithinitialaussage(1)} because of \eqref{alhaussage}. It is 
					\begin{align*}
						& \frac{128 \, \Cgrads}{\CXl}  \, \bigg(\frac{s \log(p)}{n} \bigg)^{\frac{1}{2}} = \frac{128 \, \Cgrads}{\CXl \,\Clambdainit}  \, \lambda_n^{\,\init} \sqrt{s}\\
						& \qquad \leq  \, 8\, \CQograd\, \Cinit \,\lambda_n^{\,\init} \sqrt{s} \leq 2 \, \ClambdaL \, \CQograd\, \Cinit \,\lambda_n^{\,\init} \sqrt{s}
					\end{align*}
					since $16\, \Cgrads \leq \CXl \,\Clambdainit\, \CQograd\, \Cinit$.
					Moreover, \eqref{lambdanalhwithinitial1} and \eqref{ncondalhwithinitial} together with Lemma \ref{thresholdingprocedureinitial} lead to
					\begin{align*}
						\frac{128 \,\lambda_n\,\sqrt{s} }{\CXl} &\leq \ \frac{128\,\ClambdaL\, \CQgrad \,\lambda_{n}^{\,\init}}{\CXl} \,  \bigg(\frac{\Cinit\,\big|\overline{S}\big| \log(p)}{2 \, n}\bigg)^\frac{1}{2} \,\frac{\CXl}{64 \, \Cgrads} \bigg(\frac{n}{\log(p)}\bigg)^\frac{1}{2}\\
						& \leq \ 2 \, \ClambdaL \, \CQograd\, \Cinit \,\lambda_n^{\,\init} \sqrt{s} \,.
					\end{align*}
					Under the stronger assumption \eqref{eq:boundinvmat} we show
					\begin{align*}
						\phi_{n, \infty, f,1} \leq \max\bigg\{  \frac{4\,\CQS\,\Cgrads}{\Clambdainit} \, , \, \frac{\ClambdaL\, \CQS\,\CQograd \, \Cinit \,\CXl }{16 } \bigg\} \, \lambda_n^{\,\init}\,,
					\end{align*}
					which implies \eqref{betaminconditionalhwithinitial2} because of \eqref{betaminconditionalh2}. Note that the upper bound is obviously also smaller than the right term in \eqref{betaminconditionalhwithinitial}. It is easy to see that
					\begin{align*}
						4 \, \CQS \, \Cgrads \, \bigg(\frac{\log(p)}{n} \bigg)^{\frac{1}{2}} = \frac{4 \, \CQS \, \Cgrads}{\Clambdainit} \, \lambda_n^{\,\init}
					\end{align*}
					and
					\begin{align*}
						4 \, \CQS \, \lambda_n \leq 4 \, \ClambdaL\,\CQS\, \CQograd \,\Cinit\,\lambda_{n}^{\,\init} \,  \bigg(\frac{\log(p)}{n}\bigg)^\frac{1}{2} \frac{\CXl}{64 } \bigg(\frac{n}{\log(p)}\bigg)^\frac{1}{2} = \frac{\ClambdaL\,\CQS\, \CQograd \,\Cinit\,\CXl}{16 } \,\lambda_{n}^{\,\init}
					\end{align*}
					by \eqref{lambdanalhwithinitial1}, \eqref{ncondalhwithinitial} and Lemma \ref{thresholdingprocedureinitial}, which concludes the proof. 
				\end{proof}

			}

			\bigskip
			
			\textit{Address for correspondence:}
			
			Prof.~Hajo Holzmann
			
			Philipps-Universität Marburg
			
			Department of Mathematics and Computer Science
			
			Hans-Meerweinstr. 6
			
			35043 Marburg
			
			Email: holzmann@mathematik.uni-marburg.de

\appendix

{\footnotesize

	\section{Supplement: Further technical proofs}\label{sec:proofstechnical}
	
	At first we introduce further notations. For a random variable $Y \in \R$ we write $Y \sim \subg(\tau)$ with $\tau>0$ if $\Prob(| Y | \geq t) \leq 2\, \exp\big(- t^2 /(2\, \tau^2 ) \big)$ for all $t \geq 0$, and for a random vector $\f{Y} \in \R^d$ we write $\f{Y} \sim \subg_d(\tau)$ if $\Prob(|v^\top \f{Y}  | \geq t) \leq 2\, \exp\big(- t^2 /(2\, \tau^2 \normzq{v}) \big)$ for all $v \in \R^d \setminus \{ \f{0}_d \}$ and $t \geq 0$. In addition a random variable $Y \sim \sube(\tau,b)$ is called sub-Exponential with $\tau,b > 0$ if $\E[Y]=0$ and $\E\big[\exp(t\,Y)\big] \leq \exp\big(t^2 \tau^2/2\big)$ for all $|t| < 1/b$. Furthermore, we denote by $\vec{X}_1,\dotsc,\vec{X}_p \in \R^n$ the columns of $\X_n$ and the rows are $\f{X_i} = (X_{i,1},\dotsc,X_{i,p})^\top$. Finally, $e_k$ is the $k^{\text{th}}$ unit vector, with $k^{\text{th}}$ coordinate equal to $1$, and zero entries otherwise. The dimension of $e_k$ will depend on and be clear from the context. %
	
	\setcounter{equation}{70}
	
	\subsection{Proofs for Section 7.2}\label{sec:proofprep}
	
	\begin{proof}[Proof of Lemma \ref{approximation:error:pseudo:huber}]
		Let $l(x) = x^2$, then 
		by (ii) of Assumption \ref{assfan2017} we get 
		\begin{align}
			\E \Big[l\big(Y_1 - \f{X_1^\top} \beta_{\alpha_n}^* \big) - l\big(Y_1 - \f{X_1^\top} \beta^* \big) \Big] 
			&= \big( \beta_{\alpha_n}^* - \beta^* \big)^\top \E \Big[ \f{X_1} \f{X_1^\top} \Big] \, \big( \beta_{\alpha_n}^* - \beta^* \big) \notag \\
			&\geq \CXl \, \normzq{\beta_{\alpha_n}^* - \beta^*}. \label{proof:approximation:1}
		\end{align}
		Let $g_{\alpha_n}(x) = l(x) - l_{\alpha_n}(x) = x^2 - 2 \alpha_n^{-2} \Big( \sqrt{1 + \alpha_n^2 x^2} -1 \Big)$, then
		\begin{align}
			\E \Big[l\big(Y_1 - \f{X_1^\top} \beta_{\alpha_n}^* \big) - l\big(Y_1 - \f{X_1^\top} \beta^* \big) \Big] &=  \E \Big[ l\big(Y_1 - \f{X_1^\top} \beta_{\alpha_n}^* \big) - l_{\alpha_n}\big(Y_1 - \f{X_1^\top} \beta_{\alpha_n}^* \big) \notag \\
			&\quad \quad \quad \quad+ l_{\alpha_n}\big(Y_1 - \f{X_1^\top} \beta_{\alpha_n}^* \big) - l_{\alpha_n}\big(Y_1 - \f{X_1^\top} \beta^* \big) \notag \\
			&\quad \quad \quad \quad + l_{\alpha_n}\big(Y_1 - \f{X_1^\top} \beta^* \big) - l\big(Y_1 - \f{X_1^\top} \beta^* \big) \Big] \notag \\
			&\leq \E \Big[ g_{\alpha_n}\big(Y_1 - \f{X_1^\top} \beta_{\alpha_n}^* \big) - g_{\alpha_n}\big(Y_1 - \f{X_1^\top} \beta^* \big) \Big] \label{proof:approximation:2}
		\end{align}
		because $\beta_{\alpha_n}^*$ minimizes $\E\big[l_{\alpha_n}(Y_1 - \f{X_1^\top} \beta )\big]$ over  $\normz{\beta} \leq \Cbeta$ and $\normz{\beta^*} \leq \Cbeta$ by (iv) of Assumption \ref{assfan2017}. Furthermore, the mean value theorem implies 
		\begin{align}
			\E \Big[ g_{\alpha_n}\big(Y_1 - \f{X_1^\top} \beta_{\alpha_n}^* \big) - g_{\alpha_n}\big(Y_1 - \f{X_1^\top} \beta^* \big) \Big] &=\E \Big[  g_{\alpha_n}' (Z ) \big(\f{X_1^\top} ( \beta^* -\beta_{\alpha_n}^* ) \big) \Big] \notag  \\
			&\leq \E \Big[  \big| g_{\alpha_n}' (Z )\big| \, \big|\f{X_1^\top} ( \beta^* -\beta_{\alpha_n}^* ) \big| \, \one{ \{|Z| \geq \alpha_n^{-1} \} } \Big] \label{proof:approximation:3} \\
			& \quad \quad \quad + \E \Big[  \big| g_{\alpha_n}' (Z )\big| \, \big|\f{X_1^\top} ( \beta^* -\beta_{\alpha_n}^* ) \big| \, \one{ \{|Z| < \alpha_n^{-1} \} } \Big] \notag 
		\end{align}
		with $Z = Y_1 - \f{X_1^\top} \widetilde{\beta}$ and $\widetilde{\beta}$ between $\beta^*$ and $\beta_{\alpha_n}^*$. Note that $\widetilde{\beta}$ is also a random vector. For the first summand we obtain from (19) that
		\begin{align*}
			\E \Big[  \big| g_{\alpha_n}' (Z )\big| \, &\big|\f{X_1^\top} ( \beta^* -\beta_{\alpha_n}^* ) \big| \, \one{ \{|Z| \geq \alpha_n^{-1} \} } \Big] \\
			&\leq 2 \, \E \bigg[ |Z| \bigg( 1 - \frac{1}{\sqrt{1 + \alpha_n^2 Z^2 }}\bigg) \, \big|\f{X_1^\top} ( \beta^* -\beta_{\alpha_n}^* ) \big| \, \one{ \{|Z| \geq \alpha_n^{-1} \} } \bigg] \,.
		\end{align*}
		Let $\Prob_{\varepsilon}$ be distribution of $\varepsilon_1$ conditional on $\f{X_1}$ and $\E_{\varepsilon}$ the corresponding conditional expectation. Then we get the inequality
		{\allowdisplaybreaks
			\begin{align*}
				\E_{\varepsilon} \bigg[ |Z| \bigg( 1 - \frac{1}{\alpha_n^2 Z^2}\bigg) \one{ \{|Z| \geq \alpha_n^{-1} \} }\bigg] &\leq \E_{\varepsilon} \big[ |Z| \one{ \{|Z| \geq \alpha_n^{-1} \} }\big]\\
				&= \int_0^\infty \Prob_{\varepsilon} \big( |Z| \geq \alpha_n^{-1} , |Z| > t \big) \, dt \\
				&= \int_{\alpha_n^{-1}}^\infty \Prob_{\varepsilon} \big( |Z| > t \big) \, dt + \int_0^{\alpha_n^{-1}} \Prob_{\varepsilon} \big( |Z| \geq \alpha_n^{-1} \big) \, dt \\
				&\leq \int_{\alpha_n^{-1}}^\infty \frac{\E_{\varepsilon} \big[|Z|^\Cpm \big]}{t^\Cpm} \, dt + \int_0^{\alpha_n^{-1}} \frac{\E_{\varepsilon} \big[|Z|^\Cpm\big]}{\alpha_n^{-\Cpm}} \, dt \\
				&= \frac{\alpha_n^{\Cpm-1}}{\Cpm-1} \, \E_{\varepsilon} \big[|Z|^\Cpm\big] + \alpha_n^{\Cpm-1} \, \E_{\varepsilon} \big[|Z|^\Cpm\big] \\
				&\leq 2 \, \alpha_n^{\Cpm-1} \, \E_{\varepsilon} \big[|Z|^\Cpm\big] \,,
			\end{align*}
			where $\Cpm \in \{2,3\}$ is given in Assumption \ref{assfan2017}, and in consequence}
		\begin{align}
			\E \Big[  \big| g_{\alpha_n}' (Z )\big| \, \big|\f{X_1^\top} ( \beta^* -\beta_{\alpha_n}^* ) \big| \, \one{ \{|Z| \geq \alpha_n^{-1} \} } \Big] \leq 4  \, \alpha_n^{\Cpm-1} \, \E \Big[ |Z|^\Cpm \, \big|\f{X_1^\top} ( \beta^* -\beta_{\alpha_n}^* ) \big| \Big] \, . \label{proof:approximation:4}
		\end{align}
		Now we analyze the second term in \eqref{proof:approximation:3}. Taking the derivative in the series expansion 
		%
		%
		\begin{align}
			g_{\alpha_n}(x) = - 2  \sum_{k=2}^{\infty} \binom{1/2}{k} \alpha_n^{2k-2} \, x^{2k}  \,, \qquad \alpha^2 x^2 \leq 1\,, \label{eq:taylor:expansion} 
		\end{align}
		implies that
		\begin{align*}
			\big| g_{\alpha_n}'(x) \big| = \Bigg| - 2  \sum_{k=2}^{\infty} \binom{1/2}{k} 2k \, \alpha_n^{2k-2} \, x^{2k-1} \Bigg| \leq \big| \alpha_n^2 \, x^3 \big| = \alpha_n^2 \, |x|^3 
		\end{align*}
		and hence that
		\begin{align*}
			\E \Big[  \big| g_{\alpha_n}' (Z )\big| \, \big|\f{X_1^\top} ( \beta^* -\beta_{\alpha_n}^* ) \big| \, \one{ \{|Z| < \alpha_n^{-1} \} } \Big] \leq \alpha_n^2 \, \E \Big[  | Z |^3 \, \big|\f{X_1^\top} ( \beta^* -\beta_{\alpha_n}^* ) \big| \, \one{ \{|Z| < \alpha_n^{-1} \} } \Big]
		\end{align*}
		because $\alpha_n |Z| < 1 $. Moreover, it is
		\begin{align*}
			\alpha_n^2 \, \E_{\varepsilon} \big[  | Z |^3 \, \one{ \{|Z| < \alpha_n^{-1} \} } \big] & = \alpha_n^2 \, \E_{\varepsilon} \big[  | Z |^{\Cpm} \, |Z|^{3-\Cpm} \, \one{ \{|Z| < \alpha_n^{-1} \} } \big] \\
			&\leq \alpha_n^{2+\Cpm-3} \, \E_{\varepsilon} \big[  | Z |^{\Cpm} \, \one{ \{|Z| < \alpha_n^{-1} \} }  \big] \leq  \alpha_n^{\Cpm-1} \, \E_{\varepsilon} \big[  | Z |^{\Cpm} \big]
		\end{align*}
		and in consequence
		\begin{align}
			\E \Big[  \big| g_{\alpha_n}' (Z )\big| \, \big|\f{X_1^\top} ( \beta^* -\beta_{\alpha_n}^* ) \big| \, \one{ \{|Z| < \alpha_n^{-1} \} } \Big] \leq \alpha_n^{\Cpm-1} \, \E \Big[ |Z|^\Cpm \, \big|\f{X_1^\top} ( \beta^* -\beta_{\alpha_n}^* ) \big| \Big] \, . \label{proof:approximation:5}
		\end{align}
		So in total we obtain by \eqref{proof:approximation:1} -  \eqref{proof:approximation:5} the inequality
		\begin{align}
			\normzq{\beta_{\alpha_n}^* - \beta^*} \leq \frac{5}{\CXl} \, \E \Big[ |Z|^{\Cpm} \, \big|\f{X_1^\top} ( \beta^* -\beta_{\alpha_n}^* ) \big| \Big]  \, \alpha_n^{\Cpm-1} \, . \label{proof:approximation:6}
		\end{align}
		
		The mean on the right hand side can be upper bounded by
		\begin{align}
			\E \Big[ |Z|^{\Cpm} \, \big|\f{X_1^\top} ( \beta^* -\beta_{\alpha_n}^* ) \big| \Big] &= \E \Big[ \big|\varepsilon_1 + \f{X_1^\top} ( \beta^* -\widetilde{\beta} ) \big|^{\Cpm} \, \big|\f{X_1^\top} ( \beta^* -\beta_{\alpha_n}^* ) \big| \Big] \notag \\
			& \leq 2^{\Cpm -1 } \bigg( \E \Big[ |\varepsilon_1 |^{\Cpm} \, \big|\f{X_1^\top} ( \beta^* -\beta_{\alpha_n}^* ) \big| \Big] + \E \Big[ \big|\f{X_1^\top} ( \beta^* -\widetilde{\beta} ) \big|^{\Cpm} \, \big|\f{X_1^\top} ( \beta^* -\beta_{\alpha_n}^* ) \big| \Big] \bigg) \,. \label{proof:approximation:7}
		\end{align}
		Moreover, for the first term in the brackets we obtain by H\"older's inequality and (i) of Assumption \ref{assfan2017} 
		\begin{align*}
			\E \Big[ |\varepsilon_1 |^{\Cpm} \, \big|\f{X_1^\top} ( \beta^* -\beta_{\alpha_n}^* ) \big| \Big] &= \E \Big[ \E \big[|\varepsilon_1|^{\Cpm} \big| \f{X_1} \big] \, \big|\f{X_1^\top} ( \beta^* -\beta_{\alpha_n}^* ) \big| \Big] \\
			&\leq \E \Big[ \E \big[|\varepsilon_1|^{\Cpm} \big| \f{X_1} \big]^q \Big]^{\frac{1}{q}} \, \E \Big[ \big|\f{X_1^\top} ( \beta^* -\beta_{\alpha_n}^* ) \big|^\frac{q}{q-1} \Big]^\frac{q-1}{q} \\
			&\leq (\Cm)^\frac{1}{q} \, \E \Big[ \big|\f{X_1^\top} ( \beta^* -\beta_{\alpha_n}^* ) \big|^\frac{q}{q-1} \Big]^\frac{q-1}{q} \,.
		\end{align*}
		In addition note that $\f{X_1^\top} ( \beta^* -\beta_{\alpha_n}^* ) \sim \subg \big( \CXsub \, \normz{\beta^* -\beta_{\alpha_n}^*} \big)$ by (iii) of Assumption \ref{assfan2017}, and that the moments of a sub-Gaussian random variable $Q \sim \subg(\tau)$ with $\tau>0$ are bounded by 
		\begin{align}
			\E \big[|Q|^r\big] \leq \big(2 \tau^2 \big)^{\frac{r}{2}} \, r \, \Gamma\bigg(\frac{r}{2}\bigg)\,, \quad \E \big[|Q|^r\big]^\frac{1}{r} \leq \sqrt{2} \, \Bigg( r \, \Gamma\bigg(\frac{r}{2}\bigg) \Bigg)^\frac{1}{r} \, \tau \label{subgaussian:moments}
		\end{align}
		for $r>1$. This can be proven analog to \citet[Lemma 1.4]{Rigollet2019}. 
		Hence
		\begin{align}
			\E \Big[ |\varepsilon_1 |^{\Cpm} \, \big|\f{X_1^\top} ( \beta^* -\beta_{\alpha_n}^* ) \big| \Big] \leq \sqrt{2} \, (\Cm)^\frac{1}{q} \, \Bigg( \frac{q}{q-1} \, \Gamma\bigg(\frac{q}{2(q-1)} \bigg) \Bigg)^\frac{q-1}{q}   \CXsub \, \normz{\beta^* -\beta_{\alpha_n}^*} \,. \label{proof:approximation:8}
		\end{align}
		For the second term in the brackets in \eqref{proof:approximation:7} the Cauchy-Schwarz inequality implies
		\begin{align}
			\E \Big[ \big|\f{X_1^\top} ( \beta^* -\widetilde{\beta} ) \big|^{\Cpm} \, \big|\f{X_1^\top} ( \beta^* -\beta_{\alpha_n}^* ) \big| \Big] &\leq \bigg( \E \Big[ \big|\f{X_1^\top} ( \beta^* -\widetilde{\beta} ) \big|^{2\Cpm} \Big] \, \E \Big[\big|\f{X_1^\top} ( \beta^* -\beta_{\alpha_n}^* ) \big|^2 \Big] \bigg)^\frac{1}{2} \notag \\
			&\leq 2 \, \E \Big[ \big|\f{X_1^\top} ( \beta^* -\widetilde{\beta} ) \big|^{2\Cpm} \Big]^\frac{1}{2} \,\CXsub \, \normz{\beta^* -\beta_{\alpha_n}^*} \,. \label{proof:approximation:9}
		\end{align}
		To give a upper bound for the remaining expected value we consider at first a tail bound for the appropriate random variable. Let $L$ be the line between $\beta^*$ and $\beta_{\alpha_n}^*$, then $L$ is also the convex hull of $\mathcal{V}(L) = \{\beta^*,\beta_{\alpha_n}^*\}$ and we obtain
		\begin{align*}
			\Prob \Big( \big| ( \beta^* - \widetilde{\beta} )^\top \f{X_1} \big| > x \Big) \leq \Prob \big( \max_{u \in L} \big| u^\top \f{X_1} \big| > x \big)
		\end{align*}
		for $x \geq 0$ because $\widetilde{\beta}$ lies between $\beta^*$ and $\beta_{\alpha_n}^*$. Moreover, $\f{X_1^\top} \beta^*$ and $\f{X_1^\top} \beta_{\alpha_n}^* $ are sub-Gaussian with variance proxy $\Cbetas \, \CXsubs$ by (iii) and (iv) of Assumption \ref{assfan2017} and $\normz{\beta_{\alpha_n}^*} \leq \Cbeta$ by \eqref{eq:betaalpha}. Hence \citet[Theorem 1.16]{Rigollet2019} leads to
		\begin{align*}
			\Prob \big( \big| ( \beta^* - \widetilde{\beta} )^\top \f{X_1} \big| > x \big) \leq \Prob \big( \max_{u \in L} \big| u^\top \f{X_1} \big| > x \big) \leq 4 \, \exp \bigg(-\frac{x^2}{2 \Cbetas \, \CXsubs}\bigg) \,.
		\end{align*} 
		In addition \citet[Lemma 1.4]{Rigollet2019} and the corresponding proof imply
		\begin{align}
			\E \Big[ \big|\f{X_1^\top} ( \beta^* - \widetilde{\beta} ) \big|^{2\Cpm} \Big] \leq 2 \, \big(2 \Cbetas \, \CXsubs\big)^\Cpm \, (2\Cpm)!\, \Gamma(\Cpm) \,. \label{proof:approximation:10}
		\end{align}
		In total \eqref{proof:approximation:6} - \eqref{proof:approximation:10} leads to
		\begin{align*}
			\normz{\beta_{\alpha_n}^* - \beta^*} \leq  \Capprox \, \alpha_n^{\Cpm-1}
		\end{align*}
		with
		\begin{align*}
			\Capprox = \frac{5\,2^{\Cpm } \, \CXsub}{\CXl} \, \Bigg( (\Cm)^\frac{1}{q} \Bigg( \frac{q}{q-1} \, \Gamma\bigg(\frac{q}{2(q-1)} \bigg) \Bigg)^\frac{q-1}{q}    + \Big(2 \, \big(2 \Cbetas \, \CXsubs\big)^\Cpm \, (2\Cpm)!\, \Gamma(\Cpm) \Big)^\frac{1}{2} \Bigg) \,.
		\end{align*}
	\end{proof}

	\begin{proof}[Proof of Lemma \ref{lemma:RSC:pseudo:huber}]
		We obtain
		\begin{align*}
			\skalar{\nabla \loss_{n,\alpha}^{\,\Hu} (\beta+\Delta) - \nabla \loss_{n,\alpha}^{\,\Hu} (\beta)}{\Delta} = \frac{1}{n} \sum_{i=1}^{n} \Big( l_{\alpha}' \big( Y_i - \f{X_i^\top} \beta \big) - l_{\alpha}' \big( Y_i - \f{X_i^\top} (\beta + \Delta) \big) \Big) \, \f{X_i^\top} \Delta
		\end{align*}
		for $\beta, \Delta \in \R^p$ by (21). Firstly we show that
		\begin{align}
			\skalar{\nabla \loss_{n,\alpha}^{\,\Hu} (\beta+\Delta) - \nabla \loss_{n,\alpha}^{\,\Hu} (\beta)}{\Delta} \geq \frac{1}{2n} \sum_{i=1}^{n} \varphi_{\tau \normz{\Delta}} \Big( \f{X_i^\top} \Delta \, \one{\{ |Y_i - \f{X_i^\top} \beta| \leq T \}} \Big) \label{proof:RSC:pseudo:lasso:1}
		\end{align}
		for all $\alpha \leq 1/(T + 8 \tau \,\Cbeta)$ and $(\beta,\Delta) \in A \defeq \big\{(\beta,\Delta) : \normz{\beta} \leq 4 \Cbeta \textrm{ and } \normz{\Delta } \leq 8 \Cbeta \big\}$, where 
		\begin{align*}
			\varphi_t(u) = u^2 \, \one{\{ |u| \leq t/2 \}} + \big(t-|u|\big)^2 \, \one{\{ t/2 < |u| \leq t \}}
		\end{align*}
		and 
		\begin{align*}
			T = 96\, \frac{\CXsubs \, \sqrt{\CXu} \, \Cbeta}{\CXl} \,, \quad \quad \tau = \max \Big\{ 4 \CXsub \sqrt{\log(12 \CXsubs/\CXl)},1 \Big\}\,.
		\end{align*}
		The function $\varphi_t$ satisfies obviously $\varphi_t(u) \leq u^2 \, \one{\{|u| \leq t \}}$. Let $i \in \{1,\dotsc,n\}$ be fixed, then we get on the one hand
		\begin{align*}
			\varphi_{\tau \normz{\Delta}} \Big( \f{X_i^\top} \Delta \, \one{\{ |Y_i - \f{X_i^\top} \beta| \leq T \}} \Big) = 0
		\end{align*}
		if $|\f{X_i^\top} \Delta| > \tau \normz{\Delta}$ or $|Y_i - \f{X_i^\top} \beta| > T$. In addition we have always 
		\begin{align*}
			\Big( l_{\alpha}' \big( Y_i - \f{X_i^\top} \beta \big) - l_{\alpha}' \big( Y_i - \f{X_i^\top} (\beta + \Delta) \big) \Big) \, \f{X_i^\top} \Delta \geq 0
		\end{align*}
		because of the convexity of $g(\beta) = l_\alpha( Y_i - \f{X_i^\top} \beta)$. On the other hand, if $|\f{X_i^\top} \Delta| \leq \tau \normz{\Delta}$ and $|Y_i - \f{X_i^\top} \beta| \leq T$ we get
		\begin{align*}
			\big| Y_i - \f{X_i^\top} \beta \big| \leq T \leq \alpha^{-1}
		\end{align*}
		and
		\begin{align*}
			\big| Y_i - \f{X_i^\top} (\beta + \Delta) \big| \leq \big| Y_i - \f{X_i^\top} \beta \big| + \big| \f{X_i^\top} \Delta \big| \leq T + \tau \normz{\Delta} \leq T + 8 \tau \, \Cbeta \leq \alpha^{-1}
		\end{align*}
		because $(\beta,\Delta) \in A$ and the choice of $\alpha$. In addition the mean value theorem implies
		\begin{align*}
			l_{\alpha}' \big( Y_i - \f{X_i^\top} \beta \big) - l_{\alpha}' \big( Y_i - \f{X_i^\top} (\beta + \Delta) \big)  = l_{\alpha}''(c) \Big( Y_i - \f{X_i^\top} \beta - Y_i + \f{X_i^\top} (\beta + \Delta) \Big) = l_{\alpha}''(c) \, \f{X_i^\top} \Delta
		\end{align*}
		with $c \in \big(Y_i - \f{X_i^\top} \beta,Y_i - \f{X_i^\top} (\beta + \Delta) \big)$ since the pseudo Huber loss $l_\alpha$ is twice differentiable. The above conditions lead to $|c| \leq \alpha^{-1}$ as well. Moreover, note that
		\begin{align*}
			l_\alpha''(c) = \frac{2 \alpha^{-3}}{(\alpha^{-2} + c^2)^{3/2}} \geq \frac{2 \alpha^{-3}}{(2 \alpha^{-2})^{3/2}} = \frac{2}{ 2^{3/2}} \geq \frac{1}{2}
		\end{align*}
		for all $|c| \leq \alpha^{-1}$. Hence it follows that
		\begin{align*}
			\Big( l_{\alpha}' \big( Y_i - \f{X_i^\top} \beta \big) - l_{\alpha}' \big( Y_i - \f{X_i^\top} (\beta + \Delta) \big) \Big) \, \f{X_i^\top} \Delta &= l_{\alpha}''(c) \, \big(\f{X_i^\top} \Delta\big)^2 \geq \frac{1}{2} \, \big(\f{X_i^\top} \Delta\big)^2 \\
			&\geq \frac{1}{2} \, \varphi_{\tau \normz{\Delta}} \Big( \f{X_i^\top} \Delta \, \one{\{ |Y_i - \f{X_i^\top} \beta| \leq T \}} \Big)
		\end{align*}
		if $|\f{X_i^\top} \Delta| \leq \tau \normz{\Delta}$ and $|Y_i - \f{X_i^\top} \beta| \leq T$. So in total inequality \eqref{proof:RSC:pseudo:lasso:1} is satisfied for all $(\beta,\Delta) \in A$ and $\alpha \leq 1/(T + 8 \tau \Cbeta)$. Furthermore, the condition of $\alpha$ reduces to $\alpha \leq \calpha$ where $\calpha$ is a positive constant depending on $\CXl$, $\CXu$, $\CXsub$ and $\Cbeta$, because of the choice of $T$ and $\tau$. The proof of \citet[Lemma 2]{Fan2017} provides
		\begin{align}
			\frac{1}{n} \sum_{i=1}^{n} \varphi_{\tau \normz{\Delta}} \Big( \f{X_i^\top} \Delta \, \one{\{ |Y_i - \f{X_i^\top} \beta| \leq T \}} \Big) \geq c_1 \normz{\Delta} \Bigg(\normz{\Delta} - c_2 \bigg(\frac{\log(p)}{n}\bigg)^\frac{1}{2} \norme{\Delta} \Bigg) \label{proof:RSC:pseudo:lasso:2}
		\end{align}
		with $c_1 = \CXl/4$ and $c_2=160 \, \tau^2 \CXsub/ \CXl$. Additionally the proof of \citet[Lemma 4]{Fan2017} leads to
		\begin{align}
			c_1 \normz{\Delta} \Bigg(\normz{\Delta} - c_2 \bigg(\frac{\log(p)}{n}\bigg)^\frac{1}{2} \norme{\Delta} \Bigg) \geq \frac{c_1}{2} \normzq{\Delta} - \frac{c_1 c_2^2}{2} \frac{\log(p)}{n} \normeq{\Delta} \,. \label{proof:RSC:pseudo:lasso:3}
		\end{align}
		All in all the inequalities \eqref{proof:RSC:pseudo:lasso:1} - \eqref{proof:RSC:pseudo:lasso:3} imply the assertion of Lemma \ref{lemma:RSC:pseudo:huber}. 
	\end{proof}

	\begin{proof}[Proof of Lemma \ref{lem:derivativelem1}]
		
		For $v \in B_{S}\defeq \big\{ v \in \R^p \mid \supp(v) \subseteq S, \normz{v} = 1 \big\} $ we have that
		\begin{align*}
			\big(\nabla^2 \loss_{n,\alpha}^{\,\Hu}(\beta) \big) \, v = \lim_{t \to 0} \frac{\nabla\loss_{n,\alpha}^{\,\Hu}\big(\beta+t\,v\big) -\nabla\loss_{n,\alpha}^{\,\Hu}(\beta)}{t}
		\end{align*}
		and hence that
		\begin{align}
			v^\top \big(\nabla^2 \loss_{n,\alpha}^{\,\Hu}(\beta) \big) v = \lim_{t \to 0} \frac{ \skalar{\nabla\loss_{n,\alpha}^{\,\Hu}\big(\beta+t\,v\big) -\nabla\loss_{n,\alpha}^{\,\Hu}(\beta)}{t\,v}}{t^2} \,. \label{limest0}
		\end{align}
		
		The RSC condition \eqref{eq:rsccond} implies that for $t\leq 1$ and $v \in B_S$  we obtain
		\begin{align}
			\skalar{\nabla\loss_{n,\alpha}^{\,\Hu}\big(\beta+t\,v\big) -\nabla\loss_{n,\alpha}^{\,\Hu}(\beta)}{t\,v} \geq  t^2\bigg( \Crscf \normzq{v} - \Crscs\, \frac{\log(p)}{n} \normeq{v} \bigg) \geq  t^2\bigg( \Crscf - \Crscs\, \frac{s \log(p)}{n} \bigg) \label{rscfortv}
		\end{align}
		where we used $\norme{v} \leq \sqrt{s}\,\normz{v}$ since $\supp(v) \subseteq S$ and $\normz{v} = 1 $. 
		Plugging this into \eqref{limest0} together with the condition $n \geq \Crsct s \log(p)$ gives  
		\begin{align*}
			v^\top \big(\nabla^2 \loss_{n,\alpha}^{\,\Hu}(\beta) \big) v \geq \Crscf - \frac{\Crscf}{2} = \frac{\Crscf}{2} \,,
		\end{align*}
		which is equivalent to the estimate \eqref{positvedefinitenesshessian}. 
	\end{proof}

	\begin{proof}[Proof of Lemma \ref{ratelinftygradient}]
		By (19) and (21) we obtain
		\begin{align*}
			\nabla \loss_{n,\alpha_n}^{\,\Hu} \big(\beta_{\alpha_n}^*\big) = - \frac{1}{n} \sum_{i=1}^{n} l_{\alpha_n}'\big(Y_i-\f{X_i^\top} \beta_{\alpha_n}^* \big) \f{X_i} 
		\end{align*}
		with $\big|l_{\alpha_n}'(x)\big| \leq 2 \alpha_n^{-1}$ and $\big|l_{\alpha_n}'(x)\big| \leq 2 |x|$ for all $x \in \R$. Furthermore, by \eqref{subgaussian:moments} in the proof of Lemma  \ref{approximation:error:pseudo:huber} it follows that
		\begin{align}
			\E \big[|Q|^{ru}\big]^\frac{1}{r} \leq \Bigg(\big(2 \tau^2 \big)^{\frac{ru}{2}} \, ru \, \Gamma\bigg(\frac{ru}{2}\bigg) \Bigg)^\frac{1}{r} \leq \big(2 \tau^2 \big)^{\frac{u}{2}} \Big( \big(ru \big)! \Big)^\frac{1}{r} \leq \big(2 \tau^2 \big)^{\frac{u}{2}} \Big(\big(u!\big)^r r^{ru}\Big)^\frac{1}{r} = \big(2 \tau^2 \big)^{\frac{u}{2}} u! \, r^u \label{proof:ratelinftygradient:1}
		\end{align}
		for $Q \sim \subg(\tau)$ with $\tau>0$ and $u,r \in \N$ with $u \geq 2$ and $r/2 \in \N$. In the last inequality we bound the $r$ largest factors of $(ru)!$ by $ru$, then the next $r$ largest factors by $r(u-1)$ and so on. Now we choose $1 < q_1 \leq \Choeld$, where $\Choeld$ is given in Assumption \ref{assfan2017}, such that $r_1=q_1/(q_1-1) \in \N$ and $r_1$ is even. Then we obtain
		\begin{align*}
			\E\bigg[ \Big( l_{\alpha_n}'\big(Y_i-\f{X_i^\top} \beta_{\alpha_n}^* \big) X_{i,k} \Big)^2 \bigg] &\leq 4 \, \E \Big[ \big(\varepsilon_i-\f{X_i^\top} (\beta^* - \beta_{\alpha_n}^*) \big)^2 \, X_{i,k}^2 \Big] \notag \\
			&\leq 8 \, \E \bigg[ \Big( \varepsilon_i^2 + \big(\f{X_i^\top} (\beta^* - \beta_{\alpha_n}^*) \big)^2 \Big) \, X_{i,k}^2 \bigg] \notag \\
			&= 8 \, \E \Big[ \E\big[ \varepsilon_i^2 \big| \f{X_i} \big] \, X_{i,k}^2 + \big(\f{X_i^\top} (\beta^* - \beta_{\alpha_n}^*) \big)^2 \, X_{i,k}^2 \Big] \notag \\
			&\leq 8 \, \E \bigg[ \Big( 1 + \E\big[ |\varepsilon_i|^m \big| \f{X_i} \big] \Big)\, X_{i,k}^2 + \big(\f{X_i^\top} (\beta^* - \beta_{\alpha_n}^*) \big)^2 \, X_{i,k}^2 \bigg] \notag \\
			&\leq \Cgradt
		\end{align*}
		with
		\begin{align*} 
			\Cgradt = 32 \, \CXsubs \Big( 1 + (1+\Cm)^\frac{1}{q_1} r_1^2 + 2^2 \, 64 \, \CXsubs \, \Cbetas \Big) 
		\end{align*}
		for $k=1,\dotsc,p$. In the last inequality we used the H\"older and Cauchy-Schwarz inequality, \eqref{proof:ratelinftygradient:1} with $u=2$ and $r \in \{2,r_1\}$, and the fact that $X_{i,k} \sim \subg(\CXsub)$ and $\f{X_i^\top} (\beta^* - \beta_{\alpha_n}^*) \sim \subg(2 \Cbeta\,\CXsub)$ by (iii) of Assumption \ref{assfan2017}. Analog we obtain for higher moments, $u \geq 3$, using $\big|l_{\alpha_n}'(x)\big|^u \leq 4 \, (2 \alpha_n^{-1})^{u-2} \, x^2$, the estimate
		\begin{align*}
			\E\bigg[ \Big| l_{\alpha_n}'\big(Y_i-\f{X_i^\top} \beta_{\alpha_n}^* \big) \, X_{i,k} \Big|^u \bigg] &\leq 4 \, \bigg( \frac{2}{\alpha_n}\bigg)^{u-2} \E \Big[ \big(\varepsilon_i-\f{X_i^\top} (\beta^* - \beta_{\alpha_n}^*) \big)^2 \, \big|X_{i,k}\big|^u \Big] \notag \\
			&\leq 8 \, \bigg( \frac{2}{\alpha_n}\bigg)^{u-2} \E \bigg[ \Big( 1 + \E\big[ |\varepsilon_i|^m \big| \f{X_i} \big] \Big)\, \big|X_{i,k}\big|^u + \big(\f{X_i^\top} (\beta^* - \beta_{\alpha_n}^*) \big)^2 \, \big|X_{i,k}\big|^u \bigg] \notag \\
			&\leq 8 \, \bigg( \frac{2}{\alpha_n}\bigg)^{u-2} 2^\frac{u}{2} \,\CXsubu \, u! \,\Big( 1 + (1+\Cm)^\frac{1}{q_1} r_1^u  + 2^u \, 64 \, \CXsubs \, \Cbetas \Big) \notag \\
			&= u!\, \bigg( \frac{\sqrt{2} \, 2 \, \CXsub}{\alpha_n}\bigg)^{u-2} 16 \,\CXsubs \Big(1 + (1+\Cm)^\frac{1}{q_1} r_1^u +  2^u \, 64 \, \CXsubs \, \Cbetas \Big) \notag \\
			&\leq \frac{u!}{2} \, \bigg(\frac{2\,\Cgradv}{\alpha_n} \bigg)^{u-2} \, \Cgradt
		\end{align*}
		with
		\begin{align*}
			\Cgradv = \sqrt{2} \, \max(r_1,2) \, \CXsub \,.
		\end{align*}
		In addition note that $\E\big[ l_{\alpha_n}'(Y_i-\f{X_i^\top} \beta_{\alpha_n}^* ) \, X_{i,k} \big] = 0 $ because of (18) and (23). Now Bernstein's inequality, cf. \citet[Proposition 2.9]{Massart2007}, leads to
		\begin{align*}
			\Prob\Bigg(\bigg| \frac{1}{n} \sum_{i=1}^{n} l_{\alpha_n}'\big(Y_i-\f{X_i^\top} \beta_{\alpha_n}^* \big) \, X_{i,k} \bigg| \geq \bigg( \frac{2\Cgradt \,x}{n}\bigg)^\frac{1}{2} + \frac{2 \Cgradv \,x}{\alpha_n \, n}\Bigg) \leq 2 \exp(-x)
		\end{align*}
		for $x >0$ since the terms of the sum are independent. Let $x = 3 \log(p)$ and $\Cgradf =\sqrt{96/\Cgradt} \, \Cgradv/4$, then by the choice of $\alpha_n$ we get
		\begin{align*}
			\frac{2 \Cgradv \,x}{\alpha_n \, n} \leq \frac{24 \Cgradv }{\Cgradv } \, \bigg(\frac{\Cgradt \, \log(p)}{96 n}\bigg)^\frac{1}{2} = \bigg( \frac{2\Cgradt \,x}{n}\bigg)^\frac{1}{2}
		\end{align*}
		and hence
		\begin{align*}
			\Prob\Bigg(\bigg| \frac{1}{n} \sum_{i=1}^{n} l_{\alpha_n}'\big(Y_i-\f{X_i^\top} \beta_{\alpha_n}^* \big) \, X_{i,k} \bigg| \geq 2 \,\bigg( \frac{6 \Cgradt \,\log(p)}{n}\bigg)^\frac{1}{2} \Bigg) \leq 2 \exp\big(-3 \log(p)\big) \,.
		\end{align*}
		Union bound implies
		\begin{align*}
			\Prob\Bigg(\normi{\nabla \loss_{n,\alpha_n}^{\,\Hu} \big(\beta_{\alpha_n}^*\big)} \geq 2 \,\bigg( \frac{6 \Cgradt \,\log(p)}{n}\bigg)^\frac{1}{2} \Bigg) \leq 2 \exp\big(-3 \log(p) + \log(p) \big) = \frac{2}{p^2} 
		\end{align*}
		and $\Cgrads = 2 (6 \Cgradt)^\frac{1}{2}$.
	\end{proof}
	
	\subsection{Proof of Lemma \ref{lemmaassstrictdual(ii)}}\label{sec:lemmaassstrictdual(ii)}

	The proof of Lemma \ref{lemmaassstrictdual(ii)} relies on the following two technical results. 
	
	\setcounter{lemma}{14}
	
	\begin{lemma} \label{lemmaspektralnormQssCovariancematrix}
		Suppose Assumption \ref{assfan2017} and $\alpha_n = \Calpha \big(\frac{\log(p)}{n}\big)^\frac{1}{2}$ for some positive constant $\Calpha >0$ hold. If in addition $n \geq \max \big\{ (576\,\log(6) \, \Calpha \, \Cbetas \,\CXsubs)^2 s^2 \log(p), 16 \log(24)\, s \log(p) \big\}$, then there exist positive constants $C_1,C_2,C_3 >0$ such that 
		\begin{align}
			\normzM{\widehat{Q}_{SS} - \E\big[\f{X_1}\f{X_1^\top}\big]_{SS}} &\leq C_2 \, \max\Bigg\{ \bigg(\frac{s}{n}\bigg)^\frac{1}{2} \,,\, \frac{s}{n} \,,\, \bigg(\frac{\log(p)}{n}\bigg)^\frac{1}{2} \,,\, \bigg(\frac{s \log(p)}{n}\bigg)^\frac{1}{2} \,,\, \alpha_n^\frac{\Cpm}{2} \,,\, \alpha_n^{\Cpm-\frac{1}{2}} \,,\, \alpha_n \Bigg\} \notag \\
			&\leq \frac{C_3}{\sqrt{s}}
		\end{align}
		with probability at least $1-C_1/ p^{2} - 6/p^{5s}$.
	\end{lemma}
	
	\begin{proof}[Proof of Lemma \ref{lemmaspektralnormQssCovariancematrix}]
		
		The following proof uses elements of the proof of Lemma 1 in \citet{Sun2019}.
		Let $\mathcal{B}_2^s = \big\{ u \in \R^s \mid \normz{u} \leq 1 \big\}$. Then using (34) in Lemma 6 we have
		\begin{align}
			\normzM{\frac{2}{n} \sum_{i=1}^{n} \,\big(\f{X_i}\big)_S \big(\f{X_i}\big)_S^\top - \widehat{Q}_{SS}} 
			&= \max_{u \in \mathcal{B}_2^s} \, u^\top \bigg( \frac{2}{n} \sum_{i=1}^{n} \big(1-d_i\big) \,\big(\f{X_i}\big)_S \big(\f{X_i}\big)_S^\top \bigg) u \notag \\
			&= \max_{u \in \mathcal{B}_2^s} \Big( Z_n^1 (u) + Z_n^2 (u) \Big) \notag \\
			&\leq \max_{u \in \mathcal{B}_2^s} Z_n^1 (u) + \max_{u \in \mathcal{B}_2^s} Z_n^2 (u) \label{proof:1:lemmaspektralnormQssCovariancematrix}
		\end{align}
		with
		\begin{align*}
			Z_n^1(u) &= \frac{1}{n} \sum_{i=1}^{n} \Bigg( \int_{0}^{1} \bigg( 2 - l_{\alpha_n}''\,\Big(Y_i-\f{X_i^\top} \big(\beta_{\alpha_n}^* + t \, \big(\widehat{\beta}_n^{\,\PDW}-\beta_{\alpha_n}^*\big)\big)\Big) \bigg) \,  \\
			& \quad \quad \quad \quad \quad \quad \cdot \mathbbm{1}_{[0,\alpha_n^{-1/2}]} \bigg( \Big| Y_i-\f{X_i^\top} \big(\beta_{\alpha_n}^* + t \, \big(\widehat{\beta}_n^{\,\PDW}-\beta_{\alpha_n}^*\big) \big) \Big| \bigg) dt  \Bigg) \Big(u^\top\big(\f{X_i}\big)_S \Big)^2 \,,\\
			Z_n^2(u) &= \frac{1}{n} \sum_{i=1}^{n} \Bigg( \int_{0}^{1} \bigg( 2 - l_{\alpha_n}''\,\Big(Y_i-\f{X_i^\top} \big(\beta_{\alpha_n}^* + t \, \big(\widehat{\beta}_n^{\,\PDW}-\beta_{\alpha_n}^*\big)\big)\Big) \bigg) \,  \\
			& \quad \quad \quad \quad \quad \quad \cdot \mathbbm{1}_{(\alpha_n^{-1/2},\infty)} \bigg( \Big| Y_i-\f{X_i^\top} \big(\beta_{\alpha_n}^* + t \, \big(\widehat{\beta}_n^{\,\PDW}-\beta_{\alpha_n}^*\big) \big) \Big| \bigg) dt  \Bigg) \Big(u^\top\big(\f{X_i}\big)_S \Big)^2 \,.
		\end{align*}
		To handle the first sum in \eqref{proof:1:lemmaspektralnormQssCovariancematrix} we consider the series expansion in \eqref{eq:taylor:expansion}, which implies
		\begin{align*}
			\big| 2 - l_{\alpha_n}''(x) \big| = \Bigg| - 2  \sum_{k=2}^{\infty} \binom{1/2}{k} 2k \, (2k -1) \,\alpha_n^{2k-2} \, x^{2k-2} \Bigg| \leq \big| 3 \alpha_n^2 \, x^2 \big| = 3 \alpha_n^2 \, x^2 \,,
		\end{align*}
		if $\alpha_n^2 \, x^2 < 1$. Hence for small $\alpha_n$ we get
		\begin{align*}
			\max_{u \in \mathcal{B}_2^s} Z_n^1 (u) \leq \max_{u \in \mathcal{B}_2^s} \, \frac{3 \alpha_n}{n} \sum_{i=1}^{n} \Big(u^\top\big(\f{X_i}\big)_S \Big)^2 \,.
		\end{align*}
		Standard spectral norm bounds on the sample covariance matrix (with independent and identically distributed sub-Gaussian rows), cf. \citet[Theorem 6.5]{Wainwright2019}, and (ii) of Assumption \ref{assfan2017} lead to 
		\begin{align}
			\max_{u \in \mathcal{B}_2^s} \, \frac{1}{n} \sum_{i=1}^{n} \Big(u^\top\big(\f{X_i}\big)_S \Big)^2 &= \normzM{\frac{1}{n} \sum_{i=1}^{n} \,\big(\f{X_i}\big)_S \big(\f{X_i}\big)_S^\top} \notag \\
			&\leq \normzM{\E\big[\f{X_1}\f{X_1^\top}\big]_{SS}} + \normzM{\frac{1}{n} \sum_{i=1}^{n} \,\big(\f{X_i}\big)_S \big(\f{X_i}\big)_S^\top - \E\big[\f{X_1}\f{X_1^\top}\big]_{SS}} \notag \\
			&\leq \CXu + C_4 \Bigg( \bigg(\frac{s}{n}\bigg)^\frac{1}{2} + \frac{s}{n} + \bigg(\frac{\log(p)}{n}\bigg)^\frac{1}{2} \Bigg) \label{spektral:norm:bound:sample:covariance}
		\end{align}
		with probability at least $1-C_1/ p^2$ for some positive constants $C_1,C_4 >0$. Hence
		\begin{align}
			\max_{u \in \mathcal{B}_2^s} Z_n^1 (u) \leq 3 \big( \CXu + 3 C_4 \big)\,\alpha_n \label{proof:2:lemmaspektralnormQssCovariancematrix}
		\end{align}
		with high probability.
		For the second sum in \eqref{proof:1:lemmaspektralnormQssCovariancematrix} we firstly estimate
		\begin{align}
			\max_{u \in \mathcal{B}_2^s} Z_n^2 (u) \leq \max_{u \in \mathcal{B}_2^s}  \, \frac{2}{n} \sum_{i=1}^{n} \Bigg( \int_{0}^{1} \mathbbm{1}_{(\alpha_n^{-1/2},\infty)} \bigg( \Big| Y_i-\f{X_i^\top} \Big(\beta_{\alpha_n}^* + t \, \big(\widehat{\beta}_n^{\,\PDW}-\beta_{\alpha_n}^*\big) \Big) \Big| \bigg) dt  \Bigg) \Big(u^\top\big(\f{X_i}\big)_S \Big)^2 \label{proof:3:lemmaspektralnormQssCovariancematrix}
		\end{align}
		because of (20). Now we can rearrange the term in the indicator function as
		\begin{align*}
			\Big| Y_i-\f{X_i^\top} \Big(\beta_{\alpha_n}^* + t \, \big(\widehat{\beta}_n^{\,\PDW}-\beta_{\alpha_n}^*\big) \Big) \Big| 
			&= \Big| \varepsilon_i + (1-t) \, \f{X_i^\top} \big( \beta^* - \beta_{\alpha_n}^* \big)  + t \, \f{X_i^\top} \big(\beta^* - \widehat{\beta}_n^{\,\PDW} \big)  \Big|.
		\end{align*}
		Using the inequality
		\begin{align*}
			\mathbbm{1}_{(\alpha_n^{-1/2},\infty)} \Big(\big| Q_1 + Q_2 + Q_3 \big|\Big) \leq \mathbbm{1}_{(\alpha_n^{-1/2}/3,\infty)} \Big(\big| Q_1 \big|\Big) + \mathbbm{1}_{(\alpha_n^{-1/2}/3,\infty)} \Big(\big| Q_2 \big|\Big) + \mathbbm{1}_{(\alpha_n^{-1/2}/3,\infty)} \Big(\big| Q_3 \big|\Big)
		\end{align*}
		for random variables $Q_1$, $Q_2$ and $Q_3$ leads to
		\begin{align}
			&\frac{2}{n} \sum_{i=1}^{n} \Bigg( \int_{0}^{1} \mathbbm{1}_{(\alpha_n^{-1/2},\infty)} \bigg( \Big| Y_i-\f{X_i^\top} \Big(\beta_{\alpha_n}^* + t \, \big(\widehat{\beta}_n^{\,\PDW}-\beta_{\alpha_n}^*\big) \Big) \Big| \bigg) dt \Bigg) \Big(u^\top\big(\f{X_i}\big)_S \Big)^2 \notag \\
			& \quad \quad \quad \quad \leq \frac{2}{n} \sum_{i=1}^{n} \Bigg( \int_{0}^{1} \bigg[ \mathbbm{1}_{(\alpha_n^{-1/2}/3,\infty)} \big(| \varepsilon_i |\big) + \mathbbm{1}_{(\alpha_n^{-1/2}/3,\infty)} \bigg(\Big|(1-t) \, \f{X_i^\top} \big( \beta^* - \beta_{\alpha_n}^* \big) \Big|\bigg) \notag \\
			& \quad \quad \quad \quad \quad \quad \quad \quad \quad \quad \quad \quad + \mathbbm{1}_{(\alpha_n^{-1/2}/3,\infty)} \bigg(\Big| t \, \f{X_i^\top} \big(\beta^* - \widehat{\beta}_n^{\,\PDW} \big) \Big|\bigg) \bigg] dt \Bigg) \Big(u^\top\big(\f{X_i}\big)_S \Big)^2 \notag \\
			& \quad \quad \quad \quad \leq \frac{2}{n} \sum_{i=1}^{n} \Bigg( \int_{0}^{1} \bigg[ \mathbbm{1}_{(\alpha_n^{-1/2}/3,\infty)} \big(| \varepsilon_i |\big) + \mathbbm{1}_{(\alpha_n^{-1/2}/3,\infty)} \bigg(\Big| \f{X_i^\top} \big( \beta^* - \beta_{\alpha_n}^* \big) \Big|\bigg) \notag \\
			& \quad \quad \quad \quad \quad \quad \quad \quad \quad \quad \quad \quad + \mathbbm{1}_{(\alpha_n^{-1/2}/3,\infty)} \bigg(\Big| \f{X_i^\top} \big(\beta^* - \widehat{\beta}_n^{\,\PDW} \big) \Big|\bigg) \bigg] dt \Bigg) \Big(u^\top\big(\f{X_i}\big)_S \Big)^2 \notag \\
			& \quad \quad \quad \quad = \frac{2}{n} \sum_{i=1}^{n} \mathbbm{1}_{(\alpha_n^{-1/2}/3,\infty)} \big(| \varepsilon_i | \big) \, \Big(u^\top\big(\f{X_i}\big)_S \Big)^2 \notag \\
			& \quad \quad \quad \quad \quad \quad  \quad \quad + \frac{2}{n} \sum_{i=1}^{n} \mathbbm{1}_{(\alpha_n^{-1/2}/3,\infty)} \bigg(\Big| \f{X_i^\top} \big( \beta^* - \beta_{\alpha_n}^* \big) \Big|\bigg) \, \Big(u^\top\big(\f{X_i}\big)_S \Big)^2 \notag \\
			& \quad \quad \quad \quad \quad \quad \quad \quad \quad \quad \quad \quad + \frac{2}{n} \sum_{i=1}^{n} \mathbbm{1}_{(\alpha_n^{-1/2}/3,\infty)} \bigg(\Big| \f{X_i^\top} \big(\beta^* - \widehat{\beta}_n^{\,\PDW} \big) \Big|\bigg) \, \Big(u^\top\big(\f{X_i}\big)_S \Big)^2 \,. \label{proof:4:lemmaspektralnormQssCovariancematrix}
		\end{align}
		We consider each of the three terms separately. By (iii) of Assumption \ref{assfan2017} we get for fixed $u \in \mathcal{B}_2^s$ that $u^\top\big(\f{X_i}\big)_S \sim \subg(\CXsub)$, and following the proof of \citet[Lemma 1.12]{Rigollet2019} together with $\big(\mathbbm{1}_{(\alpha_n^{-1/2}/3,\infty)} (| \varepsilon_i | )\big)^2 = \mathbbm{1}_{(\alpha_n^{-1/2}/3,\infty)} (| \varepsilon_i | )$ leads to
		\begin{align*}
			Q_i(u) = \mathbbm{1}_{(\alpha_n^{-1/2}/3,\infty)} \big(| \varepsilon_i | \big)& \, \Big(u^\top\big(\f{X_i}\big)_S \Big)^2 - \E \bigg[ \mathbbm{1}_{(\alpha_n^{-1/2}/3,\infty)} \big(| \varepsilon_i | \big) \, \Big(u^\top\big(\f{X_i}\big)_S \Big)^2 \bigg] \\
			&\sim \sube \big(16 \, \CXsubs, 16 \, \CXsubs \big) \,.
		\end{align*}
		Bernstein's inequality, cf. \citet[Theorem 1.13]{Rigollet2019}, implies
		\begin{align*}
			\Prob \Bigg(\bigg| \frac{2}{n} \sum_{i=1}^{n} Q_i (u) \bigg| > x \Bigg) \leq 2 \max \Bigg\{ \exp \bigg(-\frac{x^2\,n}{2048\,\CXsubf}\bigg), \exp\bigg(-\frac{x\,n}{64\,\CXsubs }\bigg) \Bigg\} 
		\end{align*}
		for $x>0$ and fixed $u \in \mathcal{B}_2^s$. Now we proceed with a covering argument. Consider a $1/8$-cover $A$ of cardinality $N = N(1/8; \mathcal{B}_2^s, \normz{\,\cdot\,}) \leq 24^s$  of the unit Euclidean ball of $\R^s$ with respect to the Euclidean distance (cf. Lemma 1.18 in \citet{Rigollet2019} or Example 5.8 in \citet{Wainwright2019}). We can argue similarly to the proof in \citet[Theorem 6.5]{Wainwright2019} since we consider also a quadratic form, and obtain
		for $x=256\,\sqrt{\log(24)}\,\CXsubs \, \big(\frac{s \log(p)}{n}\big)^\frac{1}{2}$ that
		\begin{align}
			\Prob \Bigg( \max_{u \in \mathcal{B}_2^s} \bigg| \frac{2}{n} \sum_{i=1}^{n} Q_i (u) \bigg| > x \Bigg) &\leq \Prob \Bigg( \max_{u \in A} \bigg| \frac{2}{n} \sum_{i=1}^{n} Q_i (u) \bigg| > \frac{x}{2} \Bigg) \notag \\
			&\leq 2 \, \big|N\big|\,\max \bigg\{ \exp \big(-8 \log(24)\,s \log(p)\big), \exp\Big(-\big(4 \log(24) \,s \log(p)\,  n\big)^\frac{1}{2}\Big) \bigg\} \notag \\
			&\leq 2 \exp\big( \log(24)\,s - 8  \log(24)\,s \log(p) \big) \notag \\
			&\leq 2 \exp\big( - 4 \log(24)\,s \log(p) \big) \notag \\
			&\leq \frac{2}{p^{5s}} \label{proof:8:lemmaspektralnormQssCovariancematrix}
		\end{align}
		since $4 \log(p) \geq 1$ if $p \geq 2$, and by assumption $n \geq 16 \log(24)\, s \log(p)$.
		%
		In addition we obtain
		\begin{align*} 
			\frac{2}{n} \sum_{i=1}^n \E \bigg[ \mathbbm{1}_{(\alpha_n^{-1/2}/3,\infty)} \big(| \varepsilon_i | \big) \, \Big(u^\top\big(\f{X_i}\big)_S \Big)^2 \bigg] &= 2\, \E \bigg[ \E \Big[ \mathbbm{1}_{(\alpha_n^{-1/2}/3,\infty)} \big(| \varepsilon_1 | \big) \Big| \f{X_1} \Big] \Big(u^\top\big(\f{X_1}\big)_S \Big)^2 \bigg] \\
			&\leq 2 \,\big( 1 + \Cm \big) \, \big( 9 \alpha_n \big)^\frac{\Cpm}{2} \, \E \bigg[ \Big(u^\top\big(\f{X_1}\big)_S \Big)^2 \bigg] \\
			&\leq 2 \,\big( 1 + \Cm \big) \, \CXsubs \,   \big( 9 \alpha_n \big)^\frac{\Cpm}{2}\,,
		\end{align*}
		by Assumption \ref{assfan2017} and an application of the conditional version of Markov's inequality,
		\begin{align*}
			\E \Big[ \mathbbm{1}_{(\alpha_n^{-1/2}/3,\infty)} \big(| \varepsilon_1 | \big) \Big| \f{X_1} \Big] &= \Prob\bigg(|\varepsilon_1| > \frac{1}{3\alpha_n^\frac{1}{2}} \bigg| \f{X_1} \bigg) \leq (9 \alpha_n)^\frac{\Cpm}{2} \, \E \Big[\E\big[|\varepsilon_1|^{\Cpm} \big| \f{X_1}\big] \Big] \\
			&\leq (9 \alpha_n)^\frac{\Cpm}{2} \, \bigg(1 + \E \Big[\E\big[|\varepsilon_1|^{\Cpm} \big| \f{X_1}\big]^\Choeld \Big] \bigg) \\
			&\leq \big( 1 + \Cm \big) \, \big( 9 \alpha_n \big)^\frac{\Cpm}{2} \,.
		\end{align*}
		By building the maximum of the expected values over $u \in \mathcal{B}_2^s$ and collecting terms we find that 
		\begin{align}
			\max_{u \in \mathcal{B}_2^s} \, \frac{2}{n} \sum_{i=1}^{n} \mathbbm{1}_{(\alpha_n^{-1}/3,\infty)} \big(| \varepsilon_i | \big) \, \Big(u^\top\big(\f{X_i}\big)_S \Big)^2 \leq 256\,\sqrt{\log(24)}&\,\CXsubs \, \bigg(\frac{s \log(p)}{n}\bigg)^\frac{1}{2} \notag \\
			&+ 2 \,\big( 1 + \Cm \big) \, \CXsubs \,   \big( 9 \alpha_n \big)^\frac{\Cpm}{2}  \label{proof:5:lemmaspektralnormQssCovariancematrix}
		\end{align}
		with probability at least $1 - 2/p^{5s}$. We proceed similar for the second and third sum in \eqref{proof:4:lemmaspektralnormQssCovariancematrix}, hence it is sufficient to consider the rates of the expected values
		\begin{align*}
			\E \bigg[ \mathbbm{1}_{(\alpha_n^{-1/2}/3,\infty)} \bigg(\Big| \f{X_1^\top} \big( \beta^* - \beta \big) \Big|\bigg) \, \Big(u^\top\big(\f{X_1}\big)_S \Big)^2 \bigg] 
		\end{align*} 
		with $\beta =  \beta_{\alpha_n}^*$ and $\beta= \widehat{\beta}_n^{\,\PDW}$. Obviously it is 
		\begin{align*}
			\mathbbm{1}_{(\alpha_n^{-1/2}/3,\infty)} \bigg(\Big| \f{X_1^\top} \big( \beta^* - \beta_{\alpha_n}^* \big) \Big|\bigg) \leq 3 \, \Big| \f{X_1^\top} \big( \beta^* - \beta_{\alpha_n}^* \big) \Big| \, \alpha_n^\frac{1}{2}
		\end{align*}
		and hence by Assumption \ref{assfan2017}, \citet[Lemma 1.4]{Rigollet2019} and the Cauchy-Schwarz inequality 
		\begin{align*}
			\E \bigg[ \mathbbm{1}_{(\alpha_n^{-1/2}/3,\infty)} \bigg(\Big| \f{X_1^\top} \big( \beta^* - \beta_{\alpha_n}^* \big) \Big|\bigg) \, \Big(u^\top\big(\f{X_1}\big)_S \Big)^2 \bigg] &\leq 3  \alpha_n^\frac{1}{2} \, \E \bigg[ \Big| \f{X_1^\top} \big( \beta^* - \beta_{\alpha_n}^* \big) \Big| \, \Big(u^\top\big(\f{X_1}\big)_S \Big)^2 \bigg] \\
			&\leq 3  \alpha_n^\frac{1}{2} \, \Bigg( \E \bigg[ \Big( \f{X_1^\top} \big( \beta^* - \beta_{\alpha_n}^* \big) \Big)^2 \bigg] \, \E \bigg[ \Big(u^\top\big(\f{X_1}\big)_S \Big)^4 \bigg] \Bigg)^\frac{1}{2} \\
			&\leq 12  \, \CXsubt\,\normz{\beta^* - \beta_{\alpha_n}^*} \, \alpha_n^\frac{1}{2} \,.
		\end{align*}
		Lemma \ref{approximation:error:pseudo:huber} implies 
		\begin{align}
			\E \bigg[ \mathbbm{1}_{(\alpha_n^{-1/2}/3,\infty)} \bigg(\Big| \f{X_1^\top} \big( \beta^* - \beta_{\alpha_n}^* \big) \Big|\bigg) \, \Big(u^\top\big(\f{X_1}\big)_S \Big)^2 \bigg] \leq 12  \, \Capprox \, \CXsubt \, \alpha_n^{\Cpm-\frac{1}{2}} \,. \label{proof:6:lemmaspektralnormQssCovariancematrix}
		\end{align} 
		The vector $\widehat{\beta}_n^{\,\PDW}$ has support $S$ and satisfies $\big\lVert \widehat{\beta}_n^{\,\PDW}  - \beta^* \big\rVert_2 \leq 2 \Cbeta$ by (27) and (iv) of Assumption \ref{assfan2017}, hence it follows that
		\begin{align*}
			\E \Bigg[ \mathbbm{1}_{(\alpha_n^{-1/2}/3,\infty)} \bigg(\Big| \f{X_1^\top} \big( \beta^* - \widehat{\beta}_n^{\,\PDW} \big) \Big|\bigg) \Bigg] &= \Prob \bigg( \Big| \f{X_1^\top} \big( \beta^* - \widehat{\beta}_n^{\,\PDW} \big) \Big| > \frac{1}{3 \alpha_n^\frac{1}{2}}\bigg) \\
			&\leq \Prob \bigg( \max_{u \in \R^{s} : \normz{u} \leq 2 \Cbeta} \Big|u^\top \big(\f{X_1}\big)_S\Big| > \frac{1}{3 \alpha_n^\frac{1}{2}}\bigg) \\
			&=  \Prob \bigg( \max_{u \in \mathcal{B}_2^s} \, \bigg|u^\top \bigg( 2 \Cbeta \big(\f{X_1}\big)_S \Big) \bigg| > \frac{1}{3 \alpha_n^\frac{1}{2}}\bigg)\\
			&\leq \exp \bigg( \log(6) \, s - \frac{1}{288 \,\Cbetas \,\CXsubs\,\alpha_n} \bigg)
		\end{align*}
		by \citet[Theorem 1.19]{Rigollet2019} together with Assumption \ref{assfan2017}. By the choice of $\alpha_n$ and the sample size $n$ we obtain
		\begin{align*}
			\exp \bigg( \log(6) \, s - \frac{1}{288 \,\Cbetas \,\CXsubs\,\alpha_n} \bigg) &= \exp \bigg( \log(6) \, s - \frac{\sqrt{n}}{576 \,\Cbetas \,\CXsubs\,\Calpha \,\sqrt{\log(p)}} - \frac{1}{576 \,\Cbetas \,\CXsubs\,\alpha_n}\bigg)\\
			&\leq \exp \big( \log(6) \, s - \log(6) \, s\big) \, \exp \bigg(  - \frac{\alpha_n^{-1}}{576 \,\Cbetas \,\CXsubs}\bigg) \\
			&= 2 \, \big(576 \,\Cbetas \,\CXsubs\,\alpha_n\big)^2 
		\end{align*}
		since $\exp(x) \geq x^2/2$ for $x>0$. Therefore
		\begin{align}
			\E \bigg[ \mathbbm{1}_{(\alpha_n^{-1/2}/3,\infty)} \bigg(\Big| \f{X_1^\top} \big( \beta^* - \widehat{\beta}_n^{\,\PDW} \big) \Big|\bigg) \, \Big(u^\top\big(\f{X_1}\big)_S \Big)^2 \bigg] \leq \sqrt{2} \,2304 \,\Cbetas \,\CXsubf\,\alpha_n \label{proof:7:lemmaspektralnormQssCovariancematrix}
		\end{align}
		by the Cauchy-Schwarz inequality.
		So finally the previous considerations in \eqref{proof:1:lemmaspektralnormQssCovariancematrix} - \eqref{proof:7:lemmaspektralnormQssCovariancematrix} showed that 
		\begin{align*}
			\normzM{\frac{2}{n} \sum_{i=1}^{n} \,\big(\f{X_i}\big)_S \big(\f{X_i}\big)_S^\top - \widehat{Q}_{SS}} &\leq 768\,\sqrt{\log(24)}\,\CXsubs \, \bigg(\frac{s \log(p)}{n}\bigg)^\frac{1}{2} + 2 \,\big( 1 + \Cm \big) \, \CXsubs \,   \big( 9 \alpha_n \big)^\frac{\Cpm}{2} \\
			& \quad \quad + 24  \, \Capprox \, \CXsubt \, \alpha_n^{\Cpm-\frac{1}{2}} + \Big( 3 \big( \CXu + 3 C_4 \big) + \sqrt{2} \,4608 \,\Cbetas \,\CXsubf \Big) \alpha_n \\
			&\leq C_5 \, \max\Bigg\{ \bigg(\frac{s \log(p)}{n}\bigg)^\frac{1}{2} \,,\, \alpha_n^\frac{\Cpm}{2} \,,\, \alpha_n^{\Cpm-\frac{1}{2}} \,,\, \alpha_n \Bigg\}
		\end{align*}
		for a positive constant $C_5>0$ with probability at least $1-C_1/ p^{2} - 6/p^{5s}$. Furthermore, repeated application of the spectral norm bound in \eqref{spektral:norm:bound:sample:covariance} leads to
		\begin{align*}
			\normzM{\widehat{Q}_{SS} - \E\big[\f{X_1}\f{X_1^\top}\big]_{SS}} &\leq \normzM{ \widehat{Q}_{SS} - \frac{2}{n} \sum_{i=1}^{n} \,\big(\f{X_i}\big)_S \big(\f{X_i}\big)_S^\top} \\
			& \quad \quad \quad \quad \quad + 2 \,\normzM{\frac{1}{n} \sum_{i=1}^{n} \,\big(\f{X_i}\big)_S \big(\f{X_i}\big)_S^\top - \E\big[\f{X_1}\f{X_1^\top}\big]_{SS}}\\
			&\leq C_2 \, \max\Bigg\{ \bigg(\frac{s}{n}\bigg)^\frac{1}{2} \,,\, \frac{s}{n} \,,\, \bigg(\frac{\log(p)}{n}\bigg)^\frac{1}{2} \,,\, \bigg(\frac{s \log(p)}{n}\bigg)^\frac{1}{2} \,,\, \alpha_n^\frac{\Cpm}{2} \,,\, \alpha_n^{\Cpm-\frac{1}{2}} \,,\, \alpha_n \Bigg\} 
		\end{align*}
		for a positive constant $C_2>0$. By the choices of $\alpha_n$ and $n \gtrsim  s^2 \log(p)$ together with $\Cpm \in \{2,3\}$ finally it follows that
		\begin{align*}
			\normzM{\widehat{Q}_{SS} - \E\big[\f{X_1}\f{X_1^\top}\big]_{SS}} \leq C_6 \, \bigg(\frac{s \log(p)}{n}\bigg)^\frac{1}{2} \leq \frac{C_3}{\sqrt{s}}
		\end{align*}
		for some positive constants $C_3, C_6>0$ with probability at least $1-C_1/ p^{2} - 6/p^{5s}$.
	\end{proof}

	\begin{lemma} \label{lemmamatrixgradient} 
		Let $M \in \R^{|A| \times |B|}$ be a matrix with $A,B \subseteq \{1,\dotsc,p\}$ and $\max_{k \in \{1,\dotsc,|A|\}} \normz{M^\top e_k} \leq C_M$ for some positive constant $C_M > 0 $. Suppose Assumption \ref{assfan2017} and $\alpha_n \geq \Cgradf \, \big(\frac{\log(p)}{n}\big)^{\frac{1}{2}}$ holds, then with probability at least $1-2/p^2$ the $\ell_\infty$ norm of $M  \big( \nabla \loss_{n,\alpha_n}^{\,\Hu} (\beta_{\alpha_n}^*) \big)_B$ is bounded by 
		\begin{align}
			\normi{M  \Big( \nabla \loss_{n,\alpha_n}^{\,\Hu} \big(\beta_{\alpha_n}^*\big) \Big)_B}  \leq C_M \,\Cgrads \, \bigg( \frac{\log(p)}{n}\bigg)^\frac{1}{2} \label{defsetT1} \,.
		\end{align}
	\end{lemma}
	
	\begin{proof}[Proof of Lemma \ref{lemmamatrixgradient}] 
		We follow the proof of Lemma \ref{ratelinftygradient}. 
		%
		%
		It is 
		\begin{align*}
			M \Big( \nabla \loss_{n,\alpha_n}^{\,\Hu} \big(\beta_{\alpha_n}^*\big) \Big)_B = M \bigg(- \frac{1}{n} \sum_{i=1}^{n} l_{\alpha_n}'\big(Y_i-\f{X_i^\top} \beta_{\alpha_n}^* \big) \big(\f{X_i}\big)_B \bigg) = - \frac{1}{n} \sum_{i=1}^n l_{\alpha_n}'\big(Y_i-\f{X_i^\top} \beta_{\alpha_n}^* \big) \f{Z_i} 
		\end{align*}
		with $\f{Z_i} = M \big(\f{X_i}\big)_B$. The random vectors $l_{\alpha_n}'\big(Y_1-\f{X_1^\top} \beta_{\alpha_n}^* \big)  \f{Z_1},\dotsc, l_{\alpha_n}'\big(Y_n-\f{X_n^\top} \beta_{\alpha_n}^* \big)  \f{Z_n} $ are independent and identically distributed because $(\f{X_1},\varepsilon_1),\dotsc, (\f{X_n},\varepsilon_n)$ are independent and identically distributed. In addition (iii) of Assumption \ref{assfan2017} and $\max_{k \in \{1,\dotsc,|A|\}} \normz{M^\top e_k} \leq C_M$ imply that the entries $ Z_{i,k} = e_k^\top \f{Z_i}$ of $\f{Z_i}$ are sub-Gaussian with variance proxy $ C_M^2\,\CXsubs$. This leads to 
		\begin{align*}
			\E\bigg[ \Big( l_{\alpha_n}'\big(Y_i-\f{X_i^\top} \beta_{\alpha_n}^* \big) Z_{i,k} \Big)^2 \bigg] \leq C_M^2 \,\Cgradt 
		\end{align*}
		and 
		\begin{align*}
			\E\bigg[ \Big| l_{\alpha_n}'\big(Y_i-\f{X_i^\top} \beta_{\alpha_n}^* \big) \, Z_{i,k} \Big|^u \bigg] \leq \frac{u!}{2} \, \bigg(\frac{2\,C_M\,\Cgradv}{\alpha_n} \bigg)^{u-2} \, \Cgradt 
		\end{align*}
		for $u \in \N$, $u \geq 3$, where $\Cgradt$ and $\Cgradv$ are given in the proof of Lemma \ref{ratelinftygradient}. Moreover, we obtain
		\begin{align*}
			\E \Big[ l_{\alpha_n}'\big(Y_1-\f{X_1^\top} \beta_{\alpha_n}^* \big) \f{Z_1} \Big] = M \, \E \Big[ l_{\alpha_n}'\big(Y_1-\f{X_1^\top} \beta_{\alpha_n}^* \big) \big(\f{X_1}\big)_B \Big] = \f{0}_{|A|}
		\end{align*}
		since $\E \big[ l_{\alpha_n}'\big(Y_1-\f{X_1^\top} \beta_{\alpha_n}^* \big) \f{X_1} \big] = \f{0}_{p}$ (see proof of Lemma \ref{ratelinftygradient}).  Arguing as in the proof of Lemma \ref{ratelinftygradient} concludes the proof.
		%
		%
	\end{proof}
	\begin{proof}[Proof of Lemma \ref{lemmaassstrictdual(ii)}]\hfill\\
		For the first part we invoke Lemma \ref{lemmaspektralnormQssCovariancematrix} and obtain (if $C_3 \geq \max \big\{ (576\,\log(6) \, \Calpha \, \Cbetas \,\CXsubs)^2, 16 \log(24)\big\}$ in Lemma \ref{lemmaassstrictdual(ii)})
		\begin{align*}
			\normzM{\widehat{Q}_{SS} - \E\big[\f{X_1}\f{X_1^\top}\big]_{SS}} \leq \frac{C_4}{\sqrt{s}}
		\end{align*}
		with probability at least $1-C_1/ p^{2} - 6/p^{5s}$ for some positive constants $C_1,C_4>0$. Moreover, we have 
		\begin{align}
			\normzM{\Big(\E\big[\f{X_1}\f{X_1^\top}\big]_{SS}\Big)^{-1}} \leq \normiM{\Big(\E\big[\f{X_1}\f{X_1^\top}\big]_{SS}\Big)^{-1}} \leq \Csminf \label{prooflemmaassstrictdual(ii)4}
		\end{align}
		by (10) and the symmetry of the matrix. Hence by \citet[Lemma 11]{Loh2017} we conclude that 
		\begin{align}
			\normzM{\big(\widehat{Q}_{S S}\big)^{-1} - \Big(\E\big[\f{X_1}\f{X_1^\top}\big]_{SS}\Big)^{-1}} &\leq 2 \, \Csminfs \, \normzM{\widehat{Q}_{SS} - \E\big[\f{X_1}\f{X_1^\top}\big]_{SS}} \leq \frac{2 \, C_4\,\Csminfs}{\sqrt{s}} \,, \label{prooflemmaassstrictdual(ii)2}
		\end{align}
		with high probability if $\sqrt{s} \geq 2 \,C_4 \,\Csminf$. 
		Finally the triangle inequality and once again (10) lead to
		\begin{align*}
			\normiM{\big(\widehat{Q}_{S S}\big)^{-1}} &\leq \normiM{\Big(\E\big[\f{X_1}\f{X_1^\top}\big]_{SS}\Big)^{-1}} + \normiM{\big(\widehat{Q}_{S S}\big)^{-1} - \Big(\E\big[\f{X_1}\f{X_1^\top}\big]_{SS}\Big)^{-1}} \\
			&\leq \Csminf \, + \sqrt{s}\,\normzM{\big(\widehat{Q}_{S S}\big)^{-1} - \Big(\E\big[\f{X_1}\f{X_1^\top}\big]_{SS}\Big)^{-1}}\\
			&\leq \Csminf + 2 \, C_4\,\Csminfs
		\end{align*}
		with probability at least $1-C_1/ p^{2} - 6/p^{5s}$.\\
		
		To prove the second part of this lemma we follow the inequalities
		\begin{align}
			&\normi{\widehat{Q}_{S^c S} \, \big(\widehat{Q}_{S S}\big)^{-1} \Big(\nabla \loss_{n,\alpha_n}^{\,\Hu} \big(\beta_{\alpha_n}^*\big)\Big)_{S} } \notag \\
			& \quad \quad \quad \quad \leq \normi{ \E\big[\f{X_1}\f{X_1^\top}\big]_{S^c S} \, \Big(\E\big[\f{X_1}\f{X_1^\top}\big]_{S S}\Big)^{-1} \Big(\nabla \loss_{n,\alpha_n}^{\,\Hu} \big(\beta_{\alpha_n}^*\big)\Big)_{S} } \notag \\
			& \quad \quad \quad \quad \quad \quad \quad \quad + \normi{ \bigg( \widehat{Q}_{S^c S} \, \big(\widehat{Q}_{S S}\big)^{-1} - \E\big[\f{X_1}\f{X_1^\top}\big]_{S^c S} \, \Big(\E\big[\f{X_1}\f{X_1^\top}\big]_{S S}\Big)^{-1} \bigg) \Big(\nabla \loss_{n,\alpha_n}^{\,\Hu} \big(\beta_{\alpha_n}^*\big)\Big)_{S} }\label{eq:estsimilarwainloh}
		\end{align}
		and 
		{ \scriptsize
			\begin{align}
				&\normi{ \bigg( \widehat{Q}_{S^c S} \, \big(\widehat{Q}_{S S}\big)^{-1} - \E\big[\f{X_1}\f{X_1^\top}\big]_{S^c S} \, \Big(\E\big[\f{X_1}\f{X_1^\top}\big]_{S S}\Big)^{-1} \Bigg) \Big(\nabla \loss_{n,\alpha_n}^{\,\Hu} \big(\beta_{\alpha_n}^*\big)\Big)_{S} } \notag \\
				& \quad \quad \quad \quad \leq \max_{ k \in \{1,\dotsc,p-s\} } \normz{ \Bigg( e_k^\top \bigg( \widehat{Q}_{S^c S} \, \big(\widehat{Q}_{S S}\big)^{-1} - \E\big[\f{X_1}\f{X_1^\top}\big]_{S^c S} \, \Big(\E\big[\f{X_1}\f{X_1^\top}\big]_{S S}\Big)^{-1} \bigg) \Bigg)^\top } \, \normz{\Big(\nabla \loss_{n,\alpha_n}^{\,\Hu} \big(\beta_{\alpha_n}^*\big)\Big)_{S}} \notag \\
				& \quad \quad \quad \quad \leq \max_{ k \in \{1,\dotsc,p-s\} } \Bigg( \normz{ \Big( e_k^\top \,\E\big[\f{X_1}\f{X_1^\top}\big]_{S^c S} \, \Delta_1 \Big)^\top } + \normz{ \bigg( e_k^\top \, \Delta_2^\top \, \Big(\E\big[\f{X_1}\f{X_1^\top}\big]_{S S}\Big)^{-1} \bigg)^\top} + \normz{ \Big( e_k^\top \Delta_2^\top\, \Delta_1 \Big)^\top} \Bigg) \notag \\
				& \quad \quad \quad \quad \quad \quad \quad \quad \quad \quad \quad \quad \cdot \normz{\Big(\nabla \loss_{n,\alpha_n}^{\,\Hu} \big(\beta_{\alpha_n}^*\big)\Big)_{S}} \notag \\
				& \quad \quad \quad \quad \leq \max_{ k \in \{1,\dotsc,p-s\} } \Bigg( \normzM{\Delta_1} \,\normz{ \,\E\big[\f{X_1}\f{X_1^\top}\big]_{ S S^c} \, e_k  } + \normzM{\Big(\E\big[\f{X_1}\f{X_1^\top}\big]_{S S}\Big)^{-1}} \,\normz{ \Delta_2 \, e_k} + \normzM{\Delta_1} \, \normz{ \Delta_2 \, e_k}\Bigg) \notag \\
				& \quad \quad \quad \quad \quad \quad \quad \quad \quad \quad \quad \quad \cdot \normz{\Big(\nabla \loss_{n,\alpha_n}^{\,\Hu} \big(\beta_{\alpha_n}^*\big)\Big)_{S}} \label{prooflemmaassstrictdual(ii)3}
		\end{align} }
		with
		\begin{align*}
			\Delta_1 = \big(\widehat{Q}_{S S}\big)^{-1} -  \Big(\E\big[\f{X_1}\f{X_1^\top}\big]_{S S}\Big)^{-1} \quad \text{and} \quad  \Delta_2 = \widehat{Q}_{S S^c } - \E\big[\f{X_1}\f{X_1^\top}\big]_{S S^c}
		\end{align*}
		in \citet[Corollary 3]{Loh2017}. Note that \eqref{prooflemmaassstrictdual(ii)2} implies $\normzM{\Delta_1} \leq 2 \, C_4\,\Csminfs / \sqrt{s}$. For the first term in \eqref{eq:estsimilarwainloh} we shall apply Lemma \ref{lemmamatrixgradient} with $M = \E\big[\f{X_1}\f{X_1^\top}\big]_{S^c S} \, \big(\E\big[\f{X_1}\f{X_1^\top}\big]_{S S}\big)^{-1}$. We obtain
		\begin{align*}
			\max_{k \in \{1,\dotsc,p-s\}} \normz{ \E\big[\f{X_1}\f{X_1^\top}\big]_{S S^c} \, e_k} \leq \max_{k \in S^c} \normz{\E\big[\f{X_1}\f{X_1^\top}\big] \, e_k} \leq \max_{ u \in \R^p ,\normz{u}=1 } \normz{\E\big[\f{X_1}\f{X_1^\top}\big] \, u} \leq \CXu 
		\end{align*}
		by (ii) of Assumption \ref{assfan2017}, and hence together with \eqref{prooflemmaassstrictdual(ii)4} the estimate
		\begin{align*}
			\max_{k \in \{1,\dotsc,p-s\}} &\normz{\Big(\E\big[\f{X_1}\f{X_1^\top}\big]_{S S}\Big)^{-1} \, \E\big[\f{X_1}\f{X_1^\top}\big]_{S S^c} \, e_k} \\
			&  \leq \max_{k \in \{1,\dotsc,p-s\}} \normzM{\Big(\E\big[\f{X_1}\f{X_1^\top}\big]_{S S}\Big)^{-1}} \, \normz{\E\big[\f{X_1}\f{X_1^\top}\big]_{S S^c} \, e_k} \\
			& \leq  \Csminf\, \CXu \,.
		\end{align*}
		Lemma \ref{lemmamatrixgradient} and the choice of $\alpha_n$ in \eqref{rangealphaalhwithinitial} lead to
		\begin{align}
			\normi{ \E\big[\f{X_1}\f{X_1^\top}\big]_{S^c S} \, \Big(\E\big[\f{X_1}\f{X_1^\top}\big]_{S S}\Big)^{-1} \Big(\nabla \loss_{n,\alpha_n}^{\,\Hu} \big(\beta_{\alpha_n}^*\big)\Big)_{S} } \leq \Csminf\, \CXu \,\Cgrads \, \bigg( \frac{\log(p)}{n}\bigg)^\frac{1}{2} \label{prooflemmaassstrictdual(ii)6}
		\end{align}
		with probability at least $1 - 2/p^{2}$. In addition we get by Lemma \ref{ratelinftygradient} also
		\begin{align}
			\normz{\Big(\loss_{n,\alpha_n}^{\,\Hu} \big(\beta_{\alpha_n}^*\big)\Big)_{S}} \leq \sqrt{s} \, \normi{\Big(\loss_{n,\alpha_n}^{\,\Hu} \big(\beta_{\alpha_n}^*\big)\Big)_{S}} \leq \Cgrads \bigg(\frac{s \log(p)}{n}\bigg)^\frac{1}{2} \label{prooflemmaassstrictdual(ii)7}
		\end{align}
		with the same probability. The final task is now to study the rate of $\max_{k \in \{1,\dotsc,p-s\}} \normz{ \Delta_2 \, e_k } $. First of all it is
		\begin{align*}
			\max_{k \in \{1,\dotsc,p-s\}} \normz{  \Big(\widehat{Q}_{ S S^c} - \E\big[\f{X_1}\f{X_1^\top}\big]_{S S^c}\Big) \, e_k} &\leq \sqrt{s} \max_{k \in \{1,\dotsc,p-s\}} \normi{  \Big(\widehat{Q}_{ S S^c} - \E\big[\f{X_1}\f{X_1^\top}\big]_{S S^c}\Big) \, e_k} \\
			&= \sqrt{s} \max_{ \substack{l \in \{1,\dotsc,s\}, \\ k \in \{1,\dotsc,p-s\} }}   \bigg|  e_l^\top \Big(\widehat{Q}_{ S S^c} - \E\big[\f{X_1}\f{X_1^\top}\big]_{S S^c}\Big) \, e_k \bigg| \\
			&\leq \sqrt{s} \max_{ k,l \in \{1,\dotsc,p\} }   \bigg|  e_l^\top \Big(\widehat{Q}
			- \E\big[\f{X_1}\f{X_1^\top}\big] \Big) \, e_k \bigg| \,.
		\end{align*}
		We proceed similar to the proof of Lemma \ref{lemmaspektralnormQssCovariancematrix} but here we have only the maximum over $p^2$ elements in comparison to the $24^s$ elements in the mentioned proof. In addition we use the fact that the centered product of two sub-Gaussian random variables is sub-Exponential, cf. \citet[ Lemma 2.7.7]{Vershynin2018}, and that also the centered product of two sub-Gaussian random variables and a bounded random variable is sub-Exponential. Hence we don't have the rates depending on $s$ in \eqref{spektral:norm:bound:sample:covariance} and in \eqref{proof:8:lemmaspektralnormQssCovariancematrix} the factor $s$ can be dropped.
		It follows that there exist positive constants $C_2,C_5,C_6 >0 $ such that
		\begin{align*}
			\max_{ k,l \in \{1,\dotsc,p\} }   \bigg|  e_l^\top \Big(\widehat{Q}
			- \E\big[\f{X_1}\f{X_1^\top}\big] \Big) \, e_k \bigg|  \leq C_5 \, \max\Bigg\{ \bigg(\frac{\log(p)}{n}\bigg)^\frac{1}{2}  \,,\, \alpha_n^\frac{\Cpm}{2} \,,\, \alpha_n^{\Cpm-\frac{1}{2}} \,,\, \alpha_n \Bigg\} \leq \frac{C_6}{s}
		\end{align*}
		with probability at least $1 - C_2/p^2$ by the choices of $\alpha_n$ in (11) and $n \gtrsim  s^2 \log(p)$ together with $\Cpm \in \{2,3\}$. Hence
		\begin{align}
			\max_{k \in \{1,\dotsc,p-s\}} \normz{  \Delta_2 \, e_k } \leq \frac{C_6}{\sqrt{s}} \label{prooflemmaassstrictdual(ii)9}
		\end{align}
		with high probability and in total we obtain by \eqref{eq:estsimilarwainloh} - \eqref{prooflemmaassstrictdual(ii)9} the inequality
		\begin{align*}
			\normi{\widehat{Q}_{S^c S} \, \big(\widehat{Q}_{S S}\big)^{-1} \Big(\nabla \loss_{n,\alpha_n}^{\,\Hu} \big(\beta_{\alpha_n}^*\big)\Big)_{S} } &\leq \Csminf\, \CXu \,\Cgrads \, \bigg( \frac{\log(p)}{n}\bigg)^\frac{1}{2} + \Cgrads \bigg(\frac{s \log(p)}{n}\bigg)^\frac{1}{2} \\
			& \quad \quad \quad \cdot \bigg( \frac{2 \, C_4\,\CXu\,\Csminfs}{\sqrt{s}} + \frac{C_6 \, \Csminf}{\sqrt{s}} + \frac{2 \, C_4\,C_6\,\Csminfs}{s}\bigg)\\
			&\leq C_7 \,\Cgrads\,\bigg(\frac{\log(p)}{n}\bigg)^\frac{1}{2}
		\end{align*}
		with probability at least $1 - (4+C_1+C_2)/ p^{2} - 6/p^{5s}$ for some positive constant $C_7>0$. Renewed application of Lemma \ref{ratelinftygradient} and the triangular inequality lead to
		\begin{align*}
			\normi{\widehat{Q}_{S^c S} \, \big(\widehat{Q}_{S S}\big)^{-1} \Big(\nabla \loss_{n,\alpha_n}^{\,\Hu} \big(\beta_{\alpha_n}^*\big)\Big)_{S} - \Big(\nabla \loss_{n,\alpha_n}^{\,\Hu} \big(\beta_{\alpha_n}^*\big)\Big)_{S^c}} \leq \big(1 + C_7 \big) \, \Cgrads\, \bigg(\frac{\log(p)}{n}\bigg)^\frac{1}{2} \,.
		\end{align*}
	\end{proof}

	\subsection{Proofs of Lemmas \ref{lemmaassstrictdual(i)} and \ref{rateofrnS}}\label{sec:lemmaassstrictdual(i)}

	We start with proving Lemma \ref{rateofrnS}. For this purpose we need a technical result concerning the column normalization of the design matrix $\X_n$.

	\begin{lemma} \label{lemmacolumnnormalization}
		Let $\X_n = \big(\f{X_1},\dotsc, \f{X_n}\big)^\top \in \R^{n \times p}$ be a matrix with independent and identically distributed rows $\f{X_i} \sim \subg_p(\CXsub) $ with variance proxy $\CXsubs > 0$. Then for $n \geq 6 \log(p)$ the columns $\vec{X}_k$ of $\X_n$ satisfy with probability at least $1-2/p^2$ 
		\begin{align}
			\frac{1}{n} \max_{k \in \{1,\dotsc,p\}} \normzq{\vec{X}_k} \leq 17 \, \CXsubs \,. 
		\end{align}
	\end{lemma}
	
	\begin{proof}
		We have $X_{i,k} = e_k^\top \, \f{X_i} \sim \subg(\CXsub)$ for all $i=1,\dotsc,n$ and $k=1,\dotsc,p$ by the definition of a sub-Gaussian random vector. \citet[Lemma 1.12]{Rigollet2019} implies $X_{i,k}^2 - \E\big[X_{i,k}^2\big] \sim \sube(16 \,\CXsubs, 16 \, \CXsubs)$ and with Bernstein's inequality, cf. \citet[Theorem 1.13]{Rigollet2019}, it follows that
		\begin{align*}
			\Prob \Bigg(\bigg|\frac{1}{n}	\sum_{i=1}^{n} \Big(X_{i,k}^2 - \E\big[X_{i,k}^2\big] \Big) \bigg| > x\Bigg) \leq 2 \max \Bigg\{ \exp \bigg(-\frac{x^2\,n}{512 \, \CXsubf}\bigg), \exp\bigg(-\frac{x\,n}{32 \, \CXsubs }\bigg) \Bigg\} \,.
		\end{align*}
		for all $x>0$ and $k=1,\dotsc,p$ since $X_{1,k}, \dotsc, X_{n,k}$ are independent and identically distributed.
		By the union bound and the condition $n \geq 6 \log(p)$ we obtain
		\begin{align*}
			\Prob \Bigg( \max_{k \in \{1,\dotsc,p\}} \bigg|\frac{1}{n}	&\sum_{i=1}^{n} \bigg(\big(e_k^\top \, \f{X_i}\big)^2 - \E\Big[\big(e_k^\top \, \f{X_i}\big)^2\Big] \bigg) \bigg| > 16 \,\CXsubs \Bigg) \leq 2 \, p \exp\bigg(-\frac{n}{2}\bigg) \leq \frac{2}{p^2} \,.
		\end{align*}
		Furthermore, we have for all $k =1,\dotsc,p$ the estimate 
		\begin{align*}
			\frac{1}{n} \sum_{i=1}^{n} \E\big[X_{i,k}^2\big] = \E\big[X_{1,k}^2\big] \leq \CXsubs
		\end{align*}
		since $X_{1,k}$ is sub-Gaussian with variance proxy $\CXsubs$, and therefore we get 
		\begin{align*}
			\max_{k \in \{1,\dotsc,p\}} \frac{1}{n} \, \normzq{\vec{X}_k} 
			&\leq \max_{k \in \{1,\dotsc,p\}} \bigg|\frac{1}{n}	\sum_{i=1}^{n} \Big(X_{i,k}^2 - \E\big[X_{i,k}^2\big] \Big) \bigg| + \max_{k \in \{1,\dotsc,p\}} \frac{1}{n} \sum_{i=1}^{n} \E\big[X_{i,k}^2\big] \\
			&\leq 16 \, \CXsubs + \CXsubs = 17 \, \CXsubs
		\end{align*}
		with high probability.
	\end{proof}

	\begin{proof}[Proof of Lemma \ref{rateofrnS}]
		We follow the proof of \citet[Lemma 10.3]{Zhou2009}. It is
		\begin{align*}
			\widehat{Q}_{S^c S} \big( \widehat{Q}_{S S} \big)^{-1}  =  \frac{2}{n} \, \X_{n,S^c}^\top\, D \, \X_{n,S} \bigg( \frac{2}{n} \, \X_{n,S}^\top\, D \, \X_{n,S} \bigg)^{-1} =  \X_{n,S^c}^\top\, D \, \X_{n,S} \big( \X_{n,S}^\top\, D \, \X_{n,S} \big)^{-1} \,,
		\end{align*}
		see Lemma \ref{lem:derivativelem}. For $k \in S^c$ let
		\begin{align*}
			r_k = \big( \X_{n,S}^\top\, D \, \X_{n,S} \big)^{-1} \X_{n,S}^\top \, D \, \vec{X}_k \quad \quad \in \R^s\,,
		\end{align*}
		then we have
		\begin{align*}
			\normiM{\widehat{Q}_{S^c S} \big( \widehat{Q}_{S S} \big)^{-1}} = \max_{k \in S^c} \norme{r_k}\,.
		\end{align*}
		Furthermore, on the one hand the column normalization in Lemma \ref{lemmacolumnnormalization} under the condition $n \geq 6 \log(p) $ and the submultiplicativity of the spectral norm lead to
		\begin{align*}
			\max_{k \in S^c} \normz{ D^\frac{1}{2} \, \X_{n,S}\,r_k} 
			&\leq \max_{k \in S^c} \bigg( \normzM{D^\frac{1}{2} \,\X_{n,S}\,\big( \X_{n,S}^\top\, D \, \X_{n,S} \big)^{-1} \X_{n,S}^\top\, D^\frac{1}{2} } \,\normzM{D^\frac{1}{2}}\, \normz{\vec{X}_k} \bigg) \\
			&\leq \max_{k \in S^c} \, \normz{\vec{X}_k} \leq \sqrt{17} \,\CXsub \, \sqrt{n}
		\end{align*}
		with probability at least $1-2/p^2$ since $D^\frac{1}{2}\,\X_{n,S}\,\big( \X_{n,S}^\top\, D \, \X_{n,S} \big)^{-1} \X_{n,S}^\top\, D^\frac{1}{2} $ is an orthogonal projection matrix and $D$ a diagonal matrix with entries smaller than or equal to $1$. On the other hand under the condition $n \geq \Crsct s \log(p)$ the smallest eigenvalue of $\widehat{Q}_{S S} = \frac{2}{n} \, \X_{n,S}^\top\, D \, \X_{n,S}$ is bounded below by $\CXl/32$ with probability at least $1 - \cPo \exp(-\cPo n)$, see Lemma 6, which implies
		\begin{align*}
			\normzq{ D^\frac{1}{2} \, \X_{n,S}\,r_k} &= r_k^\top \, \X_{n,S}^\top \, D \, \X_{n,S} \, r_k = \frac{n}{2} \,r_k^\top \bigg( \frac{2}{n} \, \X_{n,S}^\top \, D \, \X_{n,S} \bigg) r_k \geq \frac{\CXl}{64} \, \normzq{r_k}\,n
		\end{align*}
		for all $k \in S^c$. Hence we obtain by the last inequalities the estimate
		\begin{align*}
			\max_{k \in S^c} \normz{r_k} \leq  \max_{k \in S^c} \bigg(\frac{64}{\CXl\,n}\bigg)^\frac{1}{2} \normz{ D^\frac{1}{2} \, \X_{n,S}\,r_k} \leq \frac{33\,\CXsub}{\sqrt{\CXl}}
		\end{align*}
		and in total
		\begin{align*}
			\normiM{\widehat{Q}_{S^c S} \big( \widehat{Q}_{S S} \big)^{-1}} = \max_{k \in S^c} \norme{r_k} \leq \max_{k \in S^c} \sqrt{s} \, \normz{r_k} \leq \frac{33\,\CXsub\,\sqrt{s}}{\sqrt{\CXl}}
		\end{align*}
		with high probability.
	\end{proof}

	Lemma \ref{lemmaassstrictdual(i)} immediately follows from the following lemma and Lemma \ref{ratelinftygradient}.
	%
	
	\begin{lemma}\label{ratelinftyQhatgradient} 
		Suppose Assumption \ref{assfan2017} and $\alpha_n \geq \sqrt{4/3} \, \Cgradf \,\big( \frac{\log(p)}{n}\big)^{\frac{1}{2}}$ hold. Then for $s \leq \log(p)$ and $n \geq \max\big\{ \Crsct s \log(p), 6 \log(p) \big\}$ we have that
		\begin{align}
			\normi{\widehat{Q}_{S^c S} \, \big(\widehat{Q}_{S S}\big)^{-1} \Big(\nabla \loss_{n,\alpha_n}^{\,\Hu} \big(\beta_{\alpha_n}^*\big)\Big)_{S} } \leq \frac{\sqrt{4} \,66\,\CXsubs}{\sqrt{3\,\CXl}} \,\Cgrads \, \bigg( \frac{\log(p)}{n}\bigg)^\frac{1}{2} \label{defsetlinftyQhatgradient}
		\end{align}
		with probability at least $1-\cPo \exp(-\cPt n) - 4/p^2 $.  \\
	\end{lemma}

	\begin{proof}
		Set
		\[ \mathcal{T} = \bigg\{ \normi{\widehat{Q}_{S^c S} \, \big(\widehat{Q}_{S S}\big)^{-1} \Big(\nabla \loss_{n,\alpha_n}^{\,\Hu} \big(\beta_{\alpha_n}^*\big)\Big)_{S} } \leq \frac{\sqrt{4} \,66\,\CXsubs}{\sqrt{3\,\CXl}} \,\Cgrads \,\bigg( \frac{\log(p)}{n}\bigg)^\frac{1}{2}
		\bigg\}.\]
		Then 
		\begin{align}
			\Prob \big(\mathcal{T}^c\big) 
			&\leq \Prob \Bigg(\mathcal{T}^c \cap \bigg\{ \max_{k \in \{1\dotsc,p-s\}} \normz{ \Big(e_k^\top \widehat{Q}_{S^c S} \big( \widehat{Q}_{S S} \big)^{-1} \Big)^\top} \leq \frac{33\,\CXsub}{\sqrt{\CXl}} \bigg\} \Bigg) \notag \\
			&\quad \quad \quad \quad \quad \quad \quad \quad + \Prob \Bigg( \max_{k \in \{1\dotsc,p-s\}} \normz{ \Big(e_k^\top \widehat{Q}_{S^c S} \big( \widehat{Q}_{S S} \big)^{-1} \Big)^\top} > \frac{33\,\CXsub}{\sqrt{\CXl}} \Bigg) \notag \\
			&\leq \Prob \Bigg(\mathcal{T}^c \cap \bigg\{ \max_{k \in \{1\dotsc,p-s\}} \normz{ \Big(e_k^\top \widehat{Q}_{S^c S} \big( \widehat{Q}_{S S} \big)^{-1} \Big)^\top} \leq \frac{33\,\CXsub}{\sqrt{\CXl}} \bigg\} \Bigg) \notag \\
			&\quad \quad \quad \quad \quad \quad \quad \quad + \cPo \exp(-\cPt n) + 2/p^2 \label{proofratelinftyQhatgradient1}
		\end{align}
		because of Lemma \ref{rateofrnS}. 
		Further, by definition of the event $\mathcal{T}$, 
		\begin{align}
			&\Prob \Bigg(\mathcal{T}^c \cap \bigg\{ \max_{k \in \{1\dotsc,p-s\}} \normz{ \Big(e_k^\top \widehat{Q}_{S^c S} \big( \widehat{Q}_{S S} \big)^{-1} \Big)^\top} \leq \frac{33\,\CXsub}{\sqrt{\CXl}}  \bigg\} \Bigg) \notag \\
			&\quad \quad \quad \quad = \Prob \Bigg( \bigg\{ \max_{k \in \{1\dotsc,p-s\}} \Big| e_k^\top \widehat{Q}_{S^c S} \, \big(\widehat{Q}_{S S}\big)^{-1} \Big(\nabla \loss_{n,\alpha_n}^{\,\Hu} \big(\beta_{\alpha_n}^*\big)\Big)_{S} \Big| > \frac{\sqrt{4} \,66\,\CXsubs}{\sqrt{3\,\CXl}} \,\Cgrads \, \bigg( \frac{\log(p)}{n}\bigg)^\frac{1}{2} \bigg\} \notag \\
			&\quad \quad \quad \quad \quad \quad \quad \quad \quad \quad \quad \quad \cap \bigg\{ \max_{k \in \{1\dotsc,p-s\}} \normz{ \Big(e_k^\top \widehat{Q}_{S^c S} \big( \widehat{Q}_{S S} \big)^{-1} \Big)^\top} \leq \frac{33\,\CXsub}{\sqrt{\CXl}}  \bigg\} \Bigg) \notag \\
			& \quad \quad \quad \quad \leq \Prob \Bigg( \max_{u \in \R^s : \normz{u} \leq \frac{33\,\CXsub}{\sqrt{\CXl}} } \Big|u^\top \Big(\nabla \loss_{n,\alpha_n}^{\,\Hu} \big(\beta_{\alpha_n}^*\big)\Big)_{S} \Big| > \frac{\sqrt{4} \,66\,\CXsubs}{\sqrt{3\,\CXl}} \,\Cgrads \, \bigg( \frac{\log(p)}{n}\bigg)^\frac{1}{2} \Bigg) \notag \\
			& \quad \quad \quad \quad = \Prob \Bigg( \max_{u \in \R^s : \normz{u} \leq 1} \bigg| u^\top \bigg( \frac{33\,\CXsub}{\sqrt{\CXl}}  \, \Big(\nabla \loss_{n,\alpha_n}^{\,\Hu} \big(\beta_{\alpha_n}^*\big)\Big)_{S} \bigg) \bigg|  > \frac{\sqrt{4} \,66\,\CXsubs}{\sqrt{3\,\CXl}} \,\Cgrads \, \bigg( \frac{\log(p)}{n}\bigg)^\frac{1}{2} \Bigg) \,. \label{proofratelinftyQhatgradient2}
		\end{align}  
		%
		In the following we proceed with a covering argument. Let $A$ denote a $1/2$-cover of cardinality $N= N(1/2; \mathcal{B}_2^s, \normz{\,\cdot\,})$ of the unit Euclidean ball $\mathcal{B}_2^s = \big\{ u \in \R^s : \normz{u} \leq 1 \big\}$ of $\R^s$ with respect to the Euclidean distance (cf. \citet[ Definition 1.17]{Rigollet2019} or \citet[Definition 5.1]{Wainwright2019}). Then, as in the proof of \citet[Theorem 1.19]{Rigollet2019}, we obtain 
		\begin{align}
			&\Prob \Bigg( \max_{u \in \mathcal{B}_2^s} \bigg| u^\top \bigg( \frac{33\,\CXsub}{\sqrt{\CXl}} \, \Big(\nabla \loss_{n,\alpha_n}^{\,\Hu} \big(\beta_{\alpha_n}^*\big)\Big)_{S} \bigg) \bigg|  > \frac{\sqrt{4} \,66\,\CXsubs}{\sqrt{3\,\CXl}} \,\Cgrads \, \bigg( \frac{\log(p)}{n}\bigg)^\frac{1}{2} \Bigg) \notag  \\
			&\quad \quad \quad \quad \leq \Prob \Bigg( \max_{u \in A} \bigg| u^\top \bigg( \frac{33\,\CXsub}{\sqrt{\CXl}} \, \Big(\nabla \loss_{n,\alpha_n}^{\,\Hu} \big(\beta_{\alpha_n}^*\big)\Big)_{S} \bigg) \bigg| > \frac{\sqrt{4} \,33\,\CXsubs}{\sqrt{3\,\CXl}} \,\Cgrads \, \bigg( \frac{\log(p)}{n}\bigg)^\frac{1}{2} \Bigg)\,. \label{proofratelinftyQhatgradient3}
		\end{align}
		Now we can write for fixed $u \in A$, analog to the proof of Lemma \ref{lemmamatrixgradient}, 
		\begin{align*}
			u^\top \bigg( \frac{33\,\CXsub}{\sqrt{\CXl}} \, \Big(\nabla \loss_{n,\alpha_n}^{\,\Hu} \big(\beta_{\alpha_n}^*\big)\Big)_{S} \bigg) = - \frac{1}{n} \sum_{i=1}^n l_{\alpha_n}'\big(Y_i-\f{X_i^\top} \beta_{\alpha_n}^* \big) Z_i 
		\end{align*}
		with $Z_i = 33\,\CXsub/\sqrt{\CXl} ~ u^\top  \big(\f{X_i}\big)_S$. The random variables $l_{\alpha_n}'\big(Y_1-\f{X_1^\top} \beta_{\alpha_n}^* \big)  Z_1,\dotsc, l_{\alpha_n}'\big(Y_n-\f{X_n^\top} \beta_{\alpha_n}^* \big)  Z_n $ are independent and identically distributed and have mean equal to zero, see proof of Lemma \ref{lemmamatrixgradient} for more details. In addition (iii) of Assumption \ref{assfan2017} implies that the random variables $Z_{1},\dotsc,Z_n$ are sub-Gaussian with variance proxy $ 1089\,\CXsubf/\CXl$. This leads to 
		\begin{align*}
			\E\bigg[ \Big( l_{\alpha_n}'\big(Y_i-\f{X_i^\top} \beta_{\alpha_n}^* \big) Z_{i} \Big)^2 \bigg] \leq \frac{1089\,\CXsubs\,\Cgradt}{\CXl} 
		\end{align*}
		and 
		\begin{align*}
			\E\bigg[ \Big| l_{\alpha_n}'\big(Y_i-\f{X_i^\top} \beta_{\alpha_n}^* \big) \, Z_{i} \Big|^u \bigg] \leq \frac{u!}{2} \, \bigg(\frac{2\,33\,\CXsub\,\Cgradv}{\sqrt{\CXl} \,\alpha_n} \bigg)^{u-2} \, \Cgradt 
		\end{align*}
		for $u \in \N$, $u \geq 3$, where $\Cgradt$ and $\Cgradv$ are given in the proof of Lemma \ref{ratelinftygradient}. Bernstein's inequality and the choice of $\alpha_n$ leads for fixed $u \in A$ to
		\begin{align*}
			\Prob\Bigg(\bigg| \frac{1}{n} \sum_{i=1}^{n} l_{\alpha_n}'\big(Y_i-\f{X_i^\top} \beta_{\alpha_n}^* \big) \, Z_{i} \bigg| \geq \frac{66\,\CXsub}{\sqrt{\CXl}} \,\bigg( \frac{8 \Cgradt \,\log(p)}{n}\bigg)^\frac{1}{2} \Bigg) \leq 2 \exp\big(-4 \log(p)\big) \,,
		\end{align*}
		see proof of Lemma \ref{ratelinftygradient} for more details. By the union bound and the definition of $\Cgrads$ in the proof of Lemma \ref{ratelinftygradient} we get
		\begin{align}
			\Prob \Bigg( \max_{u \in A} \bigg| u^\top \bigg( &\frac{33\,\CXsub}{\sqrt{\CXl}} \, \Big(\nabla \loss_{n,\alpha_n}^{\,\Hu} \big(\beta_{\alpha_n}^*\big)\Big)_{S} \bigg) \bigg| > \frac{\sqrt{4} \,33\,\CXsubs}{\sqrt{3\,\CXl}} \,\Cgrads \, \bigg( \frac{\log(p)}{n}\bigg)^\frac{1}{2} \Bigg) \notag \\
			&\leq 2\,N \, \exp\big(-4 \log(p)\big) \leq  2 \exp\big(- 4 \log(p) + s \log(6) \big) \notag \\
			&\leq 2 \exp\big(- 4 \log(p) + 2 \log(p)  \big) = \frac{2}{p^2} \label{proofratelinftyQhatgradient4}
		\end{align}
		since the $1/2$-covering-number $N$ can be upper bounded by $6^s$, cf. \citet[Lemma 1.18]{Rigollet2019} or  \citet[Example 5.8]{Wainwright2019}, and we assumed $s \leq \log(p)$. In conclusion the inequalities \eqref{proofratelinftyQhatgradient1} - \eqref{proofratelinftyQhatgradient4} imply the assertion of the lemma. 
	\end{proof}

	\subsection{Proof of (50) and of (51)}\label{sec:proofsignconsweightedhuber2}

		From (34) in Lemma 6  we obtain
		\begin{align*}
			\Big( \widehat{Q}_{S^c (S_{\alpha_n} \setminus S)} &- \widehat{Q}_{S^c S} \big(\widehat{Q}_{S S}\big)^{-1} \widehat{Q}_{S (S_{\alpha_n} \setminus S)}  \Big) \,\beta_{\alpha_n,S_{\alpha_n} \setminus S}^* \\
			&= \bigg( \frac{2}{n} \, \X_{n,S^c}^\top\, D \, \X_{n,S_{\alpha_n} \setminus S} - \frac{2}{n} \X_{n,S^c}^\top\, D \, \X_{n,S} \big( \X_{n,S}^\top\, D \, \X_{n,S} \big)^{-1} \X_{n,S}^\top\, D \, \X_{n,S_{\alpha_n} \setminus S} \bigg) \,\beta_{\alpha_n,S_{\alpha_n} \setminus S}^* \\
			&= \frac{2}{n} \, \X_{n,S^c}^\top\, D^\frac{1}{2} \Big( \mathrm{I}_n - D^\frac{1}{2} \, \X_{n,S} \big( \X_{n,S}^\top\, D \, \X_{n,S} \big)^{-1} \X_{n,S}^\top\, D^\frac{1}{2} \Big) D^\frac{1}{2} \, \X_{n,S_{\alpha_n} \setminus S} \,\beta_{\alpha_n,S_{\alpha_n} \setminus S}^* \,.
		\end{align*}
		The matrix in brackets, which we will denote by $\mathrm{P}$, is an orthogonal projection matrix. Therefore 
		using Lemma \ref{lemmacolumnnormalization},
		on an event with probability at least $1-2/p^2$ we obtain
		\begin{align*}
			\max_{k \in \{ 1,\dotsc,p-s \} } \, \normz{\bigg(e_k^\top \frac{2}{n} \,\X_{n,S^c}^\top\, D^\frac{1}{2} \, \mathrm{P}  \, D^\frac{1}{2} \bigg)^\top} \leq \max_{k \in S^c} \bigg( \frac{2}{n} \, \normzM{D^\frac{1}{2}} \,\big\lVert\mathrm{P}\big\rVert_{\mathrm{M},2}\, \normzM{D^\frac{1}{2}} \normz{\vec{X}_k} \bigg) \leq \frac{2\,\sqrt{17} \, \CXsub}{\sqrt{n}} 
		\end{align*}
		since the entries of the diagonal matrix $D$ are smaller than or equal to $1$. Setting $Q = \X_{n,S_{\alpha_n} \setminus S} \,\beta_{\alpha_n,S_{\alpha_n} \setminus S}^*$, for $x>0$ this leads to
		\begin{align*}
			\Prob \Bigg( &\max_{k \in \{ 1,\dotsc,p-s \} } \, \bigg|  e_k^\top \frac{2}{n} \, \X_{n,S^c}^\top\, D^\frac{1}{2} \, \mathrm{P}  \, D^\frac{1}{2} \, \X_{n,S_{\alpha_n} \setminus S} \,\beta_{\alpha_n,S_{\alpha_n} \setminus S}^*\bigg| > x \, , \,  \max_{k \in \{1,\dotsc,p\}} \frac{1}{\sqrt{n}} \,\normz{\vec{X}_k} \leq \sqrt{17} \, \CXsub  \Bigg) \\
			& \quad \leq \Prob \Bigg( \max_{k \in \{ 1,\dotsc,p-s \}} \, \bigg|  e_k^\top \frac{2}{n} \, \X_{n,S^c}^\top\, D^\frac{1}{2} \, \mathrm{P}  \, D^\frac{1}{2} \, Q \bigg| > x \, , \,  \\
			& \quad \quad \quad \quad \quad \quad \quad \quad \quad \max_{k \in \{ 1,\dotsc,p-s \}} \, \normz{\bigg(e_k^\top \frac{2}{n} \,\X_{n,S^c}^\top\, D^\frac{1}{2} \, \mathrm{P}  \, D^\frac{1}{2} \bigg)^\top} \leq \frac{2\,\sqrt{17} \, \CXsub}{\sqrt{n}}  \Bigg) \\
			& \quad \leq \Prob \Bigg( \max_{u \in \R^n : \normz{u} \leq \frac{2\,\sqrt{17} \, \CXsub}{\sqrt{n}}} \big|u^\top Q \big| > x \Bigg) = \Prob \Bigg( \max_{u \in \R^n : \normz{u} \leq 1} \bigg|u^\top \bigg( \frac{2\,\sqrt{17} \, \CXsub}{\sqrt{n}} \, Q \bigg) \bigg| > x \Bigg). 
		\end{align*}
		The vector $Q$ has independent and sub-Gaussian entries
		\begin{align*}
			\big(\f{X_i}\big)_{S_{\alpha_n} \setminus S}^\top \,\beta_{\alpha_n,S_{\alpha_n} \setminus S}^* = \f{X_i}^\top \begin{pmatrix}
				\beta_{\alpha_n,S_{\alpha_n} \setminus S}^* \\
				\f{0}_{|(S_{\alpha_n} \setminus S)^c|}
			\end{pmatrix}
			\sim \subg\Big(\CXsub \, \big\lVert \beta_{\alpha_n,S_{\alpha_n} \setminus S}^* \big\rVert_2 \Big)
		\end{align*}
		by (iii) of Assumption \ref{assfan2017}, and \citet[Theorem 1.6]{Rigollet2019} implies
		\begin{align*}
			\frac{2\,\sqrt{17} \, \CXsub}{\sqrt{n}} \, Q \sim \subg_n \Bigg(\frac{2\,\sqrt{17} \,\CXsubs \, \big\lVert \beta_{\alpha_n,S_{\alpha_n} \setminus S}^* \big\rVert_2}{\sqrt{n}} \Bigg) \,.
		\end{align*}
		Finally \citet[Theorem 1.19]{Rigollet2019} with the choice $\delta=\exp(-2n)$ leads to
		\begin{align*}
			\Prob \Bigg( \max_{u \in \R^n : \normz{u} \leq 1} \bigg|u^\top \bigg( \frac{2\,\sqrt{17} \, \CXsub}{\sqrt{n}} \, Q \bigg) \bigg|  >  16 \,\sqrt{17} \, \CXsubs \, \big\lVert \beta_{\alpha_n,S_{\alpha_n} \setminus S}^* \big\rVert_2 \Bigg) \leq \exp\big(-2n\big) \,,
		\end{align*}
		so that we obtain overall
		%
		%
		%
		\begin{align}
			\Prob \bigg(&\normi{ \Big( \widehat{Q}_{S^c (S_{\alpha_n} \setminus S)} - \widehat{Q}_{S^c S} \big(\widehat{Q}_{S S}\big)^{-1} \widehat{Q}_{S (S_{\alpha_n} \setminus S)} \Big) \, \beta_{\alpha_n,S_{\alpha_n} \setminus S}^* } > 16 \,\sqrt{17} \, \CXsubs \, \big\lVert \beta_{\alpha_n,S_{\alpha_n} \setminus S}^* \big\rVert_2
			\bigg) \notag \\
			& \qquad \qquad \leq  \exp\big(-2n\big) + 2/p^2\,. \label{eq:hilfestagain}
		\end{align}
		Similarly, for the vector
		$\widehat{Q}_{S (S_{\alpha_n} \setminus S)} \,\beta_{\alpha_n,S_{\alpha_n} \setminus S}^* = \frac{2}{n} \,\X_{n,S}^\top\, D \, \X_{n,S_{\alpha_n} \setminus S} \,\beta_{\alpha_n,S_{\alpha_n} \setminus S}^*$, 
		arguing as for \eqref{eq:hilfestagain} we obtain 
		\begin{align*}
			\Prob \Bigg(&\normi{ \widehat{Q}_{S (S_{\alpha_n} \setminus S)} \,\beta_{\alpha_n,S_{\alpha_n} \setminus S}^* } > 16 \,\sqrt{17} \, \CXsubs \, \big\lVert \beta_{\alpha_n,S_{\alpha_n} \setminus S}^* \big\rVert_2 \Bigg) \leq \exp\big(-2n\big) + 2/p^2 \,.
		\end{align*}
		%
		%
		From Lemma \ref{approximation:error:pseudo:huber} we get
		\begin{align*}
			\normz{ \beta_{\alpha_n,S_{\alpha_n} \setminus S}^* } &= \normz{ \beta_{\alpha_n,S_{\alpha_n} \setminus S}^* - \beta_{S_{\alpha_n} \setminus S}^*} \leq \normz{\beta_{\alpha_n}^* - \beta^*} \leq \Capprox\,\alpha_n^{\Cpm-1}
		\end{align*}
		since $\beta_l^* = 0$ for all $l \in S_{\alpha_n} \setminus S$. So in total we have
		\begin{align*}
			&\normi{ \widehat{Q}_{S (S_{\alpha_n} \setminus S)} \,\beta_{\alpha_n,S_{\alpha_n} \setminus S}^* } \, , \, \normi{ \Big( \widehat{Q}_{S^c (S_{\alpha_n} \setminus S)} - \widehat{Q}_{S^c S} \big(\widehat{Q}_{S S}\big)^{-1} \widehat{Q}_{S (S_{\alpha_n} \setminus S)}  \Big) \,\beta_{\alpha_n,S_{\alpha_n} \setminus S}^* } \notag \\
			&\quad \quad \quad \quad \quad \quad \quad \quad \quad \quad \quad \quad \leq 80 \, \Capprox \, \CXsubs \,\alpha_n^{\Cpm-1} 
		\end{align*}
		with probability at least $1- 2 \exp(-2n) - 2/p^2 $, which yields the claimed inequalities in (50) and in  (51). 
		
		\section{Supplement: Additional simulation results}\label{sessupp:sims}
		
		Additional simulation in the setting of Section 4.1.  
		
		\begin{itemize}
			
			\item[(c)] \textbf{Symmetric errors with heavy tails.} \\
			Here we consider $\widetilde{\varepsilon}_i = 2 \, Q_i$ with $Q_i \sim t_3$ t-distributed with $3$ degrees of freedom.\\
			
			\begin{table}[!h]
				\centering
				\begin{tabular}{rrrrrrrr}
					\hline
					& L & AL & LH & LPH & ALH (LH) & ALPH (LH) & ALPH (LPH) \\ 
					\hline
					$\lambda$ & 0.262 & 0.901 & 0.142 & 0.080 & 0.059 & 0.040 & 0.033 \\ 
					$\alpha$ &  &  & 0.429 & 0.742 & 0.563 & 0.769 & 0.974\\ 
					$\ell_2$ norm & 2.85 & 1.89 & 2.34 & 2.35 & 1.17 & 1.18 & 1.19\\ 
					$\ell_\infty$ norm & 1.03 & 0.76 & 0.85 & 0.85 & 0.53 & 0.53 & 0.53 \\ 
					FPR in \% & 15.64 & 2.74 & 16.66 & 17.59 & 1.38 & 1.39 & 1.51\\ 
					FNR in \% & 0.03 & 0.05 & 0.00 & 0.00 & 0.00 & 0.00 & 0.00\\ 
					\hline
				\end{tabular}
				\caption{homoscedastic t-distributed errors. Simulations with $n=200$, $p=400$, $s=20$.}
			\end{table}

			\begin{table}[!h]
				\centering
				\begin{tabular}{rrrrrrrr}
					\hline
					& L & AL & LH & ALH (LH) & ALPH (LH) \\ 
					\hline
					$\lambda$ & 0.226 & 0.849 & 0.019 & 0.0005 & 0.0006\\ 
					$\alpha$ &  &  & 3.574 & 33.854 & 29.368 \\ 
					$\ell_2$ norm & 2.71 (1.42) & 1.87 (1.54) & 1.37 (0.33) & 0.28 (0.13) & 0.28 (0.12) \\ 
					$\ell_\infty$ norm & 0.94 (0.38) & 0.72 (0.40) & 0.46 (0.12) & 0.12 (0.05) & 0.11 (0.05) \\ 
					FPR in \% & 16.36 (5.08) & 3.05 (3.94) & 20.95 (1.65) & 1.13 (0.76) & 1.18 (0.74) \\ 
					FNR in \% & 0.11 (1.83) & 0.16 (2.26) & 0.00 (0.00) & 0.00 (0.00) & 0.00 (0.00) \\ 
					\hline
				\end{tabular}
				\caption{heteroscedastic t-distributed errors. Simulations with $n=200$, $p=400$, $s=20$.}
			\end{table}
		\end{itemize}
		
		\begin{table}[ht]
			\centering
			\begin{tabular}{rrrr}  \hline & L & LH & ALPH \\   \hline
				$\lambda$ & 0.2162 & 0.0551 & 0.0731 \\   
				$\alpha$ &  & 0.3242 & 0.5778 \\   
				$\ell_2$ norm & 2.8438 & 3.1088 & 1.6544 \\   
				$\ell_\infty$ norm & 1.1012 & 1.1842 & 0.7300 \\   
				FPR in \% & 19.0298 & 21.9450 & 2.5748 \\   
				FNR in \% & 0.0000 & 0.0000 & 0.0000 \\   
				FPR in \% (Knockoff Augmented) & 0.5648 & 0.5256 & 0.2076 \\   
				FNR in \% (Knockoff Augmented) & 0.0000 & 0.0000 & 0.0000 \\   
				FPR in \% (Knockoff Counting) & 1.7236 & 1.6661 & 0.1323 \\   
				FNR (in \% Knockoff Counting) & 0.0000 & 0.0000 & 22.1532 \\   
				\hline\end{tabular}
			\caption{homoscedastic t-distributed errors with $n=100$, $p=200$ and $s=20$.}
		\end{table}
		
		\begin{table}[ht]
			\centering
			\begin{tabular}{rrrr}  \hline & L & LH & ALPH \\   \hline
				$\lambda$ & 0.1982 & 0.0711 & 0.0046 \\   
				$\alpha$ &  & 1.8431 & 6.1579 \\   
				$\ell_2$ norm & 2.7213 & 2.1308 & 0.5801 \\   
				$\ell_\infty$ norm & 1.0411 & 0.7810 & 0.2459 \\   
				FPR in \% & 19.3784 & 25.7561 & 3.0001 \\   
				FNR in \% & 0.0000 & 0.0000 & 0.0000 \\   
				FPR (in \% Knockoff Augmented) & 0.5883 & 0.5360 & 0.1327 \\   
				FNR in \% (Knockoff Augmented) & 0.0000 & 0.0000 & 0.0000 \\   
				FPR in \% (Knockoff Counting) & 1.7069 & 1.6506 & 0.2356 \\   
				FNR in \% (Knockoff Counting) & 0.0000 & 0.0000 & 0.0000 \\    
				\hline\end{tabular}
			\caption{heteroscedastic t-distributed errors with $n=100$, $p=200$ and $s=20$.}
		\end{table}

		\begin{table}[ht]
			\centering
			\begin{tabular}{rrrr}  \hline & L & LH & ALPH \\   \hline
				$\lambda$ & 0.0000 & 0.0010 & 0.0391 \\   
				$\alpha$ &  & 0.0100 & 0.2611 \\   
				$\ell_2$ norm & 33.9970 & 26.1026 & 23.8758 \\   
				$\ell_\infty$ norm & 10.4033 & 7.4723 & 6.7253 \\   
				FPR in \% & 100.0000 & 39.8301 & 28.4319 \\   
				FNR in \% & 0.0000 & 2.2497 & 3.0190 \\   
				FPR in \% (Knockoff Augmented) & 77.0232 & 15.5763 & 6.4258 \\   
				FNR in \% (Knockoff Augmented) & 1.9020 & 23.0822 & 39.1649 \\   
				FPR in \% (Knockoff Counting) & 27.9123 & 12.9077 & 5.0184 \\   
				FNR in \% (Knockoff Counting) & 3.0664 & 8.8830 & 22.5527 \\    
				\hline\end{tabular}
			\caption{homoscedastic t-distributed errors with $n=100$, $p=200$ and $s=40$}
	\end{table}

	\begin{landscape}

		\begin{table}[ht]
			\centering
			\tiny
			\begin{tabular}{rrrrrrrrrrrrrrrr} 
				\hline 
				& L &  L (Aug) & L (Count) & AL & AL (Aug) & AL (Count)  & LH & LH (Aug) & LH (Count) & ALH & ALH (Aug) & ALH (Count) & ALPH & ALPH (Aug) & ALPH (Count) \\   \hline
				$\lambda$ & 0.1321 & 0.1321 & 0.1321 & 0.8108 & 0.8108 & 0.8108 & 0.0420 & 0.0420 & 0.0420 & 0.0260 & 0.0260 & 0.0260 & 0.1021 & 0.1021 & 0.1021 \\  
				$\alpha$ & 0.0000 & 0.0000 & 0.0000 & 0.0000 & 0.0000 & 0.0000 & 0.3242 & 0.3242 & 0.3242 & 0.2553 & 0.2553 & 0.2553 & 0.1000 & 0.1000 & 0.1000 \\   
				$\tau$ & 0.0000 & 0.6834 & 0.3945 & 0.0000 & 0.3113 & 0.2478 & 0.0000 & 0.7620 & 0.4087 & 0.0000 & 0.4143 & 0.1489 & 0.0000 & 0.4099 &  \\   
				$\ell_2$ norm & 2.7482 & 0.0000 & 0.0000 & 1.6623 & 0.0000 & 0.0000 & 2.7898 & 0.0000 & 0.0000 & 1.5054 & 0.0000 & 0.0000 & 1.5024 & 0.0000 & 0.0000 \\   
				$\ell_\infty$ norm & 1.0193 & 0.0000 & 0.0000 & 0.6927 & 0.0000 & 0.0000 & 1.0448 & 0.0000 & 0.0000 & 0.6586 & 0.0000 & 0.0000 & 0.6578 & 0.0000 & 0.0000 \\   
				FPR in \%  & 26.4528 & 0.6166 & 1.7414 & 5.3998 & 0.3097 & 1.0421 & 25.8134 & 0.5287 & 1.6843 & 3.6955 & 0.2698 & 0.8476 & 3.6208 & 0.2732 & 0.8276 \\   
				FNR in \% & 0.0000 & 0.0000 & 0.0000 & 0.0000 & 0.0000 & 0.0000 & 0.0000 & 0.0000 & 0.0000 & 0.0000 & 0.0000 & 0.0000 & 0.0000 & 0.0000 & 0.3080 \\ 
				\hline
			\end{tabular}
			\caption{homoscedastic normal errors with $n=100$, $p=200$ and $s=20$.}
		\end{table}

		\begin{table}[ht]
			\centering
			\tiny
			\begin{tabular}{rrrrrrrrrrrrrrrr} 
				\hline 
				& L &  L (Aug) & L (Count) & AL & AL (Aug) & AL (Count)  & LH & LH (Aug) & LH (Count) & ALH & ALH (Aug) & ALH (Count) & ALPH & ALPH (Aug) & ALPH (Count) \\   \hline
				$\lambda$ & 0.1301 & 0.1301 & 0.1301 & 0.8068 & 0.8068 & 0.8068 & 0.0702 & 0.0702 & 0.0702 & 0.0341 & 0.0341 & 0.0341 & 0.0032 & 0.0032 & 0.0032 \\   
				$\alpha$ & 0.0000 & 0.0000 & 0.0000 & 0.0000 & 0.0000 & 0.0000 & 2.0652 & 2.0652 & 2.0652 & 8.2794 & 8.2794 & 8.2794 & 8.2632 & 8.2632 & 8.2632 \\   
				$\tau$ & 0.0000 & 0.6422 & 0.3744 & 0.0000 & 0.2877 & 0.2333 & 0.0000 & 0.6249 & 0.3084 & 0.0000 & 0.1512 & 0.0400 & 0.0000 & 0.1334 & 0.0390 \\   
				$\ell_2$ norm  & 2.6448 & 0.0000 & 0.0000 & 1.6522 & 0.0000 & 0.0000 & 2.1385 & 0.0000 & 0.0000 & 0.6589 & 0.0000 & 0.0000 & 0.6410 & 0.0000 & 0.0000 \\   
				$\ell_\infty$ norm & 0.9783 & 0.0000 & 0.0000 & 0.6827 & 0.0000 & 0.0000 & 0.7822 & 0.0000 & 0.0000 & 0.2815 & 0.0000 & 0.0000 & 0.2693 & 0.0000 & 0.0000 \\   
				FPR in \% & 25.5842 & 0.5894 & 1.7303 & 4.9215 & 0.2643 & 0.9924 & 26.7555 & 0.5510 & 1.6512 & 2.8122 & 0.1989 & 0.6381 & 3.8832 & 0.1680 & 0.5126 \\   
				FNR in \% & 0.0000 & 0.0000 & 0.0000 & 0.0000 & 0.0000 & 0.0000 & 0.0000 & 0.0052 & 0.0000 & 0.0052 & 0.0052 & 0.0000 & 0.0000 & 0.0052 & 0.0000 \\   
			\end{tabular}
			\caption{heteroscedastic normal errors with $n=100$, $p=200$ and $s=20$.}
		\end{table}
		
		\begin{table}[ht]
			\centering
			\tiny
			\begin{tabular}{rrrrrrrrrrrrrrrr} 
				\hline 
				& L &  L (Aug) & L (Count) & AL & AL (Aug) & AL (Count)  & LH & LH (Aug) & LH (Count) & ALH & ALH (Aug) & ALH (Count) & ALPH & ALPH (Aug) & ALPH (Count) \\   \hline
				$\lambda$ & 0.2162 & 0.2162 & 0.2162 & 0.9029 & 0.9029 & 0.9029 & 0.0551 & 0.0551 & 0.0551 & 0.0511 & 0.0511 & 0.0511 & 0.0731 & 0.0731 & 0.0731 \\   
				$\alpha$ & 0.0000 & 0.0000 & 0.0000 & 0.0000 & 0.0000 & 0.0000 & 0.3242 & 0.3242 & 0.3242 & 0.5632 & 0.5632 & 0.5632 & 0.5778 & 0.5778 & 0.5778 \\   l
				$\tau$ & 0.0000 & 0.6606 & 0.3510 & 0.0000 & 0.2629 & 0.1395 & 0.0000 & 0.7491 & 0.3846 & 0.0000 & 0.3822 & 0.1101 & 0.0000 & 0.2184 &  \\   
				$\ell_2$ norm & 2.8438 & 0.0000 & 0.0000 & 1.7110 & 0.0000 & 0.0000 & 3.1088 & 0.0000 & 0.0000 & 1.5867 & 0.0000 & 0.0000 & 1.6544 & 0.0000 & 0.0000 \\   
				$\ell_\infty$ norm & 1.1012 & 0.0000 & 0.0000 & 0.7370 & 0.0000 & 0.0000 & 1.1842 & 0.0000 & 0.0000 & 0.7160 & 0.0000 & 0.0000 & 0.7300 & 0.0000 & 0.0000 \\   
				FPR in \% & 19.0298 & 0.5648 & 1.7236 & 2.6708 & 0.2812 & 0.7695 & 21.9450 & 0.5256 & 1.6661 & 2.2280 & 0.2841 & 0.6510 & 2.5748 & 0.2076 & 0.1323 \\   
				FNR in \% & 0.0000 & 0.0000 & 0.0000 & 0.0000 & 0.0000 & 0.0000 & 0.0000 & 0.0000 & 0.0000 & 0.0000 & 0.0000 & 0.0000 & 0.0000 & 0.0000 & 22.1532 \\ 
				\hline
			\end{tabular}
			\caption{homoscedastic t-distributed errors with $n=100$, $p=200$ and $s=20$.}
		\end{table}
		
		\begin{table}[ht]
			\centering
			\tiny
			\begin{tabular}{rrrrrrrrrrrrrrrr} 
				\hline 
				& L &  L (Aug) & L (Count) & AL & AL (Aug) & AL (Count)  & LH & LH (Aug) & LH (Count) & ALH & ALH (Aug) & ALH (Count) & ALPH & ALPH (Aug) & ALPH (Count) \\   \hline
				$\lambda$ & 0.1982 & 0.1982 & 0.1982 & 1.0050 & 1.0050 & 1.0050 & 0.0711 & 0.0711 & 0.0711 & 0.0411 & 0.0411 & 0.0411 & 0.0046 & 0.0046 & 0.0046 \\   
				$\alpha$ & 0.0000 & 0.0000 & 0.0000 & 0.0000 & 0.0000 & 0.0000 & 1.8431 & 1.8431 & 1.8431 & 6.6091 & 6.6091 & 6.6091 & 6.1579 & 6.1579 & 6.1579 \\   
				$\tau$ & 0.0000 & 0.6171 & 0.3407 & 0.0000 & 0.2239 & 0.1484 & 0.0000 & 0.6155 & 0.3045 & 0.0000 & 0.1318 & 0.0319 & 0.0000 & 0.0735 & 0.0171 \\   
				$\ell_2$ norm & 2.7213 & 0.0000 & 0.0000 & 1.7369 & 0.0000 & 0.0000 & 2.1308 & 0.0000 & 0.0000 & 0.6734 & 0.0000 & 0.0000 & 0.5801 & 0.0000 & 0.0000 \\   
				$\ell_\infty$ norm & 1.0411 & 0.0000 & 0.0000 & 0.7412 & 0.0000 & 0.0000 & 0.7810 & 0.0000 & 0.0000 & 0.2812 & 0.0000 & 0.0000 & 0.2459 & 0.0000 & 0.0000 \\   
				FPR in \% & 19.3784 & 0.5883 & 1.7069 & 2.4135 & 0.2436 & 0.7555 & 25.7561 & 0.5360 & 1.6506 & 1.9505 & 0.2126 & 0.5441 & 3.0001 & 0.1327 & 0.2356 \\   
				FNR in \% & 0.0000 & 0.0000 & 0.0000 & 0.0000 & 0.0000 & 0.0000 & 0.0000 & 0.0000 & 0.0000 & 0.0000 & 0.0000 & 0.0000 & 0.0000 & 0.0000 & 0.0000 \\    \hline
			\end{tabular}
			\caption{heteroscedastic t-distributed errors with $n=100$, $p=200$ and $s=20$.}
		\end{table}
		
		\begin{table}[ht]
			\centering
			\tiny
			\begin{tabular}{rrrrrrrrrrrrrrrr}  \hline 
				& L &  L (Aug) & L (Count) & AL & AL (Aug) & AL (Count)  & LH & LH (Aug) & LH (Count) & ALH & ALH (Aug) & ALH (Count) & ALPH & ALPH (Aug) & ALPH (Count)  \\  \hline
				$\lambda$ & 0.1061 & 0.1061 & 0.1061 & 0.6547 & 0.6547 & 0.6547 & 0.0591 & 0.0591 & 0.0591 & 0.0330 & 0.0330 & 0.0330 & 0.0397 & 0.0397 & 0.0397 \\   
				$\alpha$ & 0.0000 & 0.0000 & 0.0000 & 0.0000 & 0.0000 & 0.0000 & 0.7410 & 0.7410 & 0.7410 & 0.8459 & 0.8459 & 0.8459 & 0.6000 & 0.6000 & 0.6000 \\   
				$\tau$ & 0.0000 & 0.5567 & 0.3193 & 0.0000 & 0.2075 & 0.1757 & 0.0000 & 0.5742 & 0.2940 & 0.0000 & 0.2073 & 0.0613 & 0.0000 & 0.1288 &  \\   
				$\ell_2$ norm & 2.1458 & 0.0000 & 0.0000 & 1.2233 & 0.0000 & 0.0000 & 2.1141 & 0.0000 & 0.0000 & 0.8219 & 0.0000 & 0.0000 & 0.8876 & 0.0000 & 0.0000 \\   
				$\ell_\infty$ norm & 0.7940 & 0.0000 & 0.0000 & 0.5082 & 0.0000 & 0.0000 & 0.7933 & 0.0000 & 0.0000 & 0.3760 & 0.0000 & 0.0000 & 0.4000 & 0.0000 & 0.0000 \\   
				FPR in \% & 25.7723 & 0.6072 & 1.6208 & 4.7156 & 0.2586 & 0.8697 & 25.1944 & 0.5664 & 1.6397 & 2.3048 & 0.2580 & 0.6461 & 2.5496 & 0.2133 & 0.1525 \\   
				FNR in \% & 0.0000 & 0.0000 & 0.0000 & 0.0000 & 0.0000 & 0.0000 & 0.0000 & 0.0000 & 0.0000 & 0.0000 & 0.0000 & 0.0000 & 0.0000 & 0.0000 & 20.2270 \\   
				\hline
			\end{tabular}
			\caption{homoscedastic skew t-distributed errors with $n=100$, $p=200$ and $s=20$.}
		\end{table}

		\begin{table}[ht]
			\centering
			\tiny
			\begin{tabular}{rrrrrrrrrrrrrrrr} 
				\hline 
				& L &  L (Aug) & L (Count) & AL & AL (Aug) & AL (Count)  & LH & LH (Aug) & LH (Count) & ALH & ALH (Aug) & ALH (Count) & ALPH & ALPH (Aug) & ALPH (Count) \\   \hline
				$\lambda$ & 0.0961 & 0.0961 & 0.0961 & 0.8388 & 0.8388 & 0.8388 & 0.0671 & 0.0671 & 0.0671 & 0.0270 & 0.0270 & 0.0270 & 0.0019 & 0.0019 & 0.0019 \\   
				$\alpha$ & 0.0000 & 0.0000 & 0.0000 & 0.0000 & 0.0000 & 0.0000 & 2.3846 & 2.3846 & 2.3846 & 15.5932 & 15.5932 & 15.5932 & 10.3158 & 10.3158 & 10.3158 \\   l
				$\tau$ & 0.0000 & 0.4998 & 0.2892 & 0.0000 & 0.1602 & 0.1701 & 0.0000 & 0.4285 & 0.1950 & 0.0000 & 0.0582 & 0.0132 & 0.0000 & 0.0520 & 0.0120 \\   
				$\ell_2$ norm & 2.0060 & 0.0000 & 0.0000 & 1.2224 & 0.0000 & 0.0000 & 1.3282 & 0.0000 & 0.0000 & 0.2860 & 0.0000 & 0.0000 & 0.2908 & 0.0000 & 0.0000 \\   
				$\ell_\infty$ norm & 0.7368 & 0.0000 & 0.0000 & 0.5053 & 0.0000 & 0.0000 & 0.4869 & 0.0000 & 0.0000 & 0.1257 & 0.0000 & 0.0000 & 0.1264 & 0.0000 & 0.0000 \\   
				FPR in \% & 25.0355 & 0.6180 & 1.6615 & 3.6323 & 0.2045 & 0.7864 & 26.2532 & 0.5699 & 1.6523 & 1.6443 & 0.1970 & 0.4708 & 2.5040 & 0.1581 & 0.3442 \\   
				FNR in \% & 0.0000 & 0.0000 & 0.0000 & 0.0000 & 0.0000 & 0.0000 & 0.0000 & 0.0000 & 0.0000 & 0.0000 & 0.0000 & 0.0000 & 0.0000 & 0.0000 & 0.0000 \\   
				\hline
			\end{tabular}
			\caption{heteroscedastic skew t-distributed errors with $n=100$, $p=200$ and $s=20$.}
		\end{table}

		\begin{table}[ht]
			\centering
			\tiny
			\begin{tabular}{rrrrrrrrrrrrrrrr} 
				\hline 
				& L &  L (Aug) & L (Count) & AL & AL (Aug) & AL (Count)  & LH & LH (Aug) & LH (Count) & ALH & ALH (Aug) & ALH (Count) & ALPH & ALPH (Aug) & ALPH (Count) \\   \hline
				$\lambda$ & 0.0000 & 0.0000 & 0.0000 & 0.0881 & 0.0881 & 0.0881 & 0.0010 & 0.0010 & 0.0010 & 0.0280 & 0.0280 & 0.0280 & 0.0280 & 0.0280 & 0.0280 \\   
				$\alpha$ & 0.0000 & 0.0000 & 0.0000 & 0.0000 & 0.0000 & 0.0000 & 0.0100 & 0.0100 & 0.0100 & 0.4605 & 0.4605 & 0.4605 & 0.2000 & 0.2000 & 0.2000 \\   
				$\tau$ & 0.0000 & 0.0010 & 0.9329 & 0.0000 & 0.0014 & 2.2620 & 0.0000 & 0.7492 & 0.5411 & 0.0000 & 3.0706 & 1.2888 & 0.0000 & 3.1471 &  \\   
				$\ell_2$ norm & 33.9446 & 0.0000 & 0.0000 & 19.8070 & 0.0000 & 0.0000 & 24.4244 & 0.0000 & 0.0000 & 19.7822 & 0.0000 & 0.0000 & 20.1693 & 0.0000 & 0.0000 \\   
				$\ell_\infty$ norm & 10.4495 & 0.0000 & 0.0000 & 6.3186 & 0.0000 & 0.0000 & 7.0241 & 0.0000 & 0.0000 & 5.5899 & 0.0000 & 0.0000 & 5.6800 & 0.0000 & 0.0000 \\   
				FPR in \% & 100.0000 & 76.9096 & 27.7842 & 49.1218 & 42.0734 & 13.8221 & 38.9770 & 15.7287 & 12.6907 & 24.4292 & 6.4585 & 6.4925 & 27.4148 & 6.4859 & 5.8186 \\   
				FNR in \% & 0.0000 & 1.8732 & 3.1316 & 0.5315 & 2.6621 & 4.3375 & 2.1830 & 23.4247 & 8.8513 & 2.6930 & 39.6211 & 11.9733 & 2.7336 & 40.0334 & 15.3074 \\    \hline
			\end{tabular}
			\caption{homoscedastic normal errors with $n=100$, $p=200$ and $s=40$}
	\end{table}

	\begin{table}[ht]
		\centering
		\tiny
		\begin{tabular}{rrrrrrrrrrrrrrrr} 
			\hline 
			& L &  L (Aug) & L (Count) & AL & AL (Aug) & AL (Count)  & LH & LH (Aug) & LH (Count) & ALH & ALH (Aug) & ALH (Count) & ALPH & ALPH (Aug) & ALPH (Count) \\   \hline
			$\lambda$ & 0.0000 & 0.0000 & 0.0000 & 0.1401 & 0.1401 & 0.1401 & 0.0010 & 0.0010 & 0.0010 & 0.0310 & 0.0310 & 0.0310 & 0.0391 & 0.0391 & 0.0391 \\   
			$\alpha$ & 0.0000 & 0.0000 & 0.0000 & 0.0000 & 0.0000 & 0.0000 & 0.0100 & 0.0100 & 0.0100 & 0.6658 & 0.6658 & 0.6658 & 0.2611 & 0.2611 & 0.2611 \\   
			$\tau$& 0.0000 & 0.0010 & 0.9276 & 0.0000 & 0.0037 & 2.2606 & 0.0000 & 0.7592 & 0.5471 & 0.0000 & 3.0908 & 1.3453 & 0.0000 & 3.0886 &  \\   
			$\ell_2$ norm & 33.9970 & 0.0000 & 0.0000 & 18.9579 & 0.0000 & 0.0000 & 26.1026 & 0.0000 & 0.0000 & 23.2632 & 0.0000 & 0.0000 & 23.8758 & 0.0000 & 0.0000 \\   
			$\ell_\infty$ norm & 10.4033 & 0.0000 & 0.0000 & 6.1120 & 0.0000 & 0.0000 & 7.4723 & 0.0000 & 0.0000 & 6.5955 & 0.0000 & 0.0000 & 6.7253 & 0.0000 & 0.0000 \\   
			FPR in \% & 100.0000 & 77.0232 & 27.9123 & 43.9943 & 36.0636 & 13.7381 & 39.8301 & 15.5763 & 12.9077 & 26.7334 & 6.3007 & 6.4555 & 28.4319 & 6.4258 & 5.0184 \\   
			FNR in \% & 0.0000 & 1.9020 & 3.0664 & 0.6059 & 2.8609 & 4.3256 & 2.2497 & 23.0822 & 8.8830 & 2.9531 & 39.2677 & 12.3630 & 3.0190 & 39.1649 & 22.5527 \\    \hline
		\end{tabular}
		\caption{homoscedastic t-distributed errors with $n=100$, $p=200$ and $s=40$}
\end{table}

\begin{table}[ht]
	\centering
	\tiny
	\begin{tabular}{rrrrrrrrrrrrrrrr} 
		\hline 
		& L &  L (Aug) & L (Count) & AL & AL (Aug) & AL (Count)  & LH & LH (Aug) & LH (Count) & ALH & ALH (Aug) & ALH (Count) & ALPH & ALPH (Aug) & ALPH (Count) \\   \hline
		$\lambda$ & 0.0000 & 0.0000 & 0.0000 & 0.0380 & 0.0380 & 0.0380 & 0.0010 & 0.0010 & 0.0010 & 0.0360 & 0.0360 & 0.0360 & 0.0021 & 0.0021 & 0.0021 \\   
		$\alpha$ & 0.0000 & 0.0000 & 0.0000 & 0.0000 & 0.0000 & 0.0000 & 0.0100 & 0.0100 & 0.0100 & 2.1447 & 2.1447 & 2.1447 & 7.2105 & 7.2105 & 7.2105 \\   
		$\tau$ & 0.0000 & 0.0011 & 0.9336 & 0.0000 & 0.0012 & 2.2114 & 0.0000 & 0.6781 & 0.5191 & 0.0000 & 3.0363 & 1.1563 & 0.0000 & 3.0816 &  \\   
		$\ell_2$ norm & 33.8354 & 0.0000 & 0.0000 & 24.0536 & 0.0000 & 0.0000 & 24.1196 & 0.0000 & 0.0000 & 16.6086 & 0.0000 & 0.0000 & 19.5423 & 0.0000 & 0.0000 \\   
		$\ell_\infty$ norm & 10.3949 & 0.0000 & 0.0000 & 7.5885 & 0.0000 & 0.0000 & 6.9416 & 0.0000 & 0.0000 & 4.7185 & 0.0000 & 0.0000 & 5.5034 & 0.0000 & 0.0000 \\   
		FPR in \% & 100.0000 & 76.9475 & 27.6657 & 61.3365 & 56.1213 & 14.1694 & 38.5621 & 16.0540 & 12.4852 & 22.4497 & 6.4327 & 6.1294 & 28.3964 & 6.2848 & 4.5170 \\   
		FNR in \% & 0.0000 & 1.7456 & 2.9645 & 0.4201 & 2.2899 & 4.0385 & 2.1893 & 21.4882 & 7.7337 & 2.2189 & 38.3669 & 10.6450 & 2.6331 & 36.7604 & 16.4142 \\    \hline
	\end{tabular}
	\caption{homoscedastic skew t-distributed errors with $n=100$, $p=200$ and $s=40$}
\end{table}

\end{landscape}

\newpage

		\end{document}